\documentclass[a4paper,fleqn]{cas-sc}
\usepackage{nomencl}
\makeglossary
\usepackage[numbers]{natbib}

\def\tsc#1{\csdef{#1}{\textsc{\lowercase{#1}}\xspace}}
\tsc{WGM}
\tsc{QE}
\tsc{EP}
\tsc{PMS}
\tsc{BEC}
\tsc{DE}

\newtheorem{theorem}{Theorem}[section]
\newtheorem{lemma}{Lemma}[section]

\newtheorem{definition}{Definition}[section]
\newtheorem{remark}{Remark}[section]

\usepackage{graphics,graphicx,color}
\usepackage{rotating}
\usepackage{amssymb,amsmath}
\usepackage{eucal}
\usepackage{url}
\usepackage{algorithm}
\usepackage[noend]{algpseudocode}
\usepackage{epsfig,epstopdf}

\newenvironment{proof}[1][Proof]{\noindent\textbf{#1.} }{\ \rule{0.5em}{0.5em}}

\begin{document}
		
   \def\l({\left(}
	\def\r){\right)}

\let\WriteBookmarks\relax
\def\floatpagepagefraction{1}
\def\textpagefraction{.001}
\shorttitle{
	{Effect of blood flow on radiofrequency ablation model in cardiac tissue: mathematical analysis and numerical simulation
	}
}
\shortauthors{Bendahmane, Ouakrim, Ouzrour and Zagour}

\title[mode = title]{Effect of energy dissipation on radiofrequency ablation model in cardiac tissue: modelling, analysis and numerical simulation
}

\author[1]{Mostafa Bendahmane}
\address[1]{Institut de Mathématiques de Bordeaux, Université de Bordeaux, 33076 Bordeaux Cedex, France.} 
\ead{mostafa.bendahmane@u-bordeaux.fr}

\author[2]{Youssef Ouakrim}
\address[2]{Laboratoire de Math\'ematiques, Mod\'elisation et Physique Appliqu\'ee, Ecole Normale Sup\'erieure de F\`es, Universit\'e Sidi Mohamed Ben Abdellah, Maroc.}
\ead{youssef.ouakrim@usmba.ac.ma}

\author[2]{Yassine Ouzrour}
\ead{yassine.ouzrour@usmba.ac.ma}

\author[3]{Mohamed Zagour}
\cormark[1]
\address[3]{Euromed University of Fes, UEMF, Morocco.}
\ead{m.zagour@insa.ueuromed.org}

\begin{abstract}	
	This paper deals with the mathematical analysis and numerical simulation of a new nonlinear ablation system modeling radiofrequency ablation phenomena in cardiac tissue, {which incorporates the effects of blood flow on the heat generated when ablation by radiofrequency. The model also considers the effects of viscous energy dissipation. It consists of a coupled thermistor problem and the incompressible Navier--Stokes equations that describe the evolution of temperature, velocity and potential in cardiac tissue.} In addition to Faedo--Galerkin method, we use Schauder's fixed-point theory to prove the existence of the weak solutions in two- and three-dimensional space. Moreover, we prove the uniqueness of the solution under some additional conditions on the data and the solution. Finally, we discuss some numerical results for the validation of the proposed model using the finite element method.
	
\end{abstract}
\begin{keywords}
Bio--heat equation\sep
Navier--Stokes equation\sep
Dissipation of energy \sep
Thermistor problem\sep
Radiofrequency ablation\sep
Cardiac tissue\sep
Catheter ablation\sep
Finite element method. 
\end{keywords}  
\maketitle

 \section{Introduction} 
{
\subsection{Radiofrequency ablation procedure}
}
Radiofrequency ablation (RFA) techniques have been increasingly used in various medical fields, including the ablation of tumors in different parts of the body.
One such area is cardiac tissue, where the goal is to eliminate the tissue responsible for cardiac arrhythmia, for example, ventricular arrhythmias, atrial fibrillation, and atrial tachycardia.
During this procedure, a catheter is inserted into the heart to map its electrical activity and identify any diseased areas. These areas are then removed using an ablation catheter, which is heated by inducing electrical energy in a specific border area for a specific period of time. We refer the reader to Figure \ref{Fig-1} for a visual representation of the process.
\begin{figure}[pos=!ht]
	\centering
	\includegraphics[ width=.86\linewidth]{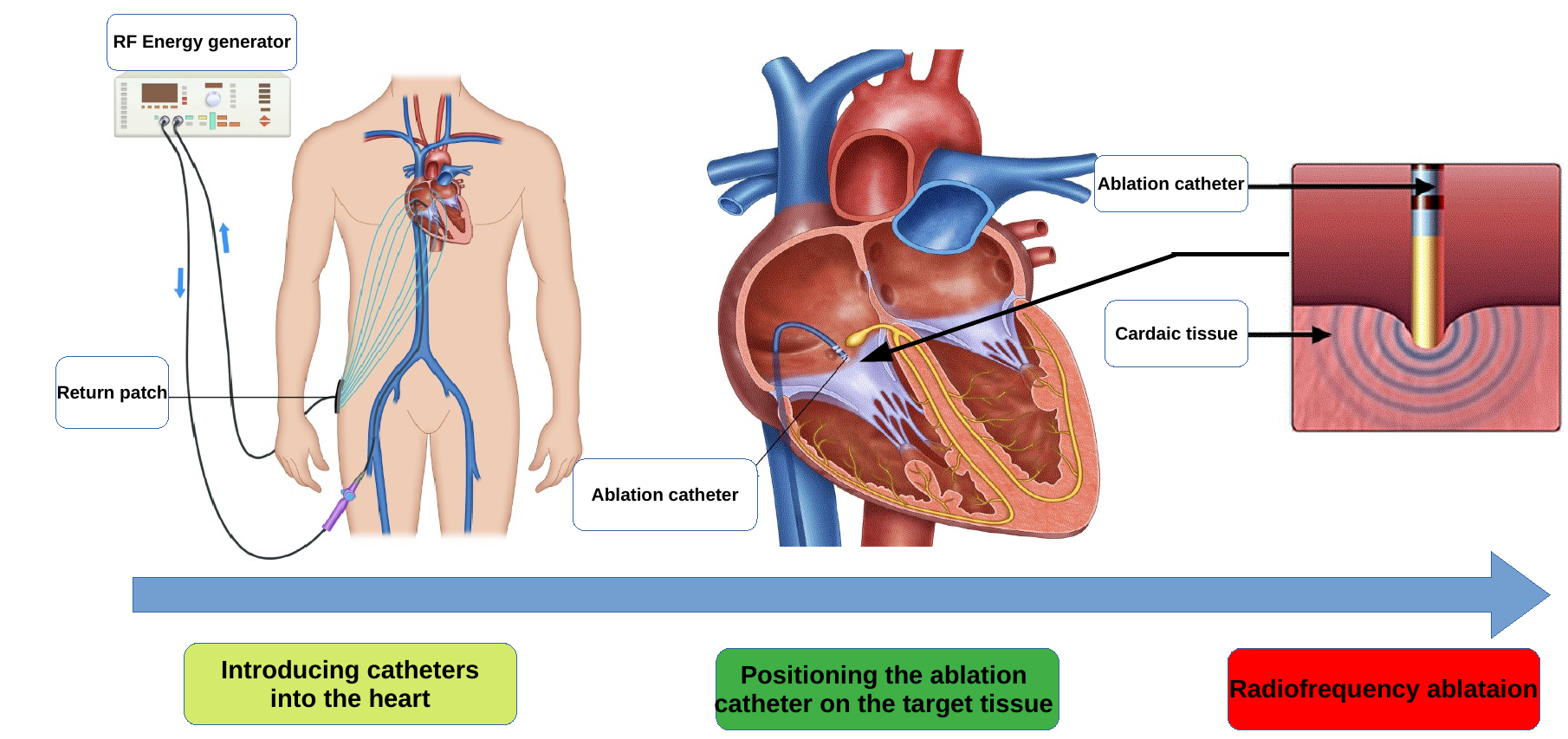}  
	\caption{
		Radiofrequency ablation procedure in cardiac tissue.
		\\
		\url{https://www.melbourneheart.com.au/procedures/electrophysiology/catheter-ablation/}
	}
	\label{Fig-1}
\end{figure} 
%

\noindent
 It is well-known that RFA models are typically described mathematically as a thermistor problem, which is presented as a coupled system of nonlinear partial differential equations (PDEs). Specifically, these equations consist of the heat equation with Joule heating as the source and the current conservation equation with temperature-dependent electrical conductivity \cite{xu1994thermistor}.
In this context, numerous works in the literature focus on accurately modelling the electrical and thermal properties of biological tissues, including those that vary over time as well as temperature, in order to quantify the relationships between characteristic values and the thermal damage function~\cite{ahmed2008}. For further details on modelling the study of radiofrequency ablation techniques, see ~\cite{Berjano2006}. 
The aforementioned reference presents important issues involved in this methodology, including experimental validation, current limitations, especially those related to the lack of precise characterization of biological tissues, and suggestions and future perspectives of this field.
 For example, the application of saline infusion requires the derivation of a suitable model to follow the behavior of the tissue during the simultaneous application of RF energy and the cooling effect. It is worth mentioning that the author in~\cite{Nikhil2021} develops realistic modeling for large and medium blood vessels. While model derivation and fluid mechanics studies of blood flow, for example, in the carotid arteries, basilar trunk, and circle of Willis, are the subject of numerous contributions, see~\cite{FQV09,Ber17,OOI17,QMV17} and references therein.


{
\subsection{Governing equations}

In this section, we describe the coupled system modeling the dynamic of RFA treatment in the presence of a fluid. Let us first describe the geometric configurations. Let  $\Omega \subset \mathbb{R}^{d}$,  $d=2,3$ is a bounded domain with a $C^{1,1}$ boundary $\partial \Omega=\Gamma $. We suppose that $\Gamma_1$, $\Gamma_2$, $\Gamma_3$, $\Gamma_{4}$ and $\Gamma_{5}$ are closed disjoint $(d-1)$-dimensional manifolds of class $C^{1,1}$ such that $\Gamma=\Gamma_1\cup\Gamma_2\cup\Gamma_3\cup\Gamma_{4} \cup \Gamma_{5}$ and are presented in Figure \ref{fig:omega}. Note that, the boundaries $\Gamma_1$ and $\Gamma_3$ are represent the inflow and the outflow, respectively. $\Gamma_2$ is the outer edge , $\Gamma_4$ the border between the heart tissue and the blood vessel,  and $\Gamma_5$ the part where the ablation catheter contacts the tissue. 
Let $T \in(0, \infty)$ be fixed throughout the paper,  $\Omega_{T}=\Omega \times (0, T)$,  $\Sigma=\Gamma \times (0, T)$, $\Sigma_N=\Gamma_{N} \times (0, T)$, $\Sigma_D=\Gamma_{D} \times (0, T)$ and $\Sigma_i=\Gamma_i \times (0, T)$ for $i=1,\cdots,5$.  

\begin{figure}[pos=!h]
	\centering
	\includegraphics[ width=0.75\linewidth]{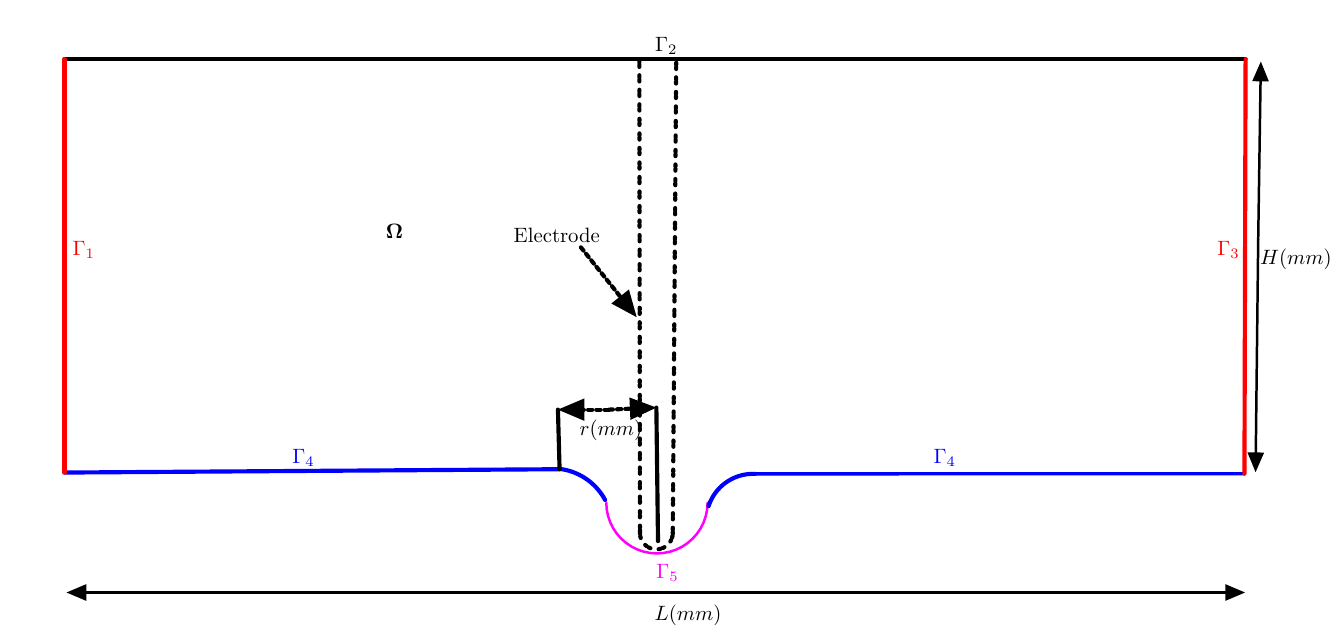}  
	\caption{Description of the computational domain $\Omega$.}
	\label{fig:omega}
\end{figure} 

\noindent The mathematical model of radiofrequency ablation in cardiac tissue consists of a system of partial differential equations that describe the evolution of the blood flow, the heat, and the electrical field generated by the RFA procedure. Then, we can divide the model into three subproblems. The first subproblem to be considered is the distribution of the blood in the blood vessel.  The second subproblem considers the modeling of the potential distribution inside the tissue. The third subproblem covers the description of the temperature distribution:

		\begin{enumerate}
			\item[$\bullet$] \textbf{Blood flow equation}: The blood flow can be characterized as an incompressible Navier–Stokes fluid in a quasi-steady regime, governed by the following system\footnote{$\operatorname{div}\boldsymbol{v}:=\sum_{i=1}^{d}\frac{\partial_i\boldsymbol{v}_i}{\partial \boldsymbol{x}_i}$, $\quad\nabla \boldsymbol{v}:=(\frac{\partial_i\boldsymbol{v}_i}{\partial \boldsymbol{x}_i},\,\frac{\partial_j\boldsymbol{v}_j}{\partial \boldsymbol{x}_j})$ if the dimension $d=2$,  $\quad\boldsymbol{a} \cdot \boldsymbol{b}$ we denote the usual scalar product in $\mathbb{R}^d$ and by $|\boldsymbol{a}|$ we denote the Euclidean norm.}:
			\begin{equation}
				\begin{array}{rclll}
					\rho\left(\boldsymbol{v}_{t}+\operatorname{div}(\boldsymbol{v} \otimes \boldsymbol{v})\right)-\operatorname{div}(\mu \mathbb{D}(\boldsymbol{v}))+\nabla P &=& \boldsymbol{F} & \hbox { in } &\Omega_{T} \\
					\operatorname{div} \boldsymbol{v} &=& 0 & \hbox { in } &\Omega_{T}
					\\
			\boldsymbol{v}&=&\boldsymbol{v}_d & \hbox { on } &\Sigma\setminus\Sigma_{3}\\
		-\mathbb{S}(\boldsymbol{v},\pi)\boldsymbol{n}&=&\mathbf{0} & \hbox { on } &\Sigma_{3}
				\end{array}
				\label{System1} 
			\end{equation}
			where $\boldsymbol{v}$  is the flow velocity ($m/s$), $P$ is the pressure (Pa) scaled by the density $\rho$ ($kg\, m^{-3}$). 
			The parameter $\mu$ represents the dynamic viscosity of the fluid (blood) and is equal to $\rho\nu$, while $\nu$ is the kinematic viscosity($kg\, m^{-1}s^{-1}$). 
			$\mathbb D(\boldsymbol{v}) :=\frac 1 2 \left( \nabla u + \nabla u ^T  \right) $ is the strain rate tensor, $\mathbb{S}(\boldsymbol{v},\pi) := \frac 1 \rho \left( \mu\mathbb D(\boldsymbol{v}) - P I \right)$ is the Cauchy stress tensor, $\operatorname{div}(\boldsymbol{v} \otimes \boldsymbol{v}):=(\boldsymbol{v}\cdot\nabla)\boldsymbol{v}$ and $\boldsymbol{F}$ is a right hand side.  
   Concerning the initial data of the fluid velocity, it is well known that it has to be carefully prescribed, since it should be divergence-free to be admissible. Unfortunately, in hemodynamic computations, this quantity is usually unknown, hence chosen equal to zero everywhere or, as a better guess, as the solution of a stationary Stokes problem. In \cite{Salmon12}, a solution to this problem is presented. 
   The issue of boundary conditions is of primary importance in simulating blood flow and a huge literature has been dedicated to this topic in the last years, as reviewed for instance in \cite{FI14,QMV17}. At the inflow(on $\Gamma_1$), we impose a constant velocity $\boldsymbol{v}=\begin{cases}
   	&\boldsymbol{v}_d \qquad \,\,\text{on}\quad \text{$\Sigma_{1}$} ,\\
   	&\mathbf{0} \qquad\quad  \text{on}\quad \text{$\Sigma_{2}\cup \Sigma_{4}\cup\Sigma_5$}.
   \end{cases}$, since blood comes from the microcirculation, modeled by a quasi-steady/steady Stokes flow.  On  $\Gamma_2$, $\Gamma_4$ and $\Gamma_5$, we impose $\boldsymbol{v} = 0$, and at the outflow (on $\Gamma_3$), we impose in a first approximation $\mathbb{S}(\boldsymbol{v}, p)n = 0$, called do-nothing classical approach. 
			
				\item[$\bullet$]\textbf{Potential distribution}:  When alternating electric fields are applied to resistive materials like tissue, heating occurs due to both conduction losses (resistive heating from ion movement) and dielectric losses (caused by the rotation of molecules in the alternating electric field). However, in the frequency range below $1~ MHz$, dielectric losses are negligible \cite{Berjano2006}, and therefore we only consider resistive heating in this model.
				Thus, the resulting electric field in the tissue can be modeled  by  the Laplace's equation 
				$
				\operatorname{div} (\lambda \nabla \varphi)=0,
				$
				where $\lambda$ is the electrical conductivity of the material $(\mathrm{S} / \mathrm{m})$ and $\varphi$ is the electric potential $(V)$. The electric field intensity $\mathbf{E}(\mathrm{V} / \mathrm{m})$ and current density $\mathbf{J}\left(\mathrm{A} / \mathrm{m}^{2}\right)$ are then computed from the equations 
				$\mathbf{E}=-\nabla \varphi $ and $\mathbf{J}=\lambda \mathbf{E} $.
				The local power density resulting in tissue heating is the product of current density $\boldsymbol{J}$ and electric field intensity $\boldsymbol{E}$, which is then used to calculate the temperature distribution via the heat-transfer equation \cite{Dieter2010}. 
				Since the alternating current between the inserted electrodes is of radiofrequency  around \(500 \mathrm{kHz}\), displacement currents are negligible. Therefore we can use the electrostatic approach to describe the potential field
				$\operatorname{div} \mathbf{J}=0$ which allows to the following equation
				\begin{equation}\label{System2}
					\begin{array}{rllll}
						- \operatorname{div}(\lambda \nabla \varphi)&=&0 & \text { in } & \Omega_T\\  
						(\lambda\nabla \varphi) \cdot \boldsymbol{n}&=& g & \text { on } &\Sigma_{5} \\
						\varphi &=&0 & \text { on } &\Sigma\setminus\Sigma_{5}
					\end{array}
				\end{equation}
				The functions  $\lambda$ represents the  electric conductivity, $g$ stands for a current which is induced via the boundary part $\Gamma_N$. 
					\item[$\bullet$] \textbf{Temperature distribution}: The application of an electrical potential at the tip electrode of the catheter produces resistive heating at the cardiac tissue and the surrounding blood. A modification of Penne's bioheat equation models both the heating by the direct application of RF current and the conductive heating \cite{ref1,gonzalez2016computational,Dieter2010} and we can write the bio-heat equation as 
$$
\begin{array}{lll}
\frac{\partial \theta}{\partial t}+\boldsymbol{v} \cdot \nabla \theta-\operatorname{div}(\gamma \nabla \theta)=Q_{RF}+Q_{m}-Q_{p}& \text { in } & \Omega_T,
\end{array}
$$
					where $\gamma$ is the thermal conductivity of the medium, $Q_{RF}=\lambda(\theta)|\nabla \varphi|^{2}$ is the distributed heat source from the electrical field, 
					 $Q_{m}$ is the metabolic heat generation and $Q_{p}$ is the heat loss due to the blood perfusion or the energy dissipation term. Note that,   					
		in \cite{wongchadakul2023natural}, the metabolic heat $Q_{m}$ and the blood perfusion $Q_{p}$ are omitted for short ablation times. Moreover, in our case of the application of the RFA, we consider this last term to be non-negligible and equal to $Q_{p}=-\nu(\theta) \mathbb{D}(\boldsymbol{v}):\mathbb{D}(\boldsymbol{v})\footnote{$\mathbb{D}(\boldsymbol{v}): \mathbb{D}(\boldsymbol{v})=\sum_{i,j=1}^{d}\frac{\partial_i\boldsymbol{v}_i}{\partial \boldsymbol{x}_j}\times\frac{\partial_i\boldsymbol{v}_j}{\partial \boldsymbol{x}_j}$} $.  
		Using this, then we get the following formulation:
						\begin{equation}
							\begin{array}{lll}
						\frac{\partial \theta}{\partial t}+\boldsymbol{v} \cdot \nabla \theta-\operatorname{div}(\gamma(\theta) \nabla \theta)=\lambda(\theta)|\nabla \varphi|^{2}+\nu(\theta) \mathbb{D}(\boldsymbol{v}): \mathbb{D}(\boldsymbol{v})& \text { in } & \Omega_T,
						\end{array}
						\label{System3}
					\end{equation}
				with \(\gamma\), $\lambda$ and $\nu$, are depending on the temperature $\theta$ and are satisfy the somes assumptions(see Section \ref{Section2}), which we will describe later (see Section \ref{Section5}). To complete the boi-heat equation we consider the followings boundary and initial conditions 
					\[
					\begin{array}{rllll}
						(\gamma(\theta)  \nabla \theta) \cdot \boldsymbol{n} + \alpha \theta &=&\alpha \theta_l & \text { on } & \Sigma \\
						\theta(x, 0)&=&\theta_{0} & \text { in } &\Omega 
					\end{array}
					\] 
					where	$\alpha$ is the heat transfer coefficient regulating the convective heat flux through the boundary $\partial \Omega$. $\theta_l$ and $\theta_0$ are given  boundary and initial data, respectively.
						\item[$\bullet$] \textbf{Coupled electro-thermo-fluid}:
			To get the electro-thermo-fluid model we couple all introduced equations described above( equations (\ref{System1}), (\ref{System2}) and (\ref{System3})), and then we have the following strong formulation:

						\begin{equation}
							\left\{
							\begin{array}{rclll}
								\boldsymbol{v}_{t}+\operatorname{div}(\boldsymbol{v} \otimes \boldsymbol{v})-\operatorname{div}(\nu(\theta) \mathbb{D}(\boldsymbol{v}))+\nabla P &=& \boldsymbol{F}(\theta), & \text { in } &\Omega_{T}, \\
								\operatorname{div} \boldsymbol{v} &=& 0, & \text { in } &\Omega_{T} \\
								\boldsymbol{v}&=&\boldsymbol{v}_d, & \text { on } &\Sigma\setminus\Sigma_{3}, \\
								-\mathbb{S}(\boldsymbol{v},P)\boldsymbol{n}&=&\mathbf{0}, & \text { on } &\Sigma_{3}, \\
								\boldsymbol{v}(\boldsymbol{x}, 0)&=&\boldsymbol{v}_{0},  & \text { on } &\Omega ,\\
								\theta_{t} - \operatorname{div}(\gamma(\theta) \nabla \theta) + \boldsymbol{v} \cdot \nabla \theta - \nu(\theta) \mathbb{D}(\boldsymbol{v}): \mathbb{D}(\boldsymbol{v})-(\lambda(\theta)\nabla \varphi)\cdot \nabla \varphi &=&0, & \text { in } & \Omega_{T} , \\
								(\gamma(\theta)\nabla \theta) \cdot \boldsymbol{n} + \alpha \theta &=&\alpha \theta_l, & \text { on } &\Sigma,\\
								\theta(x,0)&=&\theta_{0}, & \text { in }& \Omega , \\
								- \operatorname{div}(\lambda(\theta) \nabla \varphi)&=&0, & \text { in } & \Omega_T, \\
								(\lambda(\theta)\nabla \varphi) \cdot \boldsymbol{n}&=& g, & \text { on } &\Sigma_{5}, \\
								\varphi &=&0 ,& \text { on } &\Sigma\setminus\Sigma_{5}.
							\end{array}\right.
							\label{System}
						\end{equation}
		\end{enumerate}
		In the Boussinesq approximation  \cite{Latif2009,Martynenko2005}, all physical parameters are assumed to be constant and $\boldsymbol{F}$ proportional to the variation of the density and therefore the variation of temperature; $\boldsymbol{F} \propto \theta-\theta_0$. Nevertheless, in this work, those assumptions are not essential, and we will allow for a temperature dependence of the viscosity $\nu$ and consider a more general hypothesis on $\boldsymbol{F}$ that we describe later(see Section \ref{Section2} and Subsection \ref{Example2}). 

\subsection{Main contributions and difficulties}

The studies on radiofrequency ablation have led us to study these models, both theoretically and numerically, to obtain critical information on the electrical and thermal behavior of ablation in a quick and cost-effective manner. Additionally, several of these studies have raised questions about potential risks that doctors may face during surgical procedures and ways to avoid them. Moreover, the aim is to develop new techniques or improve existing ones. In another case, the temperature produced by the ablation catheter when it comes into contact with heart tissue can influence blood flow. Inversely, the impact of blood flow on this heat should be considered. In this context, we propose in this works a new system that models radiofrequency ablation phenomena by coupling the incompressible Navier--Stokes system, which modulates blood flow, with a thermistor model, whose heat source equation takes into account viscous energy dissipation and the electric field. 
\\
\noindent	The highlights of the present paper can be stated as follows: Our study concerns three main parts: modeling, well-posedness of the model, and numerical simulation. In the modeling part, our model (1) is a new improved model of the one proposed in \cite{gonzalez2016computational} by considering the phenomena of viscous energy dissipation.We mention that our proposed model is general-case and contains a term $\nu(\theta) \mathbb{D}(\boldsymbol{v}): \mathbb{D}(\boldsymbol{v})$. Thus, from a modeling viewpoint, this is more close to reality.} Concerning the mathematical analysis part, we prove the existence of the weak solutions in both two and three-dimensional spaces by using Faedo--Galerkin method and Schauder fixed-point theory to deal with the strong nonlinearities in our model. In addition, we prove the uniqueness of the weak solution under some additional conditions on the data and on the solutions. The last part deals with the numerical simulation. First, the variational formulation is discretized by the finite element method in a domain with fairly realistic geometry. Second, some numerical experiments of the proposed model are provided. It is worth mentioning that the study of the proposed model have a challenges in the theoretical and numerical investigations. In fact, the model (1) has a strong nonlinearities namely: the convective term $\operatorname{div}(\boldsymbol{v}\otimes \boldsymbol{v})$ and $\nabla\cdot(\nu(\theta)\mathbb{D}(\boldsymbol{v})) $ of Navier--Stokes and the transport term $\boldsymbol{v}\cdot \nabla \theta$, the dissipative terms ($  \nu(\theta) \mathbb{D}(\boldsymbol{v}): \mathbb{D}(\boldsymbol{v})$ and $(\lambda(\theta)\nabla \varphi)\cdot \nabla \varphi$) with quadratic growth of $\nabla \boldsymbol{v}$ and $\nabla \varphi$. Note that the classical techniques such as the energy method can not be used to prove global (in time) existence results. For this reason, addition to the Faedo--Galerkin method we use the point-fixed strategy. Additionally, the physical and biological properties of the tissues present a serious obstacles. Indeed, all model variables must fall within specific ranges and the results of numerical experiments must be consistent with these criteria. For example, the electrical and thermal conductivities show significantly variable values due to phenomena associated with the high temperatures reached during RFA, such as the vaporization of water at temperatures close to $100\, ^{\circ}C$ and the ensuing sudden increase in impedance, which hampers the delivery of RF power, thus limiting the size of the lesion.

\subsection{Related works}
 We mention that systems reduced to heat-potential coupled models (thermistors) or to Navier--Stokes-heat coupled models are widely discussed in the literature. Let us quote here some references for the theoretical analysis of the first coupling, that is to say, the models of thermal potential.
Time-dependent thermistor equations in particular have been widely studied as described in ~\cite{Allegretto1992, ref2, Li-Yang2015, Yuan1994}.
Among these works: the existence of the solution using the maximum principle and the fixed point argument in \cite{Allegretto1992}, the existence of the weak solution for an arbitrarily large time interval using the Faedo--Galerkin method in \cite{ref2}.
Recently,  the existence and uniqueness of the solution for the thermistor problem without non-degenerate assumptions in \cite{Li-Yang2015}.
For the special case where the thermal conductivity is constant, the authors in~\cite{Yuan1994} proved the existence and uniqueness of the solution in three-dimensional space and its continuity $\alpha$-H{\"o}lder, it is possible to obtain greater regularity of the solution by making appropriate assumptions about the initial and boundary conditions. Moreover, this system has motivated other areas of applied mathematics, such as optimal control and inverse problems, namely the identification of the frequency factor and the energy of the thermal damage function for different types of tissues such as liver, breast, heart, etc., and the development of rapid numerical simulation to predict tissue temperature and thus provide simultaneous guidance during an intervention~\cite{johnson2002,villard2005}.
We also cite the two interesting works~\cite{MEINLSCHMIDT-part1,MEINLSCHMIDT-part2} where the well-posed character is shown and the optimality conditions are derived by considering the parameter $g$ as a boundary check.
\\

\noindent The theoretical studies of the second coupling have been the subject of several works, we refer the reader to ~\cite{Benes2011,Benes2007,ref1,deteix2014coupled} and the references contained therein.
Among these works, the authors of \cite{deteix2014coupled} studied the case where viscosity and thermal conductivity are nonlinear and temperature dependent. In the aforementioned paper, the authors derived the existence of solutions, without restriction on the data, by Brouwer's fixed point theorem.
On the other hand, in \cite{deugoue2021globally} the authors have studied the existence and the uniqueness of the solution using the Brouwer fixed point, the Faedo--Galerkin method, and some compactness results for a model variant of this coupling namely, the globally modified Navier--Stokes problem coupled to the heat equation. The authors studied the stability of the discrete solution in time using the energy approach.
We mention the paper \cite{ref1} where the authors considered the external force in the heat equation containing an energy dissipation term. Moreover, they proved the existence of the solution for three-dimensional space using Galerkin's method and Schauder's fixed point theorem. \\

\noindent From a computational point of view, there are very few computational analyses for the general case. 
We mention the work in \cite{Allegretto1999} where the semi-discretization in space by the finite volume method has been proposed to solve the thermistor problem. 
The $L^{2}$-norm and $H^{1}$-norm error estimates have been obtained for the piecewise linear approximation, a linearized $\theta$-Galerkin finite element method is proposed to solve the coupled system, and optimal error estimates are derived in different cases, including the standard Crank--Nicolson and shifted Crank--Nicolson schemes in~\cite{Mbehou2018}.
Numerical methods and analysis for the thermistor system for special conductivities, namely, for the linear and the exponential choices, have been investigated by many authors~\cite{Akrivis2005,  ref2, Elliott1995, Li2014, Li2012}. 
For a constant thermic conductivity in two-dimensional space, the optimal $L^{2}$-norm error estimate of a mixed finite element method with a linearized semi-implicit Euler scheme was obtained in \cite{Akrivis2005} under a weak time-step condition. 
The error analysis for the three-dimensional space is given in  \cite{Elliott1995} using a linearized semi-implicit Euler scheme with a linear Galerkin finite element method. 
An optimal $L^{2}$-norm error estimate was obtained under specific conditions on the step size discretization. 
For the $d$-dimensional space $(d=2,3)$, the authors in~\cite{Li2012} proved the time-step condition of commonly-used linearized semi-implicit schemes for the time-dependent nonlinear Joule heating equations with Galerkin finite element approximations and optimal error estimates of a Crank--Nicolson Galerkin method for the nonlinear thermistor equations~\cite{Li2014} and backward differential formula type similarly schemes approximations~\cite{Gao2015}. 
Different methods have been considered to approximate the Navier--Stokes equations coupled to the heat equation \cite{ antonietti2022virtual, BULICEK2009,deteix2014coupled}. 
The authors in \cite{deteix2014coupled} presented a convergence analysis for an iterative scheme based on the so-called coupled prediction scheme. 
Finally, the virtual element discretization of the Navier--Stokes equations coupled to the heat equation where the viscosity depends on temperature was studied in \cite{antonietti2022virtual}. The authors showed that it is well-posed and proved optimal error estimates for this discretization.

\subsection{Outline of paper}
 The rest of this paper is organized as follows. In the next section, we introduce the basic notations and some appropriate functional spaces. Then, we formulate the problem according to a variational framework and introduce one of the main results of our work. In Section \ref{Section3}, we investigate the existence, uniqueness, and energy estimates of solutions to linearized (decoupled) initial boundary value problems for the Navier--Stokes, electric potential, and heat with non-smooth coefficients. Moreover, we prove the existence item of the main result using Schauder's fixed point. To complete the proof of the main result, we prove the uniqueness of the solution. Finally, we discuss in Section \ref{Section5} some numerical simulation in two-dimensional space by the finite element method. 

\section{Mathematical frameworks and variational  formulation}\label{Section2}
	In this section we introduce fundamental notations and appropriate functional spaces. Next, we formulate the problem within a variational framework and finally we present the well-posedness result for the proposed model. {Motion that, for simplicity of the mathematical analysis, we chose the flow velocity at the inflow boundary equal to zero, i.e., $\boldsymbol{v}_d=0$ on $\Sigma_1$.}

We consider $p$, $q$, $r$, $p' \in [1,\infty ]$, where $ p'$ denotes the conjugate exponent to $p > 1$ namely $1/p+1/p'=1$. For an arbitrary $r \in[1,+\infty]$, $L^{r}(\Omega)$ is the usual Lebesgue space equipped with the norm $\|\cdot\|_{L^{r}(\Omega)}$, and $W^{m, r}(\Omega)$, $m \geq 0$ $(m$ need not to be an integer$)$, denotes the usual Sobolev space with the norm $\|\cdot\|_{W^{m, r}(\Omega)}$. By $C(0,T; E)$ we denote the space of all abstract	functions $\psi$ such that $\psi$: $(0,T) \longmapsto E$ is continuous, where $E$ is a Banach space.  Further, we denote by $W^{-m,p}(\Omega)$ the dual space of $W^{m,p'}(\Omega)$.  
For simplicity reason, we denote shortly $\mathbf{W}^{m, p}(\Omega) \equiv W^{m, p}(\Omega)^{d}$, $\mathbf{L}^{r} (\Omega) \equiv L^{r}(\Omega)^{d}$.
\\

\noindent For the mathematical analysis of our model (\ref{System}), we use the following embedding results (see~\cite[Theorem 7.58]{Adams} and ~\cite{Kufner})
\begin{equation}\label{inject1}
	\begin{array}{llll}
		W^{m, p}(\Omega) & \hookrightarrow L^{q}(\Omega), & \|\phi \|_{L^q{(\Omega)}} \leq c\|\phi \|_{W^{m, p}(\Omega)}, & p \leq q<\infty, mp=d, \\
		W^{m, p}(\Omega) & \hookrightarrow L^{q}(\Omega), & \|\phi\|_{L^q(\Omega)} \leq c\|\phi\|_{W^{m, p}(\Omega)}, & p \leq q \leq d p /(d-m p), m p<d,\\
		W^{m, p}(\Omega) & \hookrightarrow L^{\infty}(\Omega), & \|\phi\|_{L^{\infty}(\Omega)} \leq c\|\phi\|_{W^{m, p}(\Omega)}, & m p>d,
	\end{array} 
\end{equation}
for every $\phi \in W^{m, p}(\Omega)$. Further, there exists a continuous operator $\mathfrak{R}_0: W^{m, p}(\Omega) \rightarrow L^{q}(\partial \Omega)$ such that
\begin{equation}\label{trace}
	\|\mathfrak{R}_0(\phi)\|_{L^{q}(\partial \Omega)} \leq c\|\phi\|_{W^{m, p}(\Omega)}\;  \forall \phi \in W^{m, p}(\Omega) \hbox{ with }\left\{\begin{array}{ll}
		1 \leq mp<d, & q=\frac{d p-p}{d-m p}, \\
		p \geq \max \{1,d / m\}, & q \in[1, \infty).
	\end{array}\right.
\end{equation}
For $s$ be real number such that $s \leq m+1, s-1 / p=k+\sigma$, where $k \geq 1$ is an integer and $0<\sigma<1$, the following mapping $\mathfrak{R}_1$ is continuous 
\begin{equation}\label{trace-N}
	\begin{aligned}
		\mathfrak{R}_1: W^{s, p}(\Omega) & \rightarrow  W^{s-1-1 / p, p}(\Gamma), \\
		\varphi & \mapsto   \frac{\partial \varphi}{\partial n}{ \mid_\Gamma}.
	\end{aligned}
\end{equation}
Let consider the following spaces  
$$
\begin{array}{ll} 
	\mathcal{E}_{\boldsymbol{v}}:=&\left\{\boldsymbol{v} \in \boldsymbol{C}^{\infty}(\overline {\Omega}) ;\text{}\operatorname{div} \boldsymbol{v}=0,\text{} \operatorname{supp} \boldsymbol{v} \cap (\Sigma\setminus\Sigma_{3})=\emptyset\right\}\text{, }\\
	\mathcal{E}_{\varphi}:=&\left\{\varphi \in C^{\infty}(\overline {\Omega}) ;\text{} \operatorname{supp} \varphi \cap \Sigma\setminus\Sigma_{5}=\emptyset\right\}\text{, }\\
	\mathcal{E}_{\theta}:=&\left\{\theta \in C^{\infty}(\overline {\Omega});\text{} \operatorname{supp}\theta \text{ is compact }\right\},
\end{array}
$$      
and let $\mathbf{V}_{\boldsymbol{v}}^{m, p}$ be the closure of $\mathcal{E}_{\boldsymbol{v}}$ in the norm of $\mathbf{W}^{m, p}(\Omega)$, $m \geq 0$ and $1 \leq p \leq \infty$. Similarly, let $V_{\varphi}^{m, p}$ and $V_{\theta}^{m, p}$ are the closures of $\mathcal{E}_{\varphi}$ and $\mathcal{E}_{\theta}$ in the norm of $W^{m, p}(\Omega)$. Then $V_{\theta}^{m, p}$, $V_{\varphi}^{m, p}$  and $\mathbf{V}_{\boldsymbol{v}}^{m, p}$ are Banach spaces with the norms of the spaces $W^{m, p}(\Omega)$ and $ \mathbf{W}^{m, p}(\Omega)$, respectively. Note that the Banach space $V_{\varphi}^{}$  is defined by 
$ V_{\varphi}=\{\phi \in {V}_{\varphi}^{1,2}, \nabla\phi\in \mathbf{L}^{4}(\Omega)\}$ equipped with the norm
$$
{\|{\phi}\|}_{V_{\varphi}}:=\|\phi\|_{V_{\varphi}^{1,2}}+\|\nabla{\phi}\|_{\mathbf{L}^4(\Omega)}.
$$ 
Finally, for $m>0, \mathbf{V}_{v}^{-m, p}$ denotes the dual space of $\mathbf{V}_{v}^{m, p'}$ normed by  $$
{\|\mathbf{v}\|}_{\mathbf{V}_{v}^{-m, p}}=\sup_{\mathbf{0} \neq \mathbf{w} \in \mathbf{V}_{v}^{m, p'}}\frac{|\langle\boldsymbol{v}, \mathbf{w}\rangle|}{\|\mathbf{w}\|_{\mathbf{W}^{m, p'}} },
$$
where $\langle\cdot, \cdot\rangle$ denotes the duality pairing.

\noindent If the functions $\boldsymbol{v}$, $\boldsymbol{w}$, $\boldsymbol{z}$, $\theta$, $\phi$, $\varphi$, $\chi$ and $\psi$ are sufficiently smooth so that the following integrals make sense, we also  introduce the following notations: 
\begin{align*} 
	(\boldsymbol{v}, \boldsymbol{w})&=\int_{\Omega} \boldsymbol{v} \cdot \boldsymbol{w} ~d{\mathbf{x}}
	, \qquad
	& (\theta, \psi)_{\Gamma}&=\int_{\Gamma}   \theta \psi ~\mathrm{d} \Gamma, \\
	a_{u}(\theta ; \boldsymbol{v}, \boldsymbol{w})&=\int_{\Omega} \nu(\theta) \mathbb{D}(\boldsymbol{v}): \mathbb{D}(\boldsymbol{w}) ~d{\mathbf{x}} , \qquad
	&
	\tilde{a}_{u}(\boldsymbol{v}, \boldsymbol{w})&=\int_{\Omega} \mathbb{D}(\boldsymbol{v}): \mathbb{D}(\boldsymbol{w}) ~d{\mathbf{x}}, \\
	a_{\theta}(\phi ; \theta, \psi)&=\int_{\Omega} \gamma(\phi) \nabla \theta \cdot \nabla \psi ~d{\mathbf{x}}, \qquad
	&
	\tilde{a}_{\theta}(\theta, \varphi)&=\int_{\Omega} \nabla \theta \cdot \nabla \varphi ~d{\mathbf{x}}, \\
	c_{\varphi}(\phi ,\varphi,\psi )&=\int_{\Omega} \lambda(\phi) \nabla \varphi \cdot \nabla \varphi \psi ~d{\mathbf{x}}, \qquad
	&
	a_{\varphi}(\phi ,\varphi,\chi )&=\int_{\Omega} \lambda(\phi) \nabla \varphi \cdot \nabla \chi ~d{\mathbf{x}}, \\
	d(\boldsymbol{v}, \theta, \psi)&=\int_{\Omega} (\boldsymbol{v} \cdot \nabla \theta) \psi ~d{\mathbf{x}} , \qquad
	&
	e(\theta ; \boldsymbol{v}, \boldsymbol{w}, \psi)&=\int_{\Omega} \nu(\theta) \mathbb{D}(\boldsymbol{v}): \mathbb{D}(\boldsymbol{w}) \psi~d{\mathbf{x}},\\
	b(\boldsymbol{v}, \boldsymbol{w}, \boldsymbol{z})&=\int_{\Gamma_{N}}(\boldsymbol{v} \otimes \boldsymbol{w}):(\boldsymbol{n} \otimes \boldsymbol{z}) \mathrm{d} \Gamma-\int_{\Omega}(\boldsymbol{v} \otimes \boldsymbol{w}): \mathbb{D}(\boldsymbol{z}) ~d{\mathbf{x}}.
\end{align*}  

To formulate model (\ref{System}) in a variational sense and then state the main result of the paper, the following smoothness property is needed. 
\begin{lemma}[cf \cite{ref1}]\label{lem} Let $\boldsymbol{\mathcal{U}}$  a Banach space  defined by  
	$$
	\boldsymbol{\mathcal{U}}:=\left\{\boldsymbol{z} \mid \boldsymbol{z} \in L^{\infty}\left(0,T ; \mathbf{V}_{\boldsymbol{v}}^{0,4}\right) \cap L^{4}\left(0,T ; \mathbf{V}_{\boldsymbol{v}}^{1,4}\right)\right\}\text{, }
	$$
	equipped with the norm 
	$$
	{\|\boldsymbol{z}\|}_{\boldsymbol{\mathcal{U}}}:=\|\boldsymbol{z}\|_{L^{\infty}\left(0,T ; \mathbf{V}_{\boldsymbol{v}}^{0,4}\right)}+\|\boldsymbol{z}\|_{L^{4}\left(0,T ; \mathbf{V}_{\boldsymbol{v}}^{1,4}\right)}.
	$$ 
	Then  \begin{equation}\label{injet2}
		\boldsymbol{\mathcal{U}} \hookrightarrow  L^{64 / 7}\left(0,T ; \mathbf{W}^{7 / 16,4}\right).
	\end{equation}
	In addition, for all $ (\boldsymbol{v}, \boldsymbol{w}) \in \boldsymbol{\mathcal{U}}^2$, $b(\boldsymbol{v}, \boldsymbol{w}, \cdot) \in L^{4}\left(0,T ; \mathbf{V}_{\boldsymbol{v}}^{-1,4}\right)$ and there exists some positive constant $C_{b},$ independent of $T,$ such that
	\begin{equation}\label{cstb}
		\|b(\boldsymbol{v}, \boldsymbol{w}, \cdot)\|_{L^{4}\left(0,T ; \mathbf{V}_{\boldsymbol{v}}^{-1,4}\right)} \leq C_{b}T^{1/32} {\|\boldsymbol{v}\|}_{\boldsymbol{\mathcal{U}}}{\|\boldsymbol{w}\|}_{\boldsymbol{\mathcal{U}}}\text{.}
	\end{equation}
	
\end{lemma}  

We will solve the system (\ref{System}) with the followings assumptions: 

\begin{itemize}
	\item [\textbf{(A1).}] The functions  
	$\boldsymbol{F}=\boldsymbol{F}(\cdot)$, $\nu=\nu(\cdot)$,   $\lambda=\lambda(\cdot)$ and  $\gamma=\gamma(\cdot)$ being positives, bounded and continuous for the temperature. Without any further reference, we assume 
	\begin{eqnarray}
		0\leq F_{i}(s) \leq C_{F}<+\infty \qquad & \forall  s\in \mathbb{R}, i=1,...,d, \label{condF}\\
		0<\nu_{1} \leq \nu(s) \leq \nu_{2}<+\infty \qquad & \forall s\in \mathbb{R},\label{cond-nu}\\	
		0<\lambda_{1} \leq \lambda(s) \leq \lambda_{2}<+\infty \qquad & \forall s\in \mathbb{R}, \label{cond-lambda}\\
		0<\gamma_{1} \leq \gamma(s) \leq \gamma_{2}<+\infty \qquad & \forall s\in \mathbb{R}, \label{cond-gamma}
	\end{eqnarray}
	where $C_F$, $\nu_1$, $\nu_2$, $\lambda_1$, $\lambda_2$, $\gamma_1$ and  $\gamma_2$  are positive constants.
	\item [\textbf{(A2).}] The initials data  $ \boldsymbol{v}_{0} \in \mathbf{V}_{\boldsymbol{v}}^{1 / 2,4}$, $ \theta_{0} \in L^{2}$.
	\item [\textbf{(A3).}] The other assumptions on the data are, 
	\begin{equation}\label{assump-data} 
		\mathbf{F} \in L^{4}\left(0,T ; \mathbf{V}_{\boldsymbol{v}}^{-1,4}\right) ,\text{ }\theta_{l} \in L^{2}\left(0,T ; V_{\theta}^{1,2}\right) \cap L^{2}\left(0,T ; V_{\theta}^{-1,2}\right) \text{ and } \, g\in L^4\left(0,T; W^{-1/2,2}(\Gamma)\right).	
	\end{equation}
	\item [\textbf{(A4).}]	 There exists a constant $C_S$ (to be specified later, cf (\ref{cnstStokes})) such  that 
	\begin{equation}\label{cond1}
		C_{S}\left(\nu_{2}-\nu_{1}\right)<1.
	\end{equation}
	\item [\textbf{(A5).}] There exists $\beta \in\left(0,1 / 2\left(1-C_{S}\left(\nu_{2}-\nu_{1}\right)\right)\right)$ such that (recall that the constants $C_F$ and $C_S$ are defined in \eqref{condF} and \eqref{cond1}, respectively)
	\begin{equation}\label{cond2}
		C_{S} C_{E} C_{F} C_d(\Omega,T) +C_{S}\left\|\boldsymbol{v}_{0}\right\|_{\mathbf{V}_{\boldsymbol{v}}^{1 / 2,4}}  \leq \frac{\beta^{2}}{C_{S} C_{b} T^{1 / 32}},
	\end{equation}
\end{itemize}
where 
$C_d(\Omega,T)=T^{1/4}d^{1/2}\operatorname{meas}(\Omega)^{1/4} $,  
$C_{E} $ is the constant of the embedding $\mathbf{W}^{1,4 / 3} \hookrightarrow ~ \mathbf{L}^{4 / 3}$ and $C_b$ is a given constant from $(\ref{cstb})$. \\

\noindent We will utilize the following notion of weak solution for our model (\ref{System}).
\begin{definition}\label{def-weak-solution} (Weak solution).
	A triplet $(\boldsymbol{v}, \theta, \varphi)$ is called variational solution of the problem (\ref{System}) if $ \boldsymbol{v}_{0} \in \mathbf{V}_{\boldsymbol{v}}^{1 / 2,4}$, $ \theta_{0} \in L^{2}$,
	$\boldsymbol{v} \in \boldsymbol{\mathcal{U}}$, $  \boldsymbol{v}_{t} \in L^{4}\left(0,T ; \mathbf{V}_{\boldsymbol{v}}^{-1,4}\right)$, 
	$ \theta \in L^{2}\left(0,T; V_{\theta}^{1,2}\right)$, $ \theta_{t} \in L^{2}\left(0,T ; V_{\theta}^{-1,2}\right)$ and  
	$\varphi \in L^{4}\left(0,T ; V_{\varphi}^{}\right)$, 		
	and  the following variational formulations
	\begin{eqnarray}
		\left\langle\boldsymbol{v}_{t}, \boldsymbol{w}\right\rangle+a_{u}(\theta ; \boldsymbol{v}, \boldsymbol{w})+b(\boldsymbol{v}, \boldsymbol{v}, \boldsymbol{w})&=&\left\langle\boldsymbol{F}(\theta), \boldsymbol{w}\right\rangle, \label{eqvar1} \\  
		\left\langle\theta_{t}, \psi \right\rangle+a_{\theta}(\theta ; \theta, \psi )+d(\boldsymbol{v}, \theta, \psi)+\alpha(\theta, \psi )_{\Gamma}   
		-e(\theta ; \boldsymbol{v}, \boldsymbol{v}, \psi)-c_{\varphi}(\theta ,\varphi,\psi )&=&\alpha(\theta_{l}, \psi )_{\Gamma,}\label{eqvar2} \\ 			a_{\varphi}(\theta ; \varphi, \chi )&=&{(g, \chi )}_{\Gamma_N},\label{eqvar3}
	\end{eqnarray}
	hold   for every $\left( \boldsymbol{w}, \psi, \chi \right) \in \mathbf{V}_{\boldsymbol{v}}^{1,4 / 3} \times V_{\theta}^{1,2}  \times V_{\varphi}^{1,2}$ and for almost every $t \in (0,T)$ and
	\begin{eqnarray}
		\boldsymbol{v}(\boldsymbol{x}, 0) &=\boldsymbol{v}_{0}(\boldsymbol{x})  & \text { in }  \Omega,\label{eqvar4} \\
		\theta(\boldsymbol{x}, 0) &=\theta_{0}(\boldsymbol{x})  & \text { in }  \Omega. \label{eqvar5}
	\end{eqnarray}
\end{definition}

\noindent Our main result is
\begin{theorem}(Well-posedness).
	\begin{enumerate}
		\item \textbf{Existence}: 
		Assume that assumptions $\textbf{(A1)}$, $\textbf{(A2)}$, $\textbf{(A3)}$, $\textbf{(A4)}$, and $\textbf{(A5)}$  hold. Then System (\ref{eqvar1})-(\ref{eqvar3}) has a weak solution $\left(\boldsymbol{v}, \theta, \varphi\right) \in \boldsymbol{\mathcal{U}}\times  C(0,T;V_{\theta}^{1,2}) \times  L^{4}\left(0,T ; V_{\varphi} \right)$ in the sense of Definition \ref{def-weak-solution}.
		
		%
		
		\item \textbf{Uniqueness}: Let, in addition to assumptions $\textbf{(A1)}$-$\textbf{(A5)}$  $\mathbf{F}$, $\nu$, $\lambda$ and $\gamma$ are Lipschitz continuous, i.e 
		\begin{equation}\label{lipchtz}
			\begin{aligned}
				\left|\mathbf{F}\left(z_{1}\right)-\mathbf{F}\left(z_{2}\right)\right| & \leq L_{\mathbf{F}}\left|z_{1}-z_{2}\right| \quad \forall z_{1}, z_{2} \in \mathbb{R}\left(L_{\mathbf{F}}=\text { const }>0\right),\\
				\left|\nu\left(z_{1}\right)-\nu\left(z_{2}\right)\right| & \leq L_{\nu}\left|z_{1}-z_{2}\right| \quad \forall z_{1}, z_{2} \in \mathbb{R}\left(L_{\nu}=\text { const }>0\right),\\
				\left|\lambda\left(z_{1}\right)-\lambda\left(z_{2}\right)\right| & \leq L_{\lambda}\left|z_{1}-z_{2}\right| \quad \forall z_{1}, z_{2} \in \mathbb{R}\left(L_{\lambda}=\text { const }>0\right),\\
				\left|\gamma\left(z_{1}\right)-\gamma\left(z_{2}\right)\right| & \leq L_{\gamma}\left|z_{1}-z_{2}\right| \quad \forall z_{1}, z_{2} \in \mathbb{R}\left(L_{\gamma}=\text { const }>0\right),
			\end{aligned}
		\end{equation}
		and if $\nabla\theta\in L^s(0,T; W^{1,2}(\Omega))$, $\boldsymbol{v}\in L^s(0,T; \boldsymbol W^{1,2}(\Omega))$ and $\varphi\in L^s(0,T; W^{1,2}(\Omega))$ where $s=\frac{8}{4-d}$, then the weak solution of problem $(\ref{eqvar1})-(\ref{eqvar3})$ is unique.  
	\end{enumerate}
	\label{Main_Result}
\end{theorem}   


\section{Well-posedness analysis}\label{Section3}
	This section deals with the proof of   Theorem \ref{Main_Result}. 
	Note that a major difficulty for our model (1) is the strong coupling in the highest derivatives. Therefore, standard parabolic theory is not directly applicable to our system due to the dissipation terms. We point out that this model is strongly nonlinear and so no maximum principle applies. In our case, we used the point fixed strategy.
	
Let us briefly describe the rough idea of the proof. 
For given temperature, say $\overline{\theta}$, in the kinematic viscosity $\nu$ and the last term in the first line in $(\ref{DecoupledNS})$ i.e the right-hand side $\boldsymbol{F}$, we find $\boldsymbol{v}$, the solution of the decoupled Navier--Stokes equations $(\ref{DecoupledNS})$ via the Banach contraction principle. Further, we find $\varphi$, the solution of decoupled potential equation (\ref{DecoupledPhi}) using  Lax-Milgram's method with the electrical conductivity is also depend of $\overline{\theta}$. Now with $\boldsymbol{v}$ and  $\varphi$ in hand, we find $\theta$, the solution of the linearized heat equation with the second member is the some of two terms, the dissipative energy  and electric field using the approach of Faedo--Galerkin. Finally, we show that the map $\overline{\theta} \rightarrow \theta$ is completely continuous and maps some ball independent of the choice $\overline{\theta}$ into itself. Hence, the existence of at least one solution follows from the Schauder's point fixe theorem. In Section \ref{section3.4}, the uniqueness of the solution is established under the assumptions of Lipschitz continuity of the data (see equation $(\ref{lipchtz})$) and higher regularity of $\theta$.


\subsection{Well-posedness of decoupled Navier--Stokes system and decoupled potential equation }\label{section3.1}
For an arbitrary fixed
$ \overline {\theta} \in L^{2}\left(0,T ; L^{2}\right)$,  we consider the decoupled Navier--Stokes  problem  
\begin{equation}
	\left\{\begin{array}{rclll}
		\boldsymbol{v}_{t}-\operatorname{div}(\nu(\overline {\theta}) \mathbb{D}(\boldsymbol{v}))+\operatorname{div}(\boldsymbol{v} \otimes \boldsymbol{v})+\nabla P & =&\boldsymbol{F}(\overline{\theta} ), & \hbox { in } \Omega_{T} \\
		\operatorname{div} \boldsymbol{v} & =&0 &\hbox{ in } \Omega_{T}, \\
		\boldsymbol{v} & =&0 &\hbox{ on }  \Sigma\setminus\Sigma_{3},\\
		-P \boldsymbol{n}+\nu(\overline {\theta}) \mathbb{D}(\boldsymbol{v}) \boldsymbol{n} & =&\mathbf{0} & \hbox { on } \Sigma_{3},\\   
		\boldsymbol{v}(\boldsymbol{x}, 0) & =&\boldsymbol{v}_{0}(\boldsymbol{x}) &\,\hbox { in } \hbox { }\Omega. 
	\end{array}\right .
	\label{DecoupledNS}
\end{equation}
and the decoupled potential problem
\begin{equation} 
	\left\{	\begin{array}{rclll}
		- \operatorname{div}(\lambda(\overline {\theta}) \nabla \varphi) & = &0 & \hbox { in } &\Omega_{T}, \\
		(\lambda(\overline {\theta})\nabla \varphi) \cdot \boldsymbol{n} & = & g & \hbox { on } &\Sigma_{5}, \\
		\varphi & = & 0 & \hbox { on } &\Sigma\setminus\Sigma_{5}. 
	\end{array}\right.
	\label{DecoupledPhi}
\end{equation}
\begin{remark}In~\cite{ref1} the authors proved  the existence and the uniqueness of the solution to the decoupled Navier--Stokes  problem \eqref{DecoupledNS} such that $\boldsymbol{v} \in \boldsymbol{\mathcal{U}}$ with $\boldsymbol{v}_{t} \in L^{4}\left(0,T ; \boldsymbol{V}_{\boldsymbol{v}}^{-1,4}\right)$ for $d=3$. For $d=2$,  the new paper~\cite{New-Benes2022} prescribed an additional condition of the viscosity on $\Gamma_{N}$ i.e the
	homogeneous Neumann boundary condition  and consider  the small data. 
	The authors  shown that the solution satisfies  $\boldsymbol{v} \in L^{\infty}\left(0,T ; \boldsymbol{V}_{\boldsymbol{v}}^{s-1,2}\right) \cap L^{2}\left(0,T ; \boldsymbol{V}_{\boldsymbol{v}}^{s,2}\right) $ with $\boldsymbol{v}_{t} \in L^{2}\left(0,T ; \boldsymbol{V}_{\boldsymbol{v}}^{s-2,2}\right)$ for $s>1$.
\end{remark}
\noindent We define the following nonlinear mapping
\begin{eqnarray}\label{T1}
	\begin{array}{lclll}
		\mathcal{S}_1: & L^{2}\left(0,T ; L^{2}\right) &\rightarrow &\boldsymbol{\mathcal{U}}\times L^{4}\left(0,T ; V_{\varphi}\right)\\
		& \overline{\theta}                           & \mapsto    & (\boldsymbol{v},\varphi) &
	\end{array}
\end{eqnarray}
where $\boldsymbol{v}$ is solution of \eqref{DecoupledNS} and $\varphi$ is solution of (\ref{DecoupledPhi}). The above mapping is well defined as we will show in the following (cf Theorem \ref{th-Stokes} and Theorem \ref{varfi} ). In order to prove $\boldsymbol{v}$ is solution of \eqref{DecoupledNS}, we need the following lemma.

\begin{lemma}[The decoupled Stokes problem]\label{Decoupled Stokes}
	Let $\boldsymbol{f} \in L^{4}\left(0,T ; \mathbf{V}_{\boldsymbol{v}}^{-1,4}\right)$ and $\boldsymbol{v}_{0} \in \mathbf{V}_{\boldsymbol{v}}^{1 / 2,4}$. Then there exists a unique function $\boldsymbol{v} \in \boldsymbol{\mathcal{U}}$ with $\boldsymbol{v}_{t} \in L^{4}\left(0,T ; \boldsymbol{V}_{\boldsymbol{v}}^{-1,4}\right)$ satisfying
	$$
	\left\langle\boldsymbol{v}_{t}, \boldsymbol{w}\right\rangle+\tilde{a}_{u}\left(\nu_{2} \boldsymbol{v}, \boldsymbol{w}\right)=\langle\boldsymbol{f}, \boldsymbol{w}\rangle,
	$$
	for all $\boldsymbol{w} \in \boldsymbol{V}_{\boldsymbol{v}}^{1,4 / 3}$ and almost every $t \in (0,T)$,
	$
	\boldsymbol{v}(\boldsymbol{.}, 0)=\boldsymbol{v}_{0}(\boldsymbol{.}) \quad \hbox { in } \Omega
	$. Moreover,  $\boldsymbol{v}$ satisfying the  following inequality  
	\begin{equation}\label{cnstStokes}
		{\|\boldsymbol{v}\|}_{\boldsymbol{\mathcal{U}}} \leq C_{S}\left(\|\boldsymbol{f}\|_{L^{4}\left(0,T ; \mathbf{V}_{\boldsymbol{v}}^{-1,4}\right)}+\left\|\boldsymbol{v}_{0}\right\|_{\mathbf{V}_{\boldsymbol{v}}^{1 / 2,4}}\right),
	\end{equation}
	where $C_{S}$ is a positive constant independent of $\boldsymbol{v}$, $\boldsymbol{f}$ and $\boldsymbol{v}_0$.
\end{lemma}
\begin{proof}
	We refer to~\cite[Theorem 4.1 and Corollary 4.2]{ref1} for the proof.
\end{proof}

The following theorem ensures the well-posedness of decoupled Navier--Stokes system $(\ref{DecoupledNS})$.
\begin{theorem}[Well-posedness of System $(\ref{DecoupledNS})$]\label{th-Stokes}Let 
	$\overline {\theta} \in L^{2}\left(0,T ; L^{2}\right)$ and $\boldsymbol{v}_{0} \in \mathbf{V}_{\boldsymbol{v}}^{1 / 2,4} $. 
	Then there exists a unique function $\boldsymbol{v} \in \boldsymbol{\mathcal{U}}$ with $\boldsymbol{v}_{t} \in L^{4}\left(0,T ; \mathbf{V}_{\boldsymbol{v}}^{-1,4}\right)$ such that
	\begin{equation}\label{VSDNS}
		\left\{ \begin{array}{rllll}
			\left\langle\boldsymbol{v}_{t}, \boldsymbol{w}\right\rangle+a_{u}(\overline {\theta} ; \boldsymbol{v}, \boldsymbol{w})+b(\boldsymbol{v}, \boldsymbol{v}, \boldsymbol{w}) & = & \langle\boldsymbol{F}(\overline {\theta}), \boldsymbol{w}\rangle, & \forall \boldsymbol{w} \in \mathbf{V}_{\boldsymbol{v}}^{1,4 / 3} \hbox{ and a.e } t \in (0,T), \\
			\boldsymbol{v}(\boldsymbol{x}, 0) & =& \boldsymbol{v}_{0}(\boldsymbol{x}),& \forall x \in \Omega .
		\end{array} \right. 
	\end{equation}
\end{theorem} 
\begin{proof}
	By H{\"o}lder inequality and the Sobolev embedding $(\ref{inject1})$, we infer 
	$$  
	\begin{aligned}
		|(\boldsymbol{F}(\overline {\theta}), \boldsymbol{w})|  
		& \leq\|\boldsymbol{F}(\overline {\theta})\|_{\mathbf{L}^{4}} \|\boldsymbol{w}\|_{\boldsymbol{L}^{4/3} }\\
		& \leq C_E\|\boldsymbol{F}(\overline {\theta})\|_{\mathbf{L}^{4}} \|\boldsymbol{w}\|_{\boldsymbol{W}^{1,4/3} },
	\end{aligned}
	$$
	for every $\boldsymbol{w} \in \mathbf{W}^{1,4/3}$. Then, 
	$$
	\begin{aligned}
		\|\boldsymbol{F}(\overline {\theta})\|_{\boldsymbol{V_u}^{-1,4}}	& \leq C_E \left(\int_{\Omega}(|\boldsymbol{F}(\overline {\theta})|_{E})^{4} ~d{\mathbf{x}}\right)^{1/4} \\
		& \leq C_E \left(\int_{\Omega}(d^{1/2}C_F)^{4} ~d{\mathbf{x}}\right)^{1/4} \\
		& \leq C_{E} C_{F}d^{1/2}\operatorname{meas}(\Omega)^{1/4},
	\end{aligned}
	$$ 
	where $|\cdot|_E$ denotes the Euclidean vector norm. Raising both sides and integrating over $(0,T)$ we get, 
	$$  
	\begin{aligned}
		\|\boldsymbol{F}(\overline{\theta})\|_{L^{4}(0,T ; \mathbf{V}_{\boldsymbol{v}}^{-1,4})} 
		& \leq C_{E} C_{F}C_d(\Omega, T).
	\end{aligned}
	$$ 
	Let  $\overline{\boldsymbol{v}} \in \boldsymbol{\mathcal{U}}$. By Lemma  \ref{lem} and Lemma \ref{Decoupled Stokes},   there exists the unique function $\boldsymbol{v} \in \boldsymbol{\mathcal{U}}$ with $\boldsymbol{v}_{t} \in L^{4}\left(0,T ; \mathbf{V}_{\boldsymbol{v}}^{-1,4}\right)$ such that
	$$ \left \{ 
	\begin{array}{rclll}
		\left\langle\boldsymbol{v}_{t}, \boldsymbol{w}\right\rangle+\tilde{a}_{u}\left(\nu_{2} \boldsymbol{v}, \boldsymbol{w}\right) & = & (\boldsymbol{F}(\overline {\theta}), \boldsymbol{w})+\tilde{a}_{u}\left(\nu_{2} \overline{\boldsymbol{v}}, \boldsymbol{w}\right)-a_{u}(\overline {\theta} , \overline{\boldsymbol{v}}, \boldsymbol{w})-b(\overline{\boldsymbol{v}},\overline{\boldsymbol{v}},\boldsymbol{w}),\\
		\boldsymbol{v}(\boldsymbol{x}, 0) & = & \boldsymbol{v}_{0}(\boldsymbol{x}) \qquad \qquad \qquad \qquad \qquad \qquad  \qquad \hbox{ in }   \Omega,
	\end{array} \right .
	$$
	for every $\boldsymbol{w} \in \mathbf{V}_{\boldsymbol{v}}^{1,4/3}$ and almost every $t \in (0,T)$ satisfying the estimate
	$$
	\begin{array}{ll}
		{\|\boldsymbol{v}\|}_{\boldsymbol{\mathcal{U}}} & \leq  C_{S}\left(  		\|\boldsymbol{F}(\overline {\theta})\|_{L^{4}\left(0,T ; \mathbf{V}_{\boldsymbol{v}}^{-1,4}\right)}+\|\tilde{a}_{u}\left(\nu_{2} \overline{\boldsymbol{v}}, \boldsymbol{.}\right)-a_{u}(\overline {\theta} , \overline{\boldsymbol{v}}, \boldsymbol{.})\|_{L^{4}\left(0,T; \mathbf{V}_{\boldsymbol{v}}^{-1,4}\right)}+\|b(\overline {\boldsymbol{v}}, \overline {\boldsymbol{v}}, \cdot)\|_{L^{4}\left(0,T; \mathbf{V}_{\boldsymbol{v}}^{-1,4}\right)}+\left\|\boldsymbol{v}_{0}\right\|_{\mathbf{V}_{\boldsymbol{v}}^{1 / 2,4}}  \right)\\ 
		& \leq	C_{S}\left(C_{E} C_{F}C_d(\Omega, T)+\left(\nu_{2}-\nu_{1}\right)\|\overline{\boldsymbol{v}}\|_{\boldsymbol{\mathcal{U}}}+ C_{b} T^{1 / 32}\|\overline{\boldsymbol{v}}\|_{\boldsymbol{\mathcal{U}}}^{2}+\left\|\boldsymbol{v}_{0}\right\|_{\mathbf{V}_{\boldsymbol{v}}^{1 / 2,4}}\right).
	\end{array}
	$$
	%
	Let us define the ball
	\begin{equation}\label{ball}
		B:=\left\{\overline{\boldsymbol{v}} \in \boldsymbol{\mathcal{U}},{\|\overline{\boldsymbol{v}}\|}_{\boldsymbol{\mathcal{U}}} \leq \frac{\beta}{C_{S} C_{b} T^{1 / 32}}\right\}.
	\end{equation}
	Under the assumptions \textbf{(A4)} and \textbf{(A5)}, and
	for every $\overline{\boldsymbol{v}} \in B$, we have 
	$$
	\begin{aligned}
		{\|\boldsymbol{v}\|}_{\boldsymbol{\mathcal{U}}} & \leq	C_{S}\left(C_{E} C_{F}C_d(\Omega, T)+\left\|\boldsymbol{v}_{0}\right\|_{\mathbf{V}_{\boldsymbol{v}}^{1 / 2,4}}+ C_{b} T^{1 / 32}{\|\overline{\boldsymbol{v}}\|}_{\boldsymbol{\mathcal{U}}}^{2}+ \left(\nu_{2}-\nu_{1}\right){\|\overline{\boldsymbol{v}}\|}_{\boldsymbol{\mathcal{U}}}\right)\\
		& \leq \frac{2 \beta^{2}}{C_{S} C_{b} T^{1 / 32}}+C_{S}\left(\nu_{2}-\nu_{1}\right) \frac{\beta}{C_{S} C_{b} T^{1 / 32}} \\
		& \leq \frac{\beta\left(2 \beta+C_{S}\left(\nu_{2}-\nu_{1}\right)\right)}{C_{S} C_{b} T^{1 / 32}} \\
		&<\frac{\beta}{C_{S} C_{b} T^{1 / 32}} .
	\end{aligned}
	$$
	Hence, the map $\mathcal{T}: \boldsymbol{\mathcal{U}} \rightarrow \boldsymbol{\mathcal{U}}$ with $\boldsymbol{\mathcal{T}}(\overline{\boldsymbol{v}})=\boldsymbol{v}$ maps $B$ into $B$. Further, by virtue of Lemma \ref{Decoupled Stokes}  and Lemma \ref{lem}, for every $\overline{\boldsymbol{v}}_{1}, \overline{\boldsymbol{v}}_{2} \in B$ we have
	$$
	\begin{aligned}
		{\left\|\boldsymbol{v}_{1}-\boldsymbol{v}_{2}\right\|}_{\boldsymbol{\mathcal{U}}} &={\left\|\mathcal{T}\left(\overline{\boldsymbol{v}}_{1}\right)-\mathcal{T}\left(\overline{\boldsymbol{v}}_{2}\right)\right\|}_{\boldsymbol{\mathcal{U}}} \\
		& \leq\left(C_{S}\left(\nu_{2}-\nu_{1}\right)+C_{S} C_{b} T^{1 / 32}\left({\left\|\overline{\boldsymbol{v}}_{1}\right\|}_{\boldsymbol{\mathcal{U}}}+{\left\|\overline{\boldsymbol{v}}_{2}\right\|}_{\boldsymbol{\mathcal{U}}}\right)\right)\left\|\overline{\boldsymbol{v}}_{1}-\overline{\boldsymbol{v}}_{2}\right\|_{\boldsymbol{\mathcal{U}}} \\
		& \leq\left(C_{S}\left(\nu_{2}-\nu_{1}\right)+2 \beta\right){\left\|\overline{\boldsymbol{v}}_{1}-\overline{\boldsymbol{v}}_{2}\right\|}_{\boldsymbol{\mathcal{U}}} .
	\end{aligned}
	$$
	From the assumptions \textbf{(A4)} and \textbf{(A5)}, it follows that $\left(C_{S}\left(\nu_{2}-\nu_{1}\right)+2 \beta\right)<1$.
	Thus,  the map $\mathcal{T}: \boldsymbol{\mathcal{U}} \rightarrow \boldsymbol{\mathcal{U}}$ with $\boldsymbol{\mathcal{T}}(\overline{\boldsymbol{v}})=\boldsymbol{v}$ is a contraction operator in the ball $B$. 
	Using the Banach fixed point theorem, we deduce the existence of at least one fixed point $\boldsymbol{v} \in \boldsymbol{\mathcal{U}}$, such that $\boldsymbol{\mathcal{T}}(\boldsymbol{v})=\boldsymbol{v}$, which is uniquely determined in the ball $B$. \\
	\noindent Let's show that the solution is globally unique in the space $\boldsymbol{\mathcal{U}}$. Let $\boldsymbol{v}_{1}$, $\boldsymbol{v}_{2} \in\boldsymbol{\mathcal{U}}$ two variational  solutions of the decoupled Navier--Stokes system (\ref{VSDNS}) and noted $\boldsymbol{v}=\boldsymbol{v}_{1}-\boldsymbol{v}_{2}$,  then $\boldsymbol{v}$ satisfied the following equation
	$$
	\left\langle\partial_t\boldsymbol{v}, \boldsymbol{w}\right\rangle+a_{\boldsymbol{v}}(\overline {\theta} ; \boldsymbol{v}, \boldsymbol{w})+b\left(\boldsymbol{v}, \boldsymbol{v}_{2}, \boldsymbol{w}\right)+b\left(\boldsymbol{v}_{1}, \boldsymbol{v}, \boldsymbol{w}\right)=0   
	$$
	holds for all $\boldsymbol{w} \in V^{1,4 / 3}$ and almost every $t \in (0,T)$ and $\boldsymbol{v}(\boldsymbol{x}, 0)=\mathbf{0}$. Hence, we consider $\boldsymbol{w}=\boldsymbol{v}$ then we have 
	$$
	\begin{aligned}
		\frac{1}{2} \frac{\mathrm{d}}{\mathrm{d} t}\|\boldsymbol{v}(t)\|_{\mathbf{V}_{\boldsymbol{v}}^{0,2}}^{2}+\nu_{1}\|\boldsymbol{v}(t)\|_{\mathbf{V}_{\boldsymbol{v}}^{1,2}}^{2} & \leq c\left(\left|b\left(\boldsymbol{v}_{1}(t), \boldsymbol{v}(t), \boldsymbol{v}(t)\right)\right|+\left|b\left(\boldsymbol{v}(t), \boldsymbol{v}_{2}(t), \boldsymbol{v}(t)\right)\right|\right) \\
		& \leq c_{1}\left\|\boldsymbol{v}_{1}(t)\right\|_{\mathbf{L}^{4}}\|\nabla \boldsymbol{v}(t)\|_{\mathbf{L}^{2}}\|\boldsymbol{v}(t)\|_{\mathbf{L}^{4}}+c_{2}\|\boldsymbol{v}(t)\|_{\mathbf{L}^{4}}^{2}\left\|\nabla \boldsymbol{v}_{2}(t)\right\|_{\mathbf{L}^{2}}.
	\end{aligned}
	$$
	By the interpolation inequality  
	$$
	\|\boldsymbol{v}(t)\|_{\mathbf{L}^{4}} \leq c\|\boldsymbol{v}(t)\|_{\mathbf{V}_{\boldsymbol{v}}^{1,2}}^{\zeta}\|\boldsymbol{v}(t)\|_{\mathbf{L}^{2}}^{1-\zeta}, \hbox{ where }\zeta=d/4, 
	$$
	we get	  
	$$
	\begin{array}{l}
		\displaystyle \frac{1}{2} \frac{\mathrm{d}}{\mathrm{d} t}\|\boldsymbol{v}(t) \|_{\mathbf{V}_{\boldsymbol{v}}^{0,2}}^{2}+\nu_{1}\|\boldsymbol{v}(t)\|_{\mathbf{V}_{\boldsymbol{v}}^{1,2}}^{2} 
		\displaystyle \leq c_{1}\left\|\boldsymbol{v}_{1}(t)\right\|_{\mathbf{L}^{4}}\|\boldsymbol{v}(t)\|_{\mathbf{V}_{\boldsymbol{v}}^{1,2}}^{1+\zeta}\|\boldsymbol{v}(t)\|_{\mathbf{L}^{2}}^{1-\zeta}+c_{2}\|\boldsymbol{v}(t)\|_{\mathbf{V}_{\boldsymbol{v}}^{1,2}}^{2\zeta}\|\boldsymbol{v}(t)\|_{\mathbf{L}^{2}}^{2(1-\zeta)}{\left\|\boldsymbol{v}_{2}(t)\right\|}_{\mathbf{W}^{1,2}}.
	\end{array}
	$$
	Applying Young's inequality,  we deduce
	\begin{equation}\label{b+bestimat}
		\frac{1}{2} \frac{\mathrm{d}}{\mathrm{d} t}\|\boldsymbol{v}(t)\|_{\mathbf{V}_{\boldsymbol{v}}^{0,2}}^{2}+\nu_{1}\|\boldsymbol{v}(t)\|_{\mathbf{V}_{\boldsymbol{v}}^{1,2}}^{2} \leq \delta\|\boldsymbol{v}(t)\|_{\mathbf{V}_{\boldsymbol{v}}^{1,2}}^{2}+c_{\delta}\|\boldsymbol{v}(t)\|_{\mathbf{L}^{2}}^{2}\left(\left\|\boldsymbol{v}_{1}(t)\right\|_{\mathbf{L}^{4}}^{\frac{2}{1-\zeta}}+\left\|\boldsymbol{v}_{2}(t)\right\|_{\mathbf{W}^{1,2}}^{\frac{1}{1-\zeta}}\right), 
	\end{equation}
	where $\delta>0$ can be chosen arbitrarily small and therefore
	$$
	\frac{\mathrm{d}}{\mathrm{d} t}\|\boldsymbol{v}(t)\|_{\mathbf{V}_{\boldsymbol{v}}^{0,2}}^{2} \leq 2 c_{\delta}\|\boldsymbol{v}(t)\|_{\mathbf{V}_{\boldsymbol{v}}^{0,2}}^{2}\left(\left\|\boldsymbol{v}_{1}(t)\right\|_{\mathbf{L}^{4}}^{\frac{2}{1-\zeta}}+\left\|\boldsymbol{v}_{2}(t)\right\|_{\mathbf{W}^{1,2}}^{\frac{1}{1-\zeta}}\right) .
	$$
	Finally, an application of Gronwall inequality and the fact that $\boldsymbol{v}(\boldsymbol{x}, 0)=\mathbf{0}$ lead to the uniqueness.  
\end{proof}


In order to ensure the well-posedness of the decoupled potential equation in space $V_{\varphi}$, 
we need the following regularity result of \cite{bulivcek2016existence}.  
\begin{lemma}\label{Lem-regularity} Let $\Omega \subset \mathbb{R}^d$ be a bounded domain with a smooth boundary. Assume that $f \in\boldsymbol{L}^2(\Omega)$ and $a \in C(\overline{\Omega})$ with $\min _{\overline{\Omega}} a>0$. Let $w$ be the weak solution of the following problem
	$$
	\begin{cases}
		&-\operatorname{div}(a \nabla w)=\operatorname{div} f \quad \mbox{ in } \Omega, \\
		&w=0 \qquad \qquad \qquad  \, \, \, \mbox{ on } \partial \Omega .
	\end{cases}
	$$
	Then for each $p>2$,  there exists a positive constant $c^*$ depending only on $d$, $\Omega$, $a$  and $p$ such that if $f \in\boldsymbol{L}^p(\Omega)$ then we have
	$$
	{\|\nabla w\|}_{\boldsymbol{L}^p} \leq c^*\left({\|f\|}_{\boldsymbol{L}^p}+{\|\nabla w\|}_{\boldsymbol{L}^2}\right)
	$$
\end{lemma}
For the decoupled problem (\ref{DecoupledPhi}), we have the following result.  
\begin{theorem}[Well-posedness of System (\ref{DecoupledPhi})]\label{varfi}
	Let the function $\overline {\theta} \in L^{2}\left(0,T ; L^{2}(\Omega)\right)$ and $g\in L^4\left(0,T; W^{-1/2,2}(\Gamma)\right)$ are be given. Then there exists a unique  function $\varphi \in L^{4}\left(0,T ; V_{\varphi}^{}\right)$ solution of   (\ref{DecoupledPhi}), such that
	\begin{equation}\label{fi}
		a_{\varphi}(\overline {\theta}(t) ,\varphi(t),\chi )=\langle g(t), \chi\rangle,
	\end{equation}
	for every $\chi \in V_{\varphi}^{1,2}$ and almost every $t \in (0,T)$, and 
	\begin{equation}\label{normfi}
		\left\|\varphi\right\|_{L^4(0,T;V_{\varphi})} \leq c \left\|g\right\|_{L^4\left(0,T; W^{-1/2,2}(\Gamma)\right)}\hbox{,}
	\end{equation}
	for some constant $c>0$ independent of $\overline {\theta}$, $\varphi$ and $\chi$.
\end{theorem}
\begin{proof}
	The existence of solution to the problem  \eqref{DecoupledPhi} in $V_{\varphi}^{1,2}$ results from  the Lax-Milgram Theorem. The estimate of $\varphi$ in  $V_{\varphi}^{1,2}$ that is  
	\begin{equation}\label{normfi-W12}
		\left\|\varphi\right\|_{V_{\varphi}^{1,2}} \leq c \left\|g\right\|_{L^{2}(\Gamma)}\hbox{,}
	\end{equation}
	where $c>0$ is a constant independent of $\overline {\theta}$, $\varphi$ and $g$. 
	The regularity of the solution $\varphi$ follows from Lemma \ref{Lem-regularity}.  
	In fact, since $g\in L^4\left(0,T; W^{-1/2,2}(\Gamma)\right)$, we can set $\phi \in V_{\varphi}^{} $ such that $(\lambda(\overline {\theta})\nabla \phi) \cdot \boldsymbol{n} = g $, which is well defined according to the trace operator defined in \eqref{trace-N}. Moreover, let $a=\lambda(\overline{\theta})$ and  $\varphi \in V_{\varphi}^{1,2}$ the solution of  \eqref{DecoupledPhi}. Noted $w=\varphi-\phi \in V_{\varphi}^{1,2}$, then $w$ is the weak solution of the following problem:  
	$$ 
	\begin{aligned}
		-\operatorname{div}(a \nabla w)&=\operatorname{div} f \quad &\text { in } \Omega, \\
		w&=0 \quad  &\text { on } \Gamma .
	\end{aligned}
	$$
	whith $f=\lambda(\overline{\theta})\nabla\phi\in L^4(\Omega) $. Then we have
	$$
	\|\nabla \varphi\|_{L^4(\Omega)} \leq c^*\left(\|f\|_{L^4(\Omega)} +\|\nabla \varphi\|_{L^2(\Omega)}\right).
	$$
	According to  \eqref{normfi-W12}, we complete the proof.
\end{proof}
\subsection{Well-posedness of the decoupled heat equation}\label{section3.2}
For a fixed $\boldsymbol{v} \in \boldsymbol{\mathcal{U}}$ and $ \varphi \in L^4(0,T; V_{\varphi}) $, consider the linear heat equation
\begin{eqnarray}
	\left\{
	\begin{array}{lllll}
		\theta_{t} - \operatorname{div}(\gamma(\overline {\theta}) \nabla \theta) + \boldsymbol{v} \cdot \nabla \theta & = & \nu(\overline {\theta}) \mathbb{D}(\boldsymbol{v}): \mathbb{D}(\boldsymbol{v})+(\lambda(\overline {\theta})\nabla {\varphi})\cdot \nabla {\varphi} & \text { in }  & \Omega_{T}, \\
		(\gamma(\overline {\theta}) \nabla \theta) \cdot \boldsymbol{n} + \alpha \theta & = & \alpha \theta_l  & \text { on } & \Sigma,\\
		\theta(\boldsymbol{x}, 0) & = & \theta_{0}(\boldsymbol{x}) & \,\text { in } & \Omega.
	\end{array}\right . 
	\label{DecoupledH}
\end{eqnarray}
Concerning the well-posedness of the decoupled heat equation,  we have the following theorem
\begin{theorem}\label{wellposdecoupledH} 
	Let  $\overline{\theta} \in L^{2}\left(0,T ; L^{2}\right)$, $\boldsymbol{v} \in \boldsymbol{\mathcal{U}}$ and $\varphi \in V_{\varphi}^{1,2}$ be the solution of the problem $(\ref{DecoupledNS})$ and $(\ref{DecoupledPhi})$ respectively. Further, let $\theta_{0} \in L^{2}$ and $\theta_{l} \in L^{2}\left(0,T ; V_{\theta}^{1,2}\right)\cap L^{2}\left(0,T ; V_{\theta}^{-1,2}\right)$. Then there exists the uniquely determined function $\theta \in L^{2}\left(0,T ; V_{\theta}^{1,2}\right)$ 
	with $\theta_{t} \in L^{2}\left(0,T ; V_{\theta}^{-1,2}\right)$ such that
	\begin{equation}\label{equ18}
		\left\langle\theta_{t}, \psi\right\rangle+a_{\theta}(\overline {\theta} ; \theta, \psi)+d(\boldsymbol{v}, \theta, \psi)+\alpha(\theta, \psi)_\Gamma=e(\overline {\theta} ; \boldsymbol{v}, \boldsymbol{v}, \psi)+c_{\varphi}(\overline {\theta} ,{\varphi},\psi )+\alpha\left\langle\theta_{l}, \psi \right\rangle_{\Gamma},
	\end{equation}
	for every $\psi \in V_{\theta}^{1,2}$ and almost every $t \in (0,T)$, and
	\begin{equation}\label{equ19}
		\theta(\boldsymbol{x}, 0)=\theta_{0}(\boldsymbol{x}) \quad \text { in } \Omega\texttt{.}
	\end{equation}
\end{theorem}
\begin{proof} We posed $\langle h(t),.\rangle=e(\overline {\theta} ; \boldsymbol{v}, \boldsymbol{v}, \cdot)+c_{\varphi}(\overline {\theta} ,{\varphi}, \cdot )+\alpha(\theta_l, \cdot)$. Since for $\overline {\theta} \in L^{2}\left(0,T ; L^{2}\right)$, $\boldsymbol{v} \in \boldsymbol{\mathcal{U}}$ and $\varphi\in L^{4}(0,T; V_{\varphi}^{})$ we have even $e(\overline {\theta} ; \boldsymbol{v}, \boldsymbol{v}, \cdot) \in L^{2}\left(0,T ; L^{2}\right)$ and $c_{\varphi}(\overline {\theta} ,{\varphi},. )\in  L^{2}\left(0,T ; W^{-1,2}\right)$, we conclude that $h\in L^{2}\left(0,T ; W^{-1,2}\right)$. Then,  the function $h$ is estimated by, 
	\begin{equation}\label{estem-h}
		\begin{aligned}
			||h(t)||_{V_{\theta}^{-1,2}}\leq & ||e(\overline {\theta} ; \boldsymbol{v}, \boldsymbol{v}, \cdot)||_{\boldsymbol{W}^{-1,2}}+||c_{\varphi}(\overline {\theta} ,{\varphi}, \cdot )||_{W^{-1,2}}+\alpha||\theta_l||_{V_{\theta}^{-1,2}}\\
			\leq & ||\nu(\overline {\theta}) \mathbb{D}(\boldsymbol{v}): \mathbb{D}(\boldsymbol{v})||_{L^2} + c{|| \varphi||}_{V_{\varphi}}^2+ \alpha||\theta_l||_{V_{\theta}^{-1,2}}\\
			\leq & c\left( ||\boldsymbol{v}||^2_{\boldsymbol{W}^{1,4}} +{|| \varphi ||}_{V_{\varphi}}^2+ ||\theta_l||_{V_{\theta}^{-1,2}}  \right).
		\end{aligned} 
	\end{equation} 
	Let $\left\{e_{n}\right\}_{n=1}^{\infty}$ be the orthogonal basis of the separable space $V_{\theta}^{1,2}$ such that
	$$
	V_{\theta}^{1,2}=\overline{\bigcup_{k=1}^{\infty} \mathcal{V}_{n}}^{W^{1,2}} \quad, \quad \mathcal{V}_{n}=\operatorname{span}\left\{e_{1}, e_{2}, \ldots, e_{n}\right\}.
	$$
	Define the Faedo--Galerkin approximation $\theta_{n} \in W^{1,2}\left(0,T ; \mathcal{V}_{k}\right)$
	\begin{equation}\label{equ21}
		\theta_{n}(t)=\sum_{i=1}^{k} \zeta_{i}(t) e_{i},
	\end{equation}
	where, $\zeta_{i}: I \rightarrow \mathbb{R}$ to be determined. Next, we consider the problem
	\begin{equation}\label{equ22} 
		\left\langle\frac{d}{dt}\theta_{n}(t), \psi\right\rangle+a_{\theta}(\overline {\theta}(t) ; \theta_{n}(t), \psi)+d\left(\boldsymbol{v}(t), \theta_{n}(t), \psi\right)+\alpha(\theta_{n}(t), \psi)_\Gamma=\langle h(t), \psi\rangle,
	\end{equation}
	for every $\psi \in \mathcal{V}_{n}$ and almost every $t \in (0,T)$, and 
	\begin{equation}\label{equ23}
		\theta_{n}(0)=\theta_{n}^0.
	\end{equation} 
	The	equations (\ref{equ22}) and (\ref{equ23}) represent the Cauchy problem for the system of linear ordinary differential equations with measurable coefficients, which ensures the local existence and uniqueness of a generalized solution $\zeta$ on the local time interval $(0,t_n)$ \cite{Kurzweil.1986}. Since $\theta_{n}(t) \in \mathcal{V}_{n},$ let us take $\psi=\theta_{n}(t)$ in (\ref{equ22}) to obtain
	$$  
	\frac{1}{2} \frac{d}{d t}\left\|\theta_{n}(t)\right\|_{L^{2}}^{2}+a_{\theta}(\overline {\theta}(t) ; \theta_{n}(t), \theta_{n}(t))+\alpha(\theta_{n}(t), \theta_{n}(t))_\Gamma=\langle {h}(t), \theta_{n}(t)\rangle-d\left(\boldsymbol{v}(t), \theta_{n}(t), \theta_{n}(t)\right)
	$$
	almost everywhere  $t\in(0,T)$. Hence, we arrive at the estimate
	$$
	\frac{1}{2} \frac{d}{d t}\left\|\theta_{n}(t)\right\|_{L^{2}}^{2}+\int_{\Omega} \gamma(\overline {\theta}(t))\left|\nabla \theta_{n}(t)\right|^{2} ~d{\mathbf{x}} +\alpha\left\|\theta_{n}(t)\right\|_{L^{2}(\Gamma)}^{2} 
	\leq \|h(t)\|_{V_{\theta}^{-1,2}}\left\|\theta_{n}(t)\right\|_{V_{\theta}^{1,2}}+\left|d\left(\boldsymbol{v}(t), \theta_{n}(t), \theta_{n}(t)\right)\right|.
	$$
	Using the Gagliardo--Nirenberg interpolation inequality (cf.~\cite[Theorem 5.8]{Adams2003})
	$$
	\begin{aligned}
		\|\theta_{n}(t)\|_{L^{4}(\Omega)} & \leq c\|\theta_{n}(t)\|_{W^{1,2}(\Omega)}^{\zeta}\|\theta_{n}(t)\|_{L^{2}(\Omega)}^{1-\zeta}, \text{ where }\zeta=d/4,
	\end{aligned}
	$$
	and Young's inequality with parameter $\delta$,
	$
	a b \leq \delta a^{p}+C(\delta) b^{q}$ with $a, b>0, \delta>0,1<p, q<\infty, 1 / p+1 / q=1 
	$
	and  $C(\delta)=(\delta p)^{-q / p} q^{-1}$, the last term can be estimated by
	\begin{equation}\label{equ24}
		\begin{aligned}
			\left|d\left(\boldsymbol{v}(t), \theta_{n}(t), \theta_{n}(t)\right)\right| & \leq\|\boldsymbol{v}(t)\|_{\mathbf{L}^{4}}\left\|\nabla \theta_{n}(t)\right\|_{\mathbf{L}^{2}}\left\|\theta_{n}(t)\right\|_{L^{4}} \\
			& \leq c\|\boldsymbol{v}(t)\|_{\mathbf{L}^{4}}\left\|\theta_{n}(t)\right\|_{W^{1,2}}^{1+\zeta}\left\|\theta_{n}(t)\right\|_{L^{2}}^{1-\zeta} \\
			& \leq \delta\left\|\theta_{n}(t)\right\|_{W^{1,2}}^{2}+C(\delta)\|\boldsymbol{v}(t)\|_{\mathbf{L}^{4}}^{2/(1-\zeta)}\left\|\theta_{n}(t)\right\|_{L^{2}}^{2} .
		\end{aligned}
	\end{equation} Choosing $\delta$ sufficiently small, we have
	\begin{equation}\label{equ25}
		\frac{d}{d t}\left\|\theta_{n}(t)\right\|_{L^{2}}^{2}+c_{1}\left\|\theta_{n}(t)\right\|_{V_{\theta}^{1,2}}^{2} \leq c_{2}\|h(t)\|_{V_{\theta}^{-1,2}}^{2}+c_{3}\|\boldsymbol{v}(t)\|_{\mathbf{L}^{4}}^{2/(1-\zeta)}\left\|\theta_{n}(t)\right\|_{L^{2}}^{2}.
	\end{equation}
	Using the Gronwall's inequality yields
	\begin{equation}\label{equ26}
		\left\|\theta_{n}(t)\right\|_{L^{2}}^{2} \leq\left[\left\|\theta_{n}^{0}\right\|_{L^{2}}^{2}+\int_{0}^{t} c_{2}\|h(s)\|_{V_{\theta}^{-1,2}}^{2} ~d{s}\right] \exp \left(\int_{0}^{t} c_{3}\|\boldsymbol{v}(s)\|_{\mathrm{L}^{4}}^{2/(1-\zeta)} ~d{s}\right)\quad\text{ for all $t \in (0,T)$.}
	\end{equation}  
	The estimates $(\ref{equ25})$ and $(\ref{equ26})$ imply that there exists some constants $C>0$ and $C'>0$ such that 
	\begin{eqnarray}
		\left\|\theta_{n}(t)\right\|_{L^{\infty}\left(0,T ; L^{2}\right)} \leq C, \label{equ27} 
		\\
		\left\|\theta_{n}(t)\right\|_{L^{2}\left(0,T ; V_{\theta}^{1,2}\right)} \leq C'.
		\label{equ28}   
	\end{eqnarray} 
	Now, from $(\ref{equ25})$ and using $(\ref{equ27})-(\ref{equ28})$ we deduce that $\left\{\left(\theta_{n}\right)_{t}\right\}_{n=1}^{\infty}$ is bounded in $L^{2}\left(0,T ; V_{\theta}^{-1,2}\right) $ and allows us to consider a subsequence, again denoted by $\left\{\theta_{n}(t)\right\}_{n=1}^{\infty},$ such that 
	\begin{eqnarray}
		\theta_{n} &\rightarrow& \theta \text { weakly in } L^{2}(0,T ; V_{\theta}^{1,2}), \label{equ29}
		\\
		\partial_t\theta_{n} &\rightarrow& \partial_t\theta \text { weakly in } L^{2}(0,T ; V_{\theta}^{-1,2}), \label{equ30}
		\\
		\theta_{n} &\rightarrow& \theta \text { strongly in } L^{2}(0,T ; L^{2}), \label{equ31}
		\\
		\theta_{n} &\rightarrow& \theta \text { almost everywhere in } \Omega_{T} \text {.} \label{equ32}
	\end{eqnarray}
	Now, we can immediately pass to the limit in $(\ref{equ22})$ and, by $(\ref{equ29})-(\ref{equ32})$, we obtain the solution $\theta \in L^{2}\left(0,T ; V_{\theta}^{1,2}\right) \cap W^{1,2}\left(0,T ; V_{\theta}^{-1,2}\right)$, which satisfies $(\ref{equ18})-(\ref{equ19})$. Consequently, we obtain
	\begin{equation}\label{equ33}
		\left\langle\partial_t\theta, \psi\right\rangle+a_{\theta}(\overline {\theta} ; \theta, \psi)+\alpha(\theta, \psi)_\Gamma=\langle h, \psi\rangle-d\left(\boldsymbol{v}, \theta, \psi\right),
	\end{equation}
	for every $\psi \in V_{\theta}^{1,2}$ and almost every $t \in (0,T)$ and the initial condition
	\begin{equation}\label{equ34}
		\theta(\boldsymbol{x}, 0)=\theta_{0}(\boldsymbol{x}) \quad \text { in } \Omega.
	\end{equation}
	Let $\psi=\theta(t)$ in $(\ref{equ33})$, then we get the estimate
	\begin{eqnarray}\label{equ35}
		\frac{1}{2} \frac{d}{d t}\left\|\theta(t)\right\|_{L^{2}}^{2}+\int_{\Omega} \gamma(\overline {\theta}(t))\left|\nabla \theta(t)\right|^{2} ~d{\mathbf{x}} +\alpha\left\|\theta(t)\right\|_{L^{2}(\Gamma)}^{2} 
		\leq\|h(t)\|_{V_{\theta}^{-1,2}}\left\|\theta(t)\right\|_{V_{\theta}^{1,2}}+\left|d\left(\boldsymbol{v}(t), \theta(t), \theta(t)\right)\right| .
	\end{eqnarray}
	Since the inequality $(\ref{equ24})$ is satisfied for $\theta \in L^{2}\left(0,T ; V_{\theta}^{1,2}\right)$, using the Young inequality and choosing $\delta$ sufficiently small   we get the following estimate  
	\begin{equation}\label{equ36}
		\frac{d}{d t}\left\|\theta(t)\right\|_{L^{2}}^{2}+c_{1}\left\|\theta(t)\right\|_{V_{\theta}^{1,2}}^{2} \leq c_{2}\|h(t)\|_{V_{\theta}^{-1,2}}^{2}+c_{3}\|\boldsymbol{v}(t)\|_{\mathbf{L}^{4}}^{2/(1-\zeta)}\left\|\theta(t)\right\|_{L^{2}}^{2}.
	\end{equation}
	Moreover, by the Gronwall's lemma, we find that
	\begin{equation}\label{equ37}
		\left\|\theta(t)\right\|_{L^{2}}^{2} \leq\left[\left\|\theta({0})\right\|_{L^{2}}^{2}+\int_{0}^{t} c_{2}\|h(s)\|_{V_{\theta}^{-1,2}}^{2} ~d{s}\right] \exp \left(\int_{0}^{t} c_{3}\|\boldsymbol{v}(s)\|_{\mathrm{L}^{4}}^{2/(1-\zeta)} ~d{s}\right)  \quad \text{for all }  t \in (0,T).
	\end{equation}
	Hence  
	\begin{equation}\label{equ38}
		\|\theta\|_{C\left(0,T ; L^{2}\right)}^{2} \leq c_{1} \exp \left(c_{2} T\|\boldsymbol{v}\|_{L^{\infty}\left(0,T ; \mathbf{L}^{4}\right)}^{2/(1-\zeta)}\right)\left[\left\|\theta_{0}\right\|_{L^{2}}^{2}+\|h\|_{L^2(0,T;V_{\theta}^{-1,2})}^{2}\right]\text{.}
	\end{equation}
	For the uniqueness, suppose there are two solutions $\theta_1$, $\theta_2\in V_{\theta}^{1,2}$ of $(\ref{equ18})-(\ref{equ19})$ on $(0,T)$ and denote $\theta=\theta_1-\theta_2$. Then,  
	\begin{equation}\label{unicite}
		\left\langle\theta_{t}, \psi\right\rangle+a_{\theta}(\overline {\theta} ; \theta, \psi)+d(\boldsymbol{v}, \theta, \psi)+\alpha(\theta, \psi)_\Gamma=0,
	\end{equation}
	for every $\psi \in V_{\theta}^{1,2}$ and almost every $t \in (0,T)$ and $\theta(x,0)=0$. Hence
	\begin{equation}
		\frac{d}{d t}\left\|\theta(t)\right\|_{L^{2}}^{2}+c_{1}\left\|\theta(t)\right\|_{V_{\theta}^{1,2}}^{2} \leq c_{2}\|\boldsymbol{v}(t)\|_{\mathbf{L}^{4}}^{2/(1-\zeta)}\left\|\theta(t)\right\|_{L^{2}}^{2}.
	\end{equation}
	Now, the uniqueness follows from Gronwall's inequality and the fact that $\theta(x, 0) = 0$. This completes the proof of the theorem. 
\end{proof}

\begin{remark}
	Note that from $(\ref{equ37})$ and $(\ref{estem-h})$ we have,
	\begin{equation}\label{normtheta}
		\left\|\theta(t)\right\|_{L^{2}}^{2} \leq\left[\left\|\theta({0})\right\|_{L^{2}}^{2}+\int_{0}^{t} c_{1}\left( ||\boldsymbol{v}||^2_{\boldsymbol{W}^{1,4}} + {|| \varphi ||}_{V_{\varphi}^{}}^2+||\theta_l||_{V_{\theta}^{-1,2}}\right)^{2} ~d{s}\right] \exp \left(\int_{0}^{t} c_{2}\|\boldsymbol{v}(s)\|_{\mathrm{L}^{4}}^{8} ~d{s}\right),
	\end{equation} 
	for all $t \in (0,T)$.
	Moreover, from the equations (\ref{ball}) and  (\ref{normfi}), $\theta$ is bounded in $C\left(0,T; L^{2}\right)$ independently of $\overline {\theta}$. 
\end{remark}  
From Theorem \ref{wellposdecoupledH}, we can then define the following nonlinear mapping 
\begin{eqnarray}\label{equ42}
	\begin{array}{lclll}
		\mathcal{S}_2: & \boldsymbol{\mathcal{U}} \times   L^{2}\left(0,T ; V_{\varphi}^{}\right) &\rightarrow & Y \\
		& \left( \boldsymbol{v} , \varphi \right)                            & \rightarrow    & \theta \quad \hbox{solution of (\ref{DecoupledH})},
	\end{array}
\end{eqnarray}
where, the space $Y$ is defined by	$
Y:=\left\{\phi ; \phi \in L^{2}\left(0,T ; W^{1,2}\right), \phi_{t} \in L^{2}\left(0,T ; V_{\theta}^{-1,2}\right)\right\}.
$ 

\subsection{Fixed point strategy}\label{section3.3}
In order to prove the first item of Theorem \ref{Main_Result}, we apply the  Schauder fixed point theorem and  the lemma of Aubin-Lions~\cite{Aubin1963}. So, we consider the Banach spaces
$W^{1,2}$, $W^{-1,2}$ and  $L^{2}$ satisfying the following embeddings
$
W^{1,2} \hookrightarrow \hookrightarrow L^{2}  \hookrightarrow W^{-1,2}
$.   Then, the space  
$
Y
$
is compactly embedded into $L^{2}\left(0,T ; L^{2}\right)$. Moreover, using the results of 
Theorem \ref{th-Stokes}, Theorem \ref{varfi} and Theorem \ref{wellposdecoupledH}, we can defined the mapping $\mathcal{S}$ by 
\begin{eqnarray}\label{equ43}
	\begin{array}{lclll}
		\mathcal{S}: & L^{2}\left(0,T ; L^{2}\right) & \rightarrow & L^{2}\left(0,T ; L^{2}\right) \\
		&    \overline {\theta}                 & \rightarrow    & \mathcal{S} ( \overline{\theta}  ) =  \mathcal{S}_2\circ  \mathcal{S}_1(\overline {\theta}):= \mathcal{S}_2 (\mathcal{S}_1 (\overline {\theta})).
	\end{array}
\end{eqnarray}
Applying the  interpolation theory and using some apriori estimates of $\boldsymbol{v}$, $\varphi$ and $\theta$, we show that $L^{2}\left(0,T ; L^{2}\right) \rightarrow Y$ is completely continuous. Hence,  using some operator theory results, we  get the compactness of the operator   $\mathcal{S}: L^{2}\left(0,T ; L^{2}\right) \rightarrow L^{2}\left(0,T ; L^{2}\right)$. 
Therefore, $\mathcal{S}$ is completely continuous if we prove its continuity.  
We show this in the following lemma.
\begin{lemma}\label{lem-contin} The mapping $\mathcal{S}$ is continuous  from $L^{2}\left(0,T ; L^{2}\right)$ into $L^{2}\left(0,T ; L^{2}\right)$.
\end{lemma}
\begin{proof} 
	Let $\overline {\theta}$, $\overline {\theta}_{n} \in L^{2}\left(0,T ; L^{2}\right)$, $\varphi$, $\varphi_n\in V_{\varphi}^{1,4} $ and $\boldsymbol{v}, \boldsymbol{v}_{n} \in \boldsymbol{\mathcal{U}}$ with $\boldsymbol{v}_{t},\left(\boldsymbol{v}_{n}\right)_{t} \in L^{4}\left(0,T ; \mathbf{V}_{\boldsymbol{v}}^{-1,4}\right)$ such that
	\begin{eqnarray}
		a_{\varphi}(\overline {\theta} ,\varphi,\chi)&=&\langle g, \chi\rangle_{\Gamma_5}, \label{varphi1}
		\\
		a_{\varphi}(\overline {\theta}_n ,\varphi_n,\chi)&=&\langle g, \chi\rangle_{\Gamma_5}, \label{varphi2}
	\end{eqnarray}
	and 
	$$
	\begin{array}{rl}
		\left\langle \boldsymbol{v}_{t}, \boldsymbol{w}\right\rangle+a_{u}(\overline {\theta} ; \boldsymbol{v}, \boldsymbol{w})+b(\boldsymbol{v}, \boldsymbol{v}, \boldsymbol{w})&=(\boldsymbol{F}(\overline {\theta}), \boldsymbol{w}), \\
		\left\langle\left(\boldsymbol{v}_{n}\right)_{t}, \boldsymbol{w}\right\rangle+a_{u}(\overline {\theta}_{n} ; \boldsymbol{v}_{n}, \boldsymbol{w})+b\left(\boldsymbol{v}_{n}, \boldsymbol{v}_{n}, \boldsymbol{w}\right)&=(\boldsymbol{F}(\overline {\theta}_{n}), \boldsymbol{w}),
	\end{array}
	$$
	for every $\chi \in V_{\varphi}^{1,2}$, $\boldsymbol{w} \in \mathbf{V}_{\boldsymbol{v}}^{1,4 / 3}$ and almost every $t \in (0,T)$, and
	$$
	\boldsymbol{v}(\boldsymbol{x}, 0)=\boldsymbol{v}_{0}(\boldsymbol{x}), \quad \boldsymbol{v}_{n}(\boldsymbol{x}, 0)=\boldsymbol{v}_{0}(\boldsymbol{x}), \quad \text { in } \Omega.
	$$
	Now, we let the difference $\omega_n=\varphi-\varphi_n$. We substracte equations (\ref{varphi1}) and (\ref{varphi2}), to arrive at 
	\begin{equation}\label{eq-lem-ref1}
		a_{\varphi}(\overline {\theta} ,\omega_n,\chi)=\int_{\Omega}[\lambda(\overline {\theta}_n)-\lambda(\overline {\theta})] \nabla {\varphi}_n \nabla \chi ~d{\mathbf{x}}
	\end{equation}
	Next, we substitute $\chi=\omega_n$ in \eqref{eq-lem-ref1} to obtain
	\begin{equation}\label{Normnablaomega}
		\lambda_1||\nabla \omega_n||_{\boldsymbol{L}^2}\leq ||[\lambda(\overline {\theta}_n)-\lambda(\overline {\theta})] \nabla {\varphi}_n||_{\boldsymbol{L}^2}.
	\end{equation}
	According to the Poincaré inequality , there is exists a constant $c>0$ such that, 
	\begin{equation}\label{Normomega}
		||\omega_n||_{L^{2}}\leq c ||[\lambda(\overline {\theta}_n)-\lambda(\overline {\theta})] \nabla {\varphi}_n||_{\boldsymbol{L}^2}.
	\end{equation}
	In the following step, we let $\theta, \theta_{n} \in L^{2}\left(0,T ; V_{\theta}^{1,2}\right)$ with $\theta_{t},\partial_{t}\theta_{n} \in L^{2}\left(0,T ; V_{\theta}^{-1,2}\right)$ be solutions of the equations
	$$
	\begin{array}{rl}
		\left\langle\theta_{t}, \psi\right\rangle+a_{\theta}(\overline {\theta} ; \theta, \psi)+d(\boldsymbol{v}, \theta, \psi)+\alpha(\theta, \psi)_\Gamma &=e(\overline {\theta} ; \boldsymbol{v}, \boldsymbol{v}, \psi)+c_{\varphi}(\overline {\theta} ,{\varphi},\psi )+\alpha\langle\theta_{l},\psi\rangle_{\Gamma},
		\\
		\left\langle\partial_t\sigma_{n}, \psi\right\rangle+a_{\theta}(\overline {\theta}_n ; \theta_n, \psi)+d(\boldsymbol{v}_n, \theta_n, \psi)+\alpha(\theta_n, \psi)_\Gamma &=e(\overline {\theta}_n ; \boldsymbol{v}_n, \boldsymbol{v}_n, \psi)+c_{\varphi}(\overline {\theta}_n ,{\varphi_n},\psi ) +\alpha\langle\theta_{l}, \psi\rangle_{\Gamma},
	\end{array} 
	$$
	for every $\psi \in V_{\theta}^{1,2}$ and almost every $t \in (0,T)$, and
	$$
	\theta(x, 0)=\theta_{0}(x), \quad \theta_{n}(x, 0)=\theta_{0}^n(x) \quad \text { in } \Omega.
	$$
	Denote the differences $\sigma_{n}=\theta-\theta_{n}$ and $\boldsymbol{z}_{n}=\boldsymbol{v}-\boldsymbol{v}_{n}$. Then, for everv $\psi \in V_{\theta}^{1,2}$ and almost every $t \in (0,T)$ we have
	$$
	\begin{array}{rl}
		\displaystyle \left\langle\partial_t\sigma_{n}, \psi\right\rangle+a_{\theta}(\overline {\theta} ; \sigma_{n}, \psi) 
		\displaystyle& =\displaystyle -\alpha(\sigma_{n}, \psi)_\Gamma -\int_{\Omega}\left[\gamma(\overline {\theta})-\gamma(\overline {\theta}_{n})\right] \nabla \theta_{n} \cdot \nabla \psi ~d{\mathbf{x}}-d(\boldsymbol{z}_{n}, \theta, \psi)-d(\boldsymbol{v}_{n}, \sigma_{n}, \psi) \\   
		&\displaystyle \qquad  +c_{\varphi}(\overline {\theta} ,{\varphi},\psi )-c_{\varphi}(\overline {\theta}_n ,{\varphi}_n,\psi )+e(\overline {\theta} ; \boldsymbol{v}, \boldsymbol{v}, \psi)-e(\overline {\theta}_{n} ; \boldsymbol{v}_{n}, \boldsymbol{v}_{n}, \psi).
	\end{array}
	$$
	Set $\psi=\sigma_{n}$ to get the estimates for terms on the right-hand side in previous equation, 
	$$
	\begin{aligned}
		K_1-K_2 &=\left|\int_{\Omega} \left( \lambda(\overline {\theta}) \nabla \varphi \cdot \nabla \varphi-\lambda(\overline {\theta}_n) \nabla \varphi_n \cdot \nabla \varphi_n\right) \sigma_n     ~d{\mathbf{x}}\right|\\
		&\leq \left| \int_{\Omega}\left( \lambda(\overline {\theta}) \nabla \varphi \cdot \nabla \varphi-\lambda(\overline {\theta}_n) \nabla \varphi \cdot \nabla \varphi_n \right)\sigma_n ~d{\mathbf{x}}\right|+\left| \int_{\Omega}\left( \lambda(\overline {\theta}_n) \nabla \varphi \cdot \nabla \varphi_n-\lambda(\overline {\theta}_n) \nabla \varphi_n \cdot \nabla \varphi_n\right) \sigma_n     ~d{\mathbf{x}} \right|\\
		&\leq  ||\nabla \varphi||_{\boldsymbol{L^2}}||\lambda(\overline {\theta})\nabla \varphi-\lambda(\overline {\theta}_n)\nabla \varphi_n||_{\boldsymbol{L^2}}||\sigma_n||_{L^2}+||\nabla \varphi_n||_{\boldsymbol{L^2}}||\lambda(\overline {\theta}_n)\nabla \varphi-\lambda(\overline {\theta}_n)\nabla \varphi_n||_{\boldsymbol{L^2}}||\sigma_n||_{L^2}\\
		&\leq \delta||\sigma_n||^2_{W^{1,2}}+C(\delta)\left(||\nabla \varphi||^2_{\boldsymbol{L^2}}||\lambda(\overline {\theta})\nabla \varphi-\lambda(\overline {\theta}_n)\nabla \varphi_n||_{\boldsymbol{L^2}}^2+||\nabla \varphi_n||_{\boldsymbol{L^2}}^2||\lambda(\overline {\theta}_n)\nabla \varphi-\lambda(\overline {\theta}_n)\nabla \varphi_n||^2_{\boldsymbol{L^2}}\right)\\
		&\leq \delta||\sigma_n||^2_{W^{1,2}}+C(\delta)\left[ ||\nabla \varphi||^2_{\boldsymbol{L^2}}\left(\lambda_2||\nabla \omega_n||_{\boldsymbol{L^2}}+||[\lambda(\overline {\theta})-\lambda(\overline {\theta}_n)]\nabla \varphi_n||_{\boldsymbol{L^2}}\right)^2+||\nabla \varphi_n||^2_{\boldsymbol{L^2}}\lambda_2||\nabla \omega_n||^2_{\boldsymbol{L^2}}\right],
	\end{aligned}
	$$ where $K_1-K_2=\left|c_{\varphi}(\overline {\theta} ,{\varphi},\sigma_n )-c_{\varphi}(\overline {\theta}_n ,{\varphi}_n,\sigma_n )\right|$, we keep the estimates:  
	\begin{equation}\label{equ48}
		\left|\int_{\Omega}\left[\gamma(\overline {\theta})-\gamma(\overline {\theta}_{n})\right] \nabla \theta_{n} \cdot \nabla \sigma_{n} ~d{\mathbf{x}}\right| \leq \delta\left\|\sigma_{n}\right\|_{W^{1,2}}^{2}+C(\delta)\left\|\left(\gamma(\overline {\theta})-\gamma(\overline {\theta}_{n})\right) \nabla \theta_{n}\right\|_{\boldsymbol{L^2}}^{2},
	\end{equation}
	and
	\begin{equation}\label{equ49}
		\begin{aligned}
			\left|d\left(\boldsymbol{z}_n, \theta, \sigma_{n}\right)\right| & \leq\|\boldsymbol{z}_n\|_{\mathbf{L}^{4}}\left\|\nabla \theta\right\|_{\mathbf{L}^{2}}\left\|\sigma_{n}\right\|_{L^{4}} \\
			& \leq c\|\boldsymbol{z}_n\|_{\mathbf{L}^{4}}\left\|\theta\right\|_{W^{1,2}}^{}\left\|\sigma_{n}\right\|_{W^{1,2}}^{} \\
			& \leq \delta\left\|\sigma_{n}(t)\right\|_{W^{1,2}}^{2}+C(\delta)\|\boldsymbol{z}_n\|_{\mathbf{L}^{4}}^{2}\left\|\theta\right\|_{W^{1,2}}^{2} .
		\end{aligned}
	\end{equation}
	Furthermore
	\begin{equation}\label{equ50}
		\begin{aligned}
			\left|d(\boldsymbol{v}_{n}, \sigma_{n}, \sigma_{n})\right| & \leq \delta\left\|\sigma_{n}\right\|_{W^{1,2}}^{2}+C(\delta)\left\|\boldsymbol{v}_{n}\right\|_{\boldsymbol{L^4}}^{2/(1-\zeta)}\left\|\sigma_{n}\right\|_{L^{2}}^{2},
		\end{aligned}
	\end{equation} 
	\begin{equation}\label{equ51}
		\begin{aligned}
			\left|e(\overline {\theta} ; \boldsymbol{v}, \boldsymbol{v}, \sigma_{n})-e(\overline {\theta}_{n} ; \boldsymbol{v}_{n}, \boldsymbol{v}_{n}, \sigma_{n})\right|\leq \left|e(\overline {\theta} ; \boldsymbol{v}, \boldsymbol{v}, \sigma_{n})-e(\overline {\theta}_{n} ; \boldsymbol{v}, \boldsymbol{v}, \sigma_{n})\right|+\left|e(\overline {\theta}_{n} ; \boldsymbol{v}+\boldsymbol{v}_{n}, \boldsymbol{z}_{n}, \sigma_{n})\right|.
		\end{aligned}
	\end{equation} 
	The first  term in $(\ref{equ51})$, can be estimated by
	$$
	\left|e(\overline {\theta} ; \boldsymbol{v}, \boldsymbol{v}, \sigma_{n})-e(\overline {\theta}_{n} ; \boldsymbol{v}, \boldsymbol{v}, \sigma_{n})\right| \leq \delta\left\|\sigma_{n}\right\|_{W^{1,2}}^{2}+C(\delta)\left\|\left[\nu(\overline {\theta})-\nu(\overline {\theta}_{n})\right] \mathbb{D}(\boldsymbol{v}): \mathbb{D}(\boldsymbol{v})\right\|_{L^{2}}^{2}
	$$
	and
	$$  
	\begin{aligned}
		\left|e(\overline {\theta}_{n} ; \boldsymbol{v}+\boldsymbol{v}_{n}, \boldsymbol{z}_{n}, \sigma_{n})\right| & \leq c \nu_2 \left\|\boldsymbol{v}_{n}+\boldsymbol{v}\right\|_{W^{1,4}}\left\|\boldsymbol{z}_{n}\right\|_{\mathbf{w}^{1,2}}\left\|\sigma_{n}\right\|_{L^4}\\
		& \leq  c \nu_2 \left\|\boldsymbol{v}_{n}+\boldsymbol{v}\right\|_{W^{1,4}}\left\|\boldsymbol{z}_{n}\right\|_{\mathbf{w}^{1,2}}\left\|\sigma_{n}\right\|_{W^{1,2}}  \\
		& \leq \delta\left\|\sigma_{n}\right\|_{W^{1,2}}^{2}+C(\delta) \nu_{2}^2\left\|\boldsymbol{v}+\boldsymbol{v}_{n}\right\|_{\mathbf{w}^{1,4}}^{2}\left\|\boldsymbol{z}_{n}\right\|_{\mathbf{w}^{1,2}}^{2}.
	\end{aligned}
	$$
	This implies
	\begin{equation}\label{equ52}
		\begin{aligned}
			\left|e(\overline {\theta} ; \boldsymbol{v}, \boldsymbol{v}, \sigma_{n})-e(\overline {\theta}_{n} ; \boldsymbol{v}_{n}, \boldsymbol{v}_{n}, \sigma_{n})\right|&\leq \delta\left\|\sigma_{n}\right\|_{W^{1,2}}^{2}+C(\delta) \nu_{2}^2\left\|\boldsymbol{v}+\boldsymbol{v}_{n}\right\|_{\mathbf{w}^{1,4}}^{2}\left\|\boldsymbol{z}_{n}\right\|_{\mathbf{w}^{1,2}}^{2}\\ & +
			C(\delta)\left\|\left[\nu(\overline {\theta})-\nu(\overline {\theta}_{n})\right] \mathbb{D}(\boldsymbol{v}): \mathbb{D}(\boldsymbol{v})\right\|_{L^{2}}^{2}.
		\end{aligned}
	\end{equation}
	Choosing $\delta$ sufficiently small we conclude
	\begin{equation}\label{equ53}
		\frac{d}{d t}\left\|\sigma_{n}\right\|_{L^{2}}^{2} \leq \alpha_n(t)\left\|\sigma_{n}\right\|_{L^{2}}^{2}+\beta_n(t),
	\end{equation}
	where    
	\begin{align}\label{beta}	
		\begin{aligned}
			\alpha_n(t)=& C(\delta)\left\|\boldsymbol{v}_{n}\right\|_{\boldsymbol{L}^{4}}^{2/(1-\zeta)},
		\end{aligned}
	\end{align}
	and  
	\begin{equation}\label{omega}
		\begin{aligned}
			\beta_n(t)=& C(\delta)\left\|\left(\gamma(\overline {\theta})-\gamma(\overline {\theta}_{n})\right) \nabla \theta_{n}\right\|_{\boldsymbol{L^2}}^{2}+C(\delta)\left\|\boldsymbol{z}_{n}\right\|_{L^{4}}^{2}\|\theta\|_{W^{1,2}}^{2}+C(\delta)\left\|\left[\nu(\overline {\theta})-\nu(\overline {\theta}_{n})\right] \mathbb{D}(\boldsymbol{v}): \mathbb{D}(\boldsymbol{v})\right\|_{\boldsymbol{L^2}}^{2}\\
			&+C(\delta)\left[||\nabla \varphi||^2_{\boldsymbol{L^2}}\left(\lambda_2||\nabla \omega_n||_{\boldsymbol{L^2}}+||[\lambda(\overline {\theta})-\lambda(\overline {\theta}_n)]\nabla \varphi_n||_{\boldsymbol{L^2}}\right)^2+||\nabla \varphi_n||^2_{\boldsymbol{L^2}}\lambda_2||\nabla \omega_n||^2_{\boldsymbol{L^2}}\right]
			\\
			&+C(\delta) \nu_{2}^2\left\|\boldsymbol{v}+\boldsymbol{v}_{n}\right\|_{\mathbf{w}^{1,4}}^{2}\left\|\boldsymbol{z}_{n}\right\|_{\mathbf{w}^{1,2}}^{2}.
		\end{aligned}
	\end{equation}
	Applying the Gronwall's inequality to the estimate (\ref{equ53}) we arrive at
	\begin{equation}\label{Normesigma}
		\left\|\sigma_{n}(t)\right\|_{L^{2}}^{2} \leq \exp \left(\int_{0}^{t} \alpha_n(s) ~d{s}\right)\left[\sigma_{n}(0)+\int_{0}^{t} \beta_n(s) ~d{s}\right],
	\end{equation}
	for all $0 \leq t \leq T$. From the estimates
	$(\ref{equ36})-(\ref{equ37})$  we deduce that there exists some positive constant $C$, independent of $\theta_{n}$ and $\overline {\theta}_{n}$ such that
	$$\left\|\theta_{n}\right\|_{L^{2}\left(0,T ; W^{1,2}\right)} \leq C\text{.}$$
	
	Recall that $\boldsymbol{z}_{n} \rightarrow \boldsymbol{0}$ in $\boldsymbol{\mathcal{U}}$ for $\overline {\theta}_{n} \rightarrow \overline {\theta}$ in $L^{2}\left(0,T ; L^{2}\right)$ (for the proof see \cite{ref1}). Moreover, by (\ref{Normnablaomega}) and
	(\ref{Normomega}), we conclude $\nabla \omega_{n} \rightarrow \boldsymbol{0}$ and $\omega_{n} \rightarrow 0$ in $L^{2}(0,T;\boldsymbol{L^2})$ and $L^{2}(0,T;V_{\varphi}^{1,2})$  respectively, for $\overline {\theta}_{n} \rightarrow \overline {\theta}$ in $L^{2}\left(0,T ; L^{2}\right)$. Hence, all terms on the right-hand side of (\ref{omega}) tend to zero. Since $\sigma_{n}(x, 0)\rightarrow 0$, from (\ref{Normesigma}) we deduce that $\sigma_{n} \rightarrow 0$ in $C\left(0,T ; L^{2}\right),$ which obviously yields the convergence in $L^{2}\left(0,T ; L^{2}\right)$,  too. This achieves the proof.
\end{proof}

\subsection{Existence of the solution to the problem $(\ref{eqvar1})-(\ref{eqvar5})$} We conclude the proof by deriving some estimates of $\boldsymbol{v}$, $\varphi$ and $\theta$. Let $\overline {\theta} \in L^{2}\left(0,T ; L^{2}\right)$. By Theorem \ref{th-Stokes} there exists the unique solution $\boldsymbol{v} \in B$ of the problem $(\ref{DecoupledNS})$. Moreover, by Theorem \ref{varfi} there exists the unique solution $\varphi$ and it is bounded in $V_{\varphi}^{}$ (see Eq. (\ref{normfi})). Furthermore, let $\theta$ be the uniquely determined solution of the problem $(\ref{DecoupledH})$, which is ensured by Theorem \ref{wellposdecoupledH}.
Hence, by the a priori estimate $(\ref{normtheta})$, $\theta=\mathcal{S}(\overline {\theta})$ is bounded in $C\left(0,T ; L^{2}\right)$ independently of $\overline {\theta}$. Consequently, there exists a fixed ball $M \subset L^{2}\left(0,T ; L^{2}\right)$ defined by
\begin{equation}\label{ballB}
	M:=\left\{\theta \in L^{2}\left(0,T ; L^{2}\right) ,\|\theta\|_{L^{2}\left(0,T ; L^{2}\right)} \leq R\right\}
\end{equation}
($R > 0$ sufficiently large)  such that $\mathcal{S}(M) \subset M,$ where the operator $\mathcal{S}: L^{2}\left(0,T ; L^{2}\right) \rightarrow L^{2}\left(0,T ; L^{2}\right)$ is completely continuous, which is ensured by Lemma \ref{lem-contin}. The existence of the solution of the problem $(\ref{eqvar1})-(\ref{eqvar5})$ follows from the Schauder fixed point theorem.

\subsection{Proof of uniqueness}\label{section3.4} 
In this section, under additional assumptions on the problem data (see Theorems \ref{Main_Result} item 2 ), we prove the uniqueness of the solution. 

For this, suppose that there are two solutions $[\boldsymbol{v}_1,  \theta_1, \varphi_1]$ and $[\boldsymbol{v}_2, \theta_2, \varphi_2]$ of the problem $(\ref{eqvar1})-(\ref{eqvar3})$. Denote $\boldsymbol{v} = \boldsymbol{v}_1 - \boldsymbol{v}_2$,  $\theta= \theta_1 - \theta_2$ and
$\varphi=\varphi_1-\varphi_2$. Then $\boldsymbol{v}$, $\theta$ and  $\varphi$ satisfy the following equations
\begin{eqnarray}
	\left\langle\boldsymbol{v}_{t}, \boldsymbol{w}\right\rangle+a_{u}({\theta}_1 ; \boldsymbol{v}, \boldsymbol{w})+\int_{\Omega}\left[\nu({\theta}_1)-\nu\left({\theta}_{2}\right)\right] \mathbb{D}(\boldsymbol{v}_{2}): \mathbb{D}(\boldsymbol{w}) ~d{\mathbf{x}}+b\left(\boldsymbol{v}, \boldsymbol{v}_{2}, \boldsymbol{w}\right)+b\left(\boldsymbol{v}_1, \boldsymbol{v}, \boldsymbol{w}\right)&
		\nonumber\\
	-(\boldsymbol{F}(\theta_1)-\boldsymbol{F}(\theta_2), \boldsymbol{w})	&=0 \label{uni1} \\  
	\left\langle\theta_t, \psi\right\rangle+a_{\theta}({\theta}_1 ; \theta, \psi)+d(\boldsymbol{v}, \theta_1, \psi)+d(\boldsymbol{v}_{2}, \theta, \psi)+\alpha(\theta, \psi)_\Gamma+\int_{\Omega}\left[\gamma( {\theta}_1)-\gamma( {\theta}_{2})\right] \nabla \theta_{2} \cdot \nabla \psi ~d{\mathbf{x}} \nonumber\\
	+	c_{\varphi}({\theta}_1 ,{\varphi}_1,\psi )-c_{\varphi}( {\theta}_2 ,{\varphi}_2,\psi )+e( {\theta}_1 ; \boldsymbol{v}_1, \boldsymbol{v}_1, \psi)-e\left( {\theta}_{2} ; \boldsymbol{v}_{2}, \boldsymbol{v}_{2}, \psi\right) &=0\label{uni2} \\ 
	a_{\varphi}( {\theta}_1 ,\varphi,\chi)-\int_{\Omega}[\lambda( {\theta}_2)-\lambda( {\theta}_1)] \nabla {\theta}_2 \nabla \chi ~d{\mathbf{x}}&= 0 \label{uni3}
\end{eqnarray} 
for every $(\boldsymbol{w},\psi ,\chi) \in V^{1,4 / 3}\times V_{\theta}^{1,2}\times V_{\varphi}^{1,2}$and almost every $t \in (0,T)$, and
$$\begin{aligned}
	\boldsymbol{v}(\boldsymbol{x}, 0)&=\mathbf{0}& \text{ in } & \Omega,\\
	\theta(\boldsymbol{x}, 0)&=\mathbf{0}& \text{ in } & \Omega.
\end{aligned}
$$
Now, we use $\boldsymbol{w} = \boldsymbol{v}(t)$ as a test function in $(\ref{uni1})$ to obtain the
following inequality
\begin{equation}\label{testU}
	\begin{aligned}
		\frac{1}{2} \frac{\mathrm{d}}{\mathrm{d} t}\|\boldsymbol{v}(t)\|_{\mathbf{V}_{\boldsymbol{v}}^{0,2}}^{2}+\nu_{1}\|\boldsymbol{v}(t)\|_{\mathbf{V}_{\boldsymbol{v}}^{1,2}}^{2} & \leq \left|b\left(\boldsymbol{v}_{1}(t), \boldsymbol{v}(t), \boldsymbol{v}(t)\right)\right|+\left|b\left(\boldsymbol{v}(t), \boldsymbol{v}_{2}(t), \boldsymbol{v}(t)\right)\right|+|(\boldsymbol{F}(\theta_1)-\boldsymbol{F}(\theta_2), \boldsymbol{v}(t))|
		\\
		& \hspace{2  cm}
		+\left|\int_{\Omega}(\nu({\theta}_1)-\nu({\theta}_{2})) \mathbb{D}(\boldsymbol{v}_{2}): \mathbb{D}(\boldsymbol{v})~d{\mathbf{x}}\right|.
	\end{aligned} 
\end{equation}
To estimate term by term on the right-hand side of \((\ref{testU})\), we use the Gagliardo--Nirenberg inequality (cf.\cite[Theorem 5.8]{Adams2003})
$$
\|\boldsymbol{v}(t)\|_{\mathbf{L}^{4}} \leq c\|\boldsymbol{v}(t)\|_{\mathbf{V}_{\boldsymbol{v}}^{1,2}}^{\zeta}\|\boldsymbol{v}(t)\|_{\mathbf{L}^{2}}^{1-\zeta}, \hbox{ where }\zeta=d/4 ,
$$
the Young's inequality with parameter $\delta$ and the Lipschitz continuity of $\boldsymbol{F}$ and $\nu$.
\\
The first two terms can be estimate by   
\begin{equation}\label{estem-b} 
	\left|b\left(\boldsymbol{v}_{1}(t), \boldsymbol{v}(t), \boldsymbol{v}(t)\right)\right|+\left|b\left(\boldsymbol{v}(t), \boldsymbol{v}_{2}(t), \boldsymbol{v}(t)\right)\right| 
	\leq \delta\|\boldsymbol{v}(t)\|_{\mathbf{V}_{\boldsymbol{v}}^{1,2}}^{2}+c_{\delta}\|\boldsymbol{v}(t)\|_{\mathbf{L}^{2}}^{2}\left(\left\|\boldsymbol{v}_{1}(t)\right\|_{\mathbf{L}^{4}}^{\frac{2}{1-\zeta}}+\left\|\boldsymbol{v}_{2}(t)\right\|_{\mathbf{w}^{1,2}}^{\frac{1}{1-\zeta}}\right),
\end{equation}
where we have used the inequality \((\ref{b+bestimat})\). 
In addition, the third term can estimate using  Young inequality (immediately
after we apply H{\"o}lder's inequality) and  Lipschitz continuity of $\boldsymbol{F}$. The result is
\begin{equation}\label{estmF}
	\begin{aligned}
		|\left(\boldsymbol{F}(\theta_1)-\boldsymbol{F}(\theta_2), \boldsymbol{v}(t)\right)| 
		& \leq \|\boldsymbol{F}({\theta}_1)-\boldsymbol{F}\left({\theta}_{2}\right)\|_{\mathbf{L}^{2}} \| \boldsymbol{v}(t)\|_{\mathbf{L}^{2}}\\
		& \leq L_{\boldsymbol{F}} \|\theta\|_{L^2} \| \boldsymbol{v}(t)\|_{\mathbf{L}^{2}}\\
		& \leq 1/2L_{\boldsymbol{F}}\left( \|\theta\|_{L^2}^2+\| \boldsymbol{v}(t)\|_{\mathbf{L}^{2}}^2\right)\\
		& \leq c(\|\theta\|_{L^2}^2+\| \boldsymbol{v}(t)\|_{\mathbf{L}^{2}}^2).
	\end{aligned}
\end{equation}
Similarly to \eqref{estmF}, for the last term in (\ref{testU}) we get
\begin{equation}\label{estmD}
	\begin{aligned}
		\left|\int_{\Omega}\left(\nu({\theta}_1)-\nu({\theta}_{2})\right) \mathbb{D}(\boldsymbol{v}_{2}): \mathbb{D}(\boldsymbol{v})~d{\mathbf{x}}\right|
		& \leq \|(\nu({\theta}_1)-\nu({\theta}_{2}))\|_{L^4}
		\| \mathbb{D}(\boldsymbol{v}_{2})\|_{\mathbf{L}^{4}} \|\mathbb{D}(\boldsymbol{v})\|_{\mathbf{L}^{2}} \\
		& \leq L_{\nu}\|\theta\|_{L^4}
		\| \mathbb{D}\left(\boldsymbol{v}_{2}\right)\|_{\mathbf{L}^{4}} \|\boldsymbol{v}(t)\|_{\boldsymbol{W}{1,2}}\\
		& \leq c \|\theta\|_{L^2}^{1-\zeta}\|\theta\|_{W^{1,2}}^{\zeta} \| \mathbb{D}(\boldsymbol{v}_{2})\|_{\mathbf{L}^{4}} \|\boldsymbol{v}(t)\|_{\boldsymbol{W}{1,2}}
		\\
		& \leq \delta\left( \|\boldsymbol{v}(t)\|_{\boldsymbol{W}{1,2}}^2+\|\theta\|_{W^{1,2}}^{2}\right)+C_\delta\|\theta\|_{L^2}^2
		\| \mathbb{D}(\boldsymbol{v}_{2})\|_{\mathbf{L}^{4}}^{\frac{2}{1-\zeta}}.
	\end{aligned}
\end{equation}
Consequently, the estimates $(\ref{estem-b})-(\ref{estmD})$ imply 
\begin{equation}\label{estmU}
	\begin{aligned}
		\frac{\mathrm{d}}{\mathrm{d} t}\|\boldsymbol{v}(t)\|_{\mathbf{V}_{\boldsymbol{v}}^{0,2}}^{2}+c\|\boldsymbol{v}(t)\|_{\mathbf{V}_{\boldsymbol{v}}^{1,2}}^{2}  & \leq \delta\left( \|\boldsymbol{v}(t)\|_{\boldsymbol{W}{1,2}}^2+\|\theta\|_{W^{1,2}}^{2}\right)+
		C_{\delta}R_1(t)\left(\|\boldsymbol{v}(t)\|_{\mathbf{L}^{2}}^{2}+\|\theta\|_{L^2}^2\right),
	\end{aligned} 
\end{equation}
where $R_1(t)=\left(\left\|\boldsymbol{v}_{1}(t)\right\|_{\mathbf{L}^{4}}^{\frac{2}{1-\zeta}}+\left\|\boldsymbol{v}_{2}(t)\right\|_{\mathbf{w}^{1,2}}^{\frac{1}{1-\zeta}}+\| \mathbb{D}(\boldsymbol{v}_{2})\|_{\mathbf{L}^{4}}^2+
1\right)$.\\ 
Now, we substitute $\psi=\theta$ in (\ref{uni2}), to obtain the
following inequality   
$$
\begin{array}{l}
	\displaystyle \left\langle\theta_t, \theta\right\rangle+a_{\theta}\left({\theta}_1 ; \theta, \theta\right) +\alpha(\theta, \theta)_\Gamma \leq \left|\int_{\Omega}\left[\mu( {\theta}_1)-\mu( {\theta}_{2})\right] \nabla \theta_{2} \cdot \nabla \theta ~d{\mathbf{x}}\right|+\left|d\left(\boldsymbol{v}, \theta, \theta\right)\right| +\left|d\left(\boldsymbol{v}_{2}, \theta, \theta\right)\right|\\
	\hspace{5 cm} + \left|c_{\varphi}({\theta}_1 ,{\varphi}_1,\theta )\right| +\left|c_{\varphi}( {\theta}_2 ,{\varphi}_2,\theta )\right|+\left|e( {\theta}_1 ; \boldsymbol{v}_1, \boldsymbol{v}_1, \theta)-e\left( {\theta}_{2} ; \boldsymbol{v}_{2}, \boldsymbol{v}_{2}, \theta\right)\right|.
\end{array}
$$
To get the estimates for terms on the right-hand side in the previous equation we use the Gagliardo--Nirenberg inequality, H{\"o}lder's inequality, and Young inequality. Evidently, we have 
\begin{equation}\label{C1-C2}
	\begin{aligned}
		|c_{\varphi}({\theta}_1 ,{\varphi}_1,\theta )-c_{\varphi}({\theta}_2 ,{\varphi}_2,\theta )|	\leq & \delta\|\theta\|_{W^{1,2}}^2\\
		&+C_{\delta}\left(\|\nabla \varphi_1\|_{\boldsymbol{L}^4}^{\frac{2}{1-\zeta}}+\|\nabla \varphi_1\|_{\boldsymbol{L}^4}^{\frac{2}{1-\zeta}}\| \varphi\|_{{W}^{1,2}}^{\frac{2}{1-\zeta}}+\|\nabla \varphi_2\|_{\boldsymbol{L}^4}^{\frac{2}{1-\zeta}}\| \varphi\|_{{W}^{1,2}}^{\frac{2}{1-\zeta}}\right)\|\theta\|_{L^2}^2.
	\end{aligned}
\end{equation}
We keep the estimates:  
\begin{equation}
	\left|\int_{\Omega}\left[\gamma({\theta}_1)-\gamma({\theta}_{2})\right] \nabla \theta_{2} \cdot \nabla \theta ~d{\mathbf{x}} \right| \leq \delta\left\|\theta\right\|_{W^{1,2}}^{2}+C(\delta)\left\|\nabla\theta_2\right\|_{L^{4}}^{\frac{2}{1-\zeta}}\left\|\theta\right\|_{L^{2}}^{2},
\end{equation}
and
\begin{equation}
	\begin{aligned}
		\left|d\left(\boldsymbol{v}, \theta_1, \theta\right)\right| & \leq\|\boldsymbol{v}\|_{\mathbf{L}^{4}}\left\|\nabla \theta_1\right\|_{\mathbf{L}^{2}}\left\|\theta\right\|_{L^{4}} \\
		& \leq c\|\boldsymbol{v}\|_{\mathbf{W}^{1,2}}^{\zeta}\|\boldsymbol{v}\|_{\mathbf{L}^{2}}^{1-\zeta}\left\|\theta_1\right\|_{W^{1,2}}^{}\left\|\theta\right\|_{W^{1,2}}^{\zeta}\left\|\theta\right\|_{L^{2}}^{1-\zeta}\\
		& \leq \delta(\left\|\theta(t)\right\|_{W^{1,2}}^{}\|\boldsymbol{v}\|_{\mathbf{W}^{1,2}}^{})+C(\delta)\left\|\theta_1\right\|_{W^{1,2}}^{\frac{1}{1-\zeta}}\|\boldsymbol{v}\|_{\mathbf{L}^{2}}^{}\left\|\theta\right\|_{L^2}\\
		& \leq\delta/2\left(\left\|\theta(t)\right\|_{W^{1,2}}^{2}+\|\boldsymbol{v}\|_{\mathbf{W}^{1,2}}^{2}\right)+C(\delta)\left\|\theta_1\right\|_{W^{1,2}}^{\frac{1}{1-\zeta}}\left(\|\boldsymbol{v}\|_{\mathbf{L}^{2}}^{2}+\left\|\theta\right\|_{L^2}^2\right).
	\end{aligned}
\end{equation}
Moreover, we obtain
\begin{equation}
	\begin{aligned}
		\left|d\left(\boldsymbol{v}_{2}, \theta, \theta\right)\right| & \leq \delta\left\|\theta\right\|_{W^{1,2}}^{2}+C(\delta)\left\|\boldsymbol{v}_{2}\right\|_{\boldsymbol{L^4}}^{2/(1-\zeta)}\left\|\theta\right\|_{L^{2}}^{2}.
	\end{aligned}
\end{equation} 
The different of dissipatives terms can be estimated by 
$$
\begin{aligned}
	\left|e\left({\theta}_1 ; \boldsymbol{v}_1, \boldsymbol{v}_1, \theta_{}\right)-e\left({\theta}_{2} ; \boldsymbol{v}_{2}, \boldsymbol{v}_{2}, \theta_{}\right)\right|\leq \left|e\left({\theta}_1 ; \boldsymbol{v}_1, \boldsymbol{v}_1, \theta_{}\right)-e\left({\theta}_{2} ; \boldsymbol{v}_1, \boldsymbol{v}_1, \theta_{}\right)\right|+\left|e\left({\theta}_{2} ; \boldsymbol{v}_1+\boldsymbol{v}_{2}, \boldsymbol{v}, \theta_{}\right)\right|.
\end{aligned}
$$ 
The first  terms can be estimated by
\begin{equation}\label{E1}
	\begin{aligned}
		\left|e\left({\theta}_1 ; \boldsymbol{v}_1, \boldsymbol{v}_1, \theta_{}\right)-e\left({\theta}_{2} ; \boldsymbol{v}_1, \boldsymbol{v}_1, \theta_{}\right)\right|
		& \leq  \|\nu({\theta}_1)-\nu({\theta}_{2})\|_{L^4} \|\boldsymbol{v}_1 \|_{\mathbf{w}^{1,4}}^2\left\|\theta_{}\right\|_{L^4}^{}
		\\
		& \leq c L_{\nu}\|\boldsymbol{v}_1 \|_{\mathbf{w}^{1,4}}^2 \left\|\theta_{}\right\|_{L^2}^{1-\zeta}\left\|\theta_{}\right\|_{W^{1,2}}^{\zeta+1}
		\\
		& \leq \delta\left\|\theta_{}\right\|_{W^{1,2}}^{2}+C(\delta)\|\boldsymbol{v}_1 \|_{\mathbf{w}^{1,4}}^{\frac{4}{1+\zeta}}\|\theta\|_{L^2}^2,
	\end{aligned}
\end{equation}
and
\begin{equation}\label{E2}
	\begin{aligned}
		\left|e\left({\theta}_{2} ; \boldsymbol{v}_1+\boldsymbol{v}_{2}, \boldsymbol{v}_{}, \theta_{}\right)\right| & \leq c \nu_2 \left\|\boldsymbol{v}_{2}+\boldsymbol{v}_1\right\|_{W^{1,4}}\left\|\boldsymbol{v}_{}\right\|_{\mathbf{w}^{1,2}}\left\|\theta_{}\right\|_{L^4}\\
		& \leq c \nu_2 \left\|\boldsymbol{v}_{2}+\boldsymbol{v}_1\right\|_{W^{1,4}}\left\|\boldsymbol{v}_{}\right\|_{\mathbf{w}^{1,2}}\left\|\theta_{}\right\|_{L^2}^{1-\zeta}\left\|\theta_{}\right\|_{W^{1,2}}^{\zeta}\\
		& \leq \delta\left\|\theta_{}\right\|_{W^{1,2}}^{\frac{2\zeta}{1+\zeta}}\left\|\boldsymbol{v}_{}\right\|_{\mathbf{w}^{1,2}}^{\frac{2}{1+\zeta}}+C(\delta)\left\|\boldsymbol{v}_1+\boldsymbol{v}_{2}\right\|_{\mathbf{w}^{1,4}}^{\frac{2}{1-\zeta}}\left\|\theta_{}\right\|_{L^2}^2.
	\end{aligned}
\end{equation}
Collecting the previous results \eqref{C1-C2}-\eqref{E2}, we deduce that
\begin{equation}\label{estmtheta}
	\frac{d}{d t}\left\|\theta\right\|_{L^{2}}^{2}+c\left(\left\|\theta(t)\right\|_{W^{1,2}}^{2}+\|\boldsymbol{v}\|_{\mathbf{W}^{1,2}}^{2}\right) \leq \delta\left(\left\|\theta(t)\right\|_{W^{1,2}}^{2}+\|\boldsymbol{v}\|_{\mathbf{W}^{1,2}}^{2}\right)+C_{\delta} R_2(t)\left(\|\theta\|_{L^{2}}^{2}+\|\boldsymbol{v}\|_{\mathbf{L}^{2}}^{2}\right)
\end{equation}
where $R_2(t)=\left(\|\nabla \varphi_1\|_{\boldsymbol{L}^4}^{\frac{2}{1-\zeta}}+\|\nabla \varphi_1\|_{\boldsymbol{L}^4}^{\frac{2}{1-\zeta}}\| \varphi\|_{{W}^{1,2}}^{\frac{2}{1-\zeta}}+\|\nabla \varphi_2\|_{\boldsymbol{L}^4}^{\frac{2}{1-\zeta}}\| \varphi\|_{{W}^{1,2}}^{\frac{2}{1-\zeta}}+\left\|\theta_1\right\|_{W^{1,2}}^{\frac{1}{1-\zeta}}+\left\|\boldsymbol{v}_1+\boldsymbol{v}_{2}\right\|_{\mathbf{w}^{1,4}}^{\frac{2}{1-\zeta}}+\|\boldsymbol{v}_1 \|_{\mathbf{w}^{1,4}}^{\frac{4}{1+\zeta}}\right. \\
\left.\qquad \qquad\qquad
\qquad+\left\|\boldsymbol{v}_{2}\right\|_{\boldsymbol{L^4}}^{2/(1-\zeta)}+\left\|\nabla\theta_2\right\|_{L^{4}}^{\frac{2}{1-\zeta}}\right).$
\\
We make the sum of $(\ref{estmU})$ and $(\ref{estmtheta})$, and we use $\delta$ small to find 
\begin{equation}\label{estm-U-Theta}
	\begin{aligned}
		\frac{d}{d t}\left(\left\|\theta\right\|_{L^{2}}^{2}+\|\boldsymbol{v}\|_{\mathbf{L}^{2}}^{2} \right) & \leq C_{\delta}'\left(R_1(t)+R_2(t)\right)\left(\left\|\theta\right\|_{L^{2}}^{2}+\|\boldsymbol{v}\|_{\mathbf{L}^{2}}^{2} \right).
	\end{aligned} 
\end{equation}
Applying the Gronwall's inequality to $(\ref{estm-U-Theta})$ and the fact that $\boldsymbol{v}(x, 0)=\theta(x, 0)=0$,  we arrive at
$\theta=\boldsymbol{v}=0$.  

Now, we use substitute $\chi=\varphi$ in (\ref{uni3}) to get   
\begin{equation}\label{}
	c_1||\nabla \varphi||_{\boldsymbol{L}^2}^2\leq ||[\lambda( {\theta}_2)-\lambda( {\theta}_1)] \nabla {\varphi}_2||_{\boldsymbol{L}^2}||\nabla \varphi||_{\boldsymbol{L}^2}.
\end{equation}
Using the Lipschitz condition of $\lambda$ and according to the inequality of Poincaré and Young, there is a constant $c>0$ such that, 
\begin{equation}\label{}
	\begin{aligned}
		{\| \varphi\|}_{W^{1,2}} & \leq c\|[\lambda( {\theta}_2)-\lambda( {\theta}_1)] \nabla \varphi_2\|_{\boldsymbol{L}^2}^2\\
		& \leq c \|\theta\|_{L^4}^2\|\nabla \varphi_2\|_{\boldsymbol{L}^4}^2\\
		& \leq c \|\theta\|_{L^2}^{2(1-\zeta)}\|\theta\|_{W^{1,2}}^{2\zeta}\|\nabla \varphi_2\|_{\boldsymbol{L}^4}^2.
	\end{aligned}
	\label{estomega}
\end{equation}
Finally, since $\theta=0$, we conclude that $\varphi=0$.

\section{Numerical experiments}\label{Section5}  
	In this section, we aim to validate the proposed model. We start by presenting the computational analysis to solve numerically the proposed model. Subsequently, we provide some examples demonstrating the impact of the presence of the energy dissipation, the external forces, the saline flow and the cooling factor. To achieve this, we based on existing parameters in the literature, for example, in \cite{materiel2,materiel1,materiel3}. Let us define these parameters: the electrical conductivity $\sigma$, the thermal conductivity $\eta$ and the blood conductivity $\nu$ are depend on a temperature-dependent function and are given by the following equation:
\begin{eqnarray*}
	\sigma(\theta)&=&\left\{\begin{array}{ll}
		\sigma_0 \exp ^{0.015\left( \theta - \theta_b\right)} & \text { for } \theta \leq 99^{\circ} \mathrm{C} \\
		2.5345 \sigma_0 & \text { for } 99^{\circ} \mathrm{C}<\theta \leq 100^{\circ} \mathrm{C} \\
		2.5345 \sigma_0\left(1- 0.198\left(\theta-100^{\circ} \mathrm{C}\right)\right) & \text { for } 100^{\circ} \mathrm{C}< \theta \leq 105^{\circ} \mathrm{C} \\
		0.025345 \sigma_0 & \text { for } \theta >105^{\circ} \mathrm{C}
	\end{array}\right .
	\\
	\eta(\theta) &=&\left\{\begin{array}{ll}
		\eta_0+ 0.0012 \left(\theta -  \theta_b\right) & \text { for } \theta \leq 100^{\circ} \mathrm{C} \\
		\eta_0+0.0012\left(100^{\circ} \mathrm{C}-\theta_b\right) & \text { for } \theta>100^{\circ} \mathrm{C}
	\end{array}\right .
\end{eqnarray*}
where $\sigma_0 = 0.6$ and  $\eta_0 = 0.54$ are the constant electrical conductivity and
the thermal conductivity, respectively, at core body temperature, $\theta_b\left(=37^{\circ} \mathrm{C}\right)$ and $\xi=1$.
The viscosity and density of blood are $0.0021 \mathrm{~Pa} \cdot \mathrm{s}$ and $1000 \mathrm{~kg} / \mathrm{m}^3$, respectively, whereas those of saline are $0.001 \mathrm{~Pa} \cdot \mathrm{s}$ and $1000 \mathrm{~kg} / \mathrm{m}^3$, respectively, based on the material property of water.  
	\subsection{Computation domain and discretization}
{In our numerical study we consider a domain $\Omega$ as illustrated in Figure \ref{fig:omega} and we fix  values  $L= 1.5$ mm, $H= 0.5$ mm and $r= 0.075$ mm. 
We assume that the thickness of the electrode is negligible, and we abound its effect in the numerical simulation.  
}

For the time discretization, fixing an integer  $M$, we define   a time subdivision $t_0= 0 < \cdots < t_M =T$  
and the time steps as $\tau_n = {t_{n+1}-t_n}, i=0, \cdots, M-1$. 
While we use a finite element discretization in space. 
Namely, we exploit the finite element P1-Bubble to compute the values of the velocity variable and the P1 finite element to approximate the temperature, pressure and potential unknowns.
In the sequel, we keep the same notations of the variables $\boldsymbol{v}$, $P$, $\theta$ and $\varphi$ for the discrete versions.
\\
We now deal with the reformulation of the studied model into an algebraic system of differential equations that allows us to use a time lag scheme. 
That is,  given the solution of the heat equation at the previous time, we solve then the decoupled potential and Navier--Stokes equations (\ref{DecoupledNS})-(\ref{DecoupledPhi})  for time step $n-1$ as
\begin{equation}
	\left\{\begin{array}{rclll}
		\boldsymbol{v}_{t}-\operatorname{div}(\nu( {\theta}^{n-1}) \mathbb{D}(\boldsymbol{v}))+\operatorname{div}(\boldsymbol{v} \otimes \boldsymbol{v})+\nabla P & =&\boldsymbol{F}({\theta}^{n-1} ) & \hbox { in } \Omega_{T} \\
		\operatorname{div} \boldsymbol{v} & =&0 &\hbox{ in } \Omega_{T} \\
		\boldsymbol{v} & =&0 &\hbox{ on }  \Sigma_{D}\\
		-P \boldsymbol{n}+\nu({\theta}^{n-1}) \mathbb{D}(\boldsymbol{v}) \boldsymbol{n} & =&\mathbf{0} & \hbox { on } \Sigma_{N }\\   
		\boldsymbol{v}(\boldsymbol{x}, 0) & =&\boldsymbol{v}_{0}(\boldsymbol{x}) &\hbox { in } \hbox { }\Omega 
	\end{array}\right . 
	\hbox{ , } 
	\left\{	\begin{array}{rclll}
		- \operatorname{div}(\sigma({\theta}^{n-1}) \nabla \varphi) & = &0 & \hbox { in } &\Omega_{T} \\
		(\sigma({\theta}^{n-1})\nabla \varphi) \cdot \boldsymbol{n} & = & g & \hbox { on } &\Sigma_{N} \\
		\varphi & = & 0 & \hbox { on } &\Sigma_{D} \\
	\end{array}\right. 
\end{equation}
We get then the potential $\varphi^{n-1}$, the velocity $\boldsymbol{v}^{n-1}$ and the the pressure  $P^{n-1}$ at the time step $n-1$. 
We solve 
then the temperature equation at  time $n$. \\
A more interesting question is how to treat the temperature advection-diffusion equation.
By default, not all discretizations of this equation are equally stable unless we use regularization techniques. 
To achieve this, we can use  discontinuous elements which is more efficient for pure advection problems. 
But in the presence of diffusion terms, the discretization of the Laplace operator  is cumbersome due to the large number of additional terms that must be integrated on each face between the cells.
A better alternative is therefore to add some nonlinear viscosity   $\tilde {\eta}(\theta)$ to the model  that only acts in the vicinity of shocks and other discontinuities. 
$\tilde {\eta}(\theta)$ is chosen in such a way that if $\theta$ satisfies the original equations, the additional viscosity is zero.
To achieve this, the literature contains a number of approaches. 
We will opt here for the stabilization strategy developed by Guermond and Popov \cite{guermond2011entropy} 
that builds on a suitably defined residual and a limiting procedure for the additional viscosity. 
To this end, let us define a residual $R_\alpha(\theta)$ as follows:
$$
R_\alpha(\theta)=\left(\frac{\partial \theta}{\partial t}+\boldsymbol{v} \cdot \nabla \theta-\operatorname{div} \eta(\overline \theta) \nabla \theta- \displaystyle 
\nu( \overline \theta) \mathbb D  (\boldsymbol{v}) : \mathbb D  (\boldsymbol{v})  
-\sigma  ( \overline \theta) |\nabla \varphi|^2 \right) \theta^{\alpha-1}, \quad \alpha \in [1,2].
$$
Note that $R_\alpha(\theta)$ will be zero if $\theta$ satisfies the temperature equation. Multiplying terms out, we get the following, entirely equivalent form:
$$
R_\alpha(\theta)=\frac{1}{\alpha} \frac{\partial\left(\theta^\alpha\right)}{\partial t}+\frac{1}{\alpha} \boldsymbol{v} \cdot \nabla\left(\theta^\alpha\right)-\frac{1}{\alpha} \operatorname{div} \eta(\overline \theta) \nabla\left(\theta^\alpha\right)+\eta(\overline \theta)(\alpha-1) \theta^{\alpha-2}|\nabla \theta|^2-\gamma \theta^{\alpha-1}.
$$
Using the latter, we can    define the artificial viscosity as a piecewise constant function defined on each cell $K$ with diameter $h_K$ separately as follows:
$$
\left.\tilde {\eta}_\alpha(\theta)\right|_K=\beta\|\boldsymbol{v}\|_{L^{\infty}(K)} \min \left\{h_K, h_K^\alpha \frac{\left\|R_\alpha(\theta)\right\|_{L^{\infty}(K)}}{c(\boldsymbol{v}, \theta)}\right\}
$$
where, 
$\beta$ is a stabilization constant and 
$\displaystyle c(\boldsymbol{v}, \theta)=c_R\|\boldsymbol{v}\|_{L^{\infty}(\Omega)} \operatorname{var}(\theta)|\operatorname{diam}(\Omega)|^{\alpha-2}$ where
$\operatorname{var}(\theta)=\max _{\Omega} \theta-\min _{\Omega} \theta$ is the range of present temperature values  and $c_R$ is a dimensionless constant. \\
If on a particular cell the temperature field is smooth, then we expect the residual to be small and the stabilization term that injects the artificial diffusion will be rather small,   when no additional diffusion is needed. 
On the other hand, if we are on or near a discontinuity in the temperature field, then the residual will be large and the artificial viscosity will ensure the stability of the scheme.

\subsection{Validation of model in different cases of data} 

In order to validate our proposed model \eqref{System}, we provides some numerical simulations demonstrating the influence of the energy dissipation, the external forces, the saline flow and cooling factor. {Note that, all computations have been implemented using the software package FreeFem++ \cite{hecht2012new} and plot by the software Matlab. }
\subsubsection{Example 1: the energy dissipation}
This example aims to demonstrate the effect of the energy dissipation. We consider the following configurations.  
We impose a velocity 
${\boldsymbol{v}} = \left(  \begin{array}{l} y(H-y) \\ 0 \end{array}\right)$ on boundary $\Gamma_1$ 
and on boundaries $\Gamma_i, i=2,  4$,  we assume that the velocity is zero. 
While on the 3th boundary $\Gamma_3$, we assume that $-\mathbb{S}(\boldsymbol{v},P)\boldsymbol{n}=\mathbf{0}$. 
Concerning the temperature, on the boundaries $\Gamma_i$,  $i=1, 2,4$, we apply the condition 
$(\eta(\theta) \nabla \theta) \cdot \boldsymbol{n} + \alpha \theta =\alpha \theta_l$, with  $\alpha = 1$ and $\theta_l = \theta_b = 37^{\circ} \mathrm{C}$.
On $\Gamma_3$, we impose an artificial boundary condition, that is the homogeneous Neumann boundary conditions. 
For the potential equation, we fix $g=3$ on $\Gamma_5$ and the homogeneous Dirichlet condition in the remaining boundaries.
In this example, we neglect the second member of the Navier--Stokes equations, so $\mathbf F = 0$.  
The initial conditions for the heat transport equation and the Navier--Stokes system are constructed by solving the associated stationary equation. 
 We notice that the computed potential evolves very slowly during the time iterations, see Figure \ref{fig:test10}. Indeed, the only data in the potential equation is the source $\varphi$ which is constant and the electrical conductivity $\sigma(\theta) = \sigma_0 \exp(0.015(\theta-{\theta}_b))$. Thus, we omit the figures of the potential as there is no significant change during the iterations.
We then focus our reading for this example on the influence of the presence of the energy dissipation term due to viscosities
$\displaystyle \nu(\theta) \mathbb{D}(\boldsymbol{v}) : \mathbb{D}(\boldsymbol{v})$ and 
$\displaystyle \lambda(\theta) \nabla \varphi\cdot \nabla \varphi$. 
As we see in the potential figures (Figure \ref{fig:test10}), with data $g>0$ on $\Gamma_5$ (the head of the electrode), we create a potential with a higher density in a neighborhood of the border $\Gamma_5$. 
In the same neighborhood, a temperature is produced. This shows the impact of the quadratic term $\displaystyle \lambda(\theta) \nabla \varphi\cdot \nabla \varphi$ as an energy source for the heat equation (Figure \ref{fig:test1} - column 2).
However, when there is a blood flow, the temperature produced will be moved to the outlet of the domain. 
This is a consequence of the transport term $\boldsymbol{v} \cdot \nabla T$.
In order to present these evolutions, we show in Figure \ref{fig:test1} the results of numerical simulation at four different times $t=0$, 
$t= \frac{T}{4}$, $t= \frac{T}{2}$,  and $t=T$, 
where each row of the figure represents the corresponding time in the same order. 
In the first column, we show the velocity field and the pressure. In the second column, we show the heat transport.
\begin{figure}[pos=!ht]
	\begin{minipage}{\linewidth} 
		\includegraphics[   width=.5\linewidth]{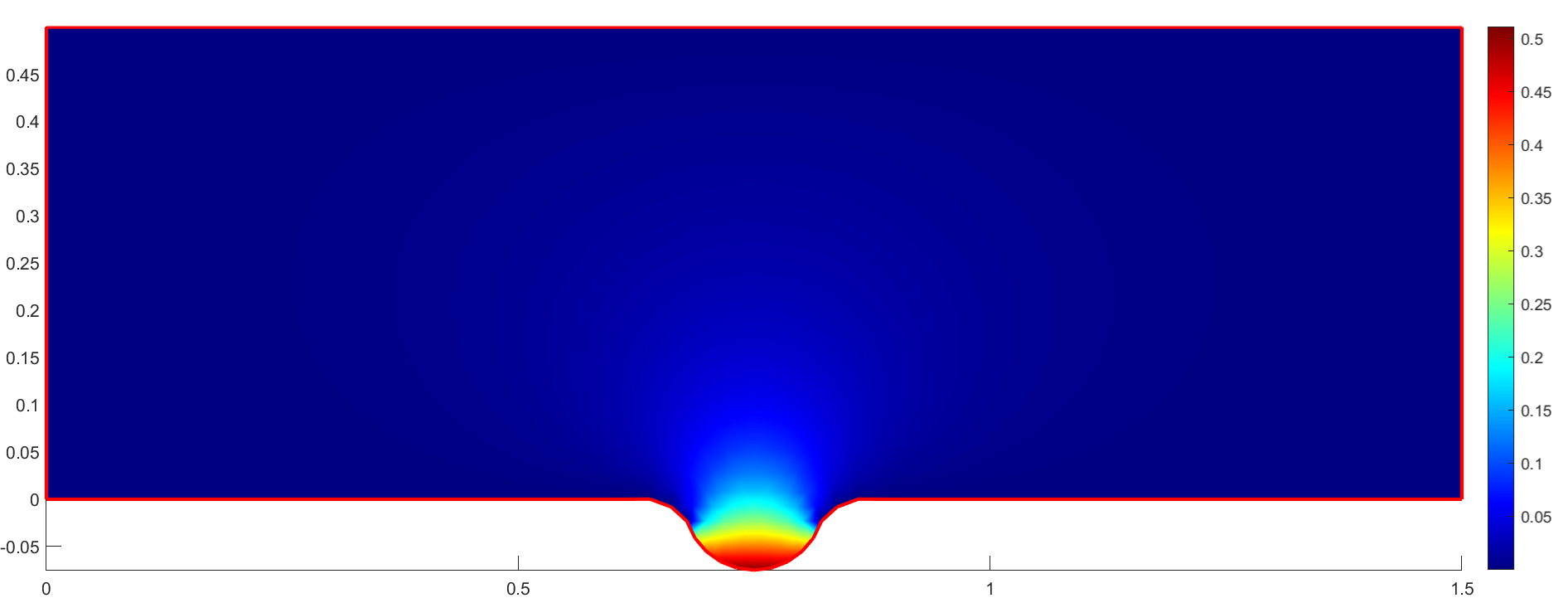} 
		\includegraphics[   width=.5\linewidth]{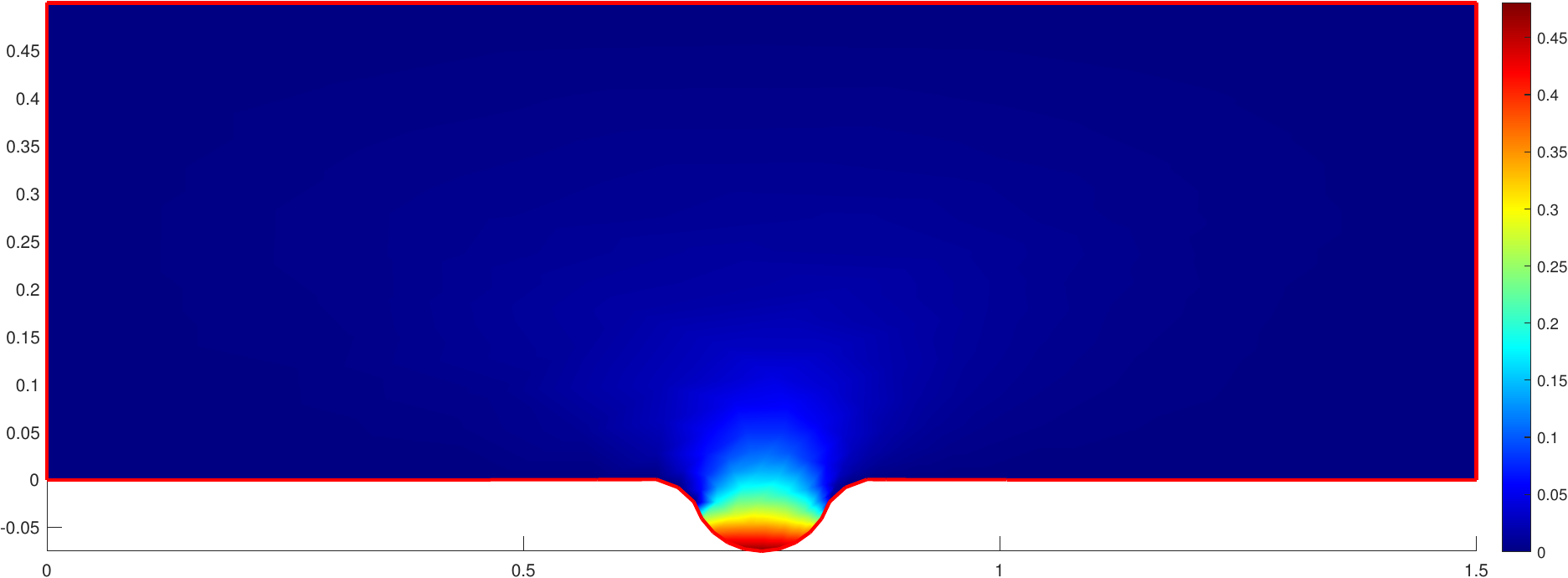} 
	\end{minipage}\\
	\begin{minipage}{\linewidth}  
		\includegraphics[   width=.5\linewidth]{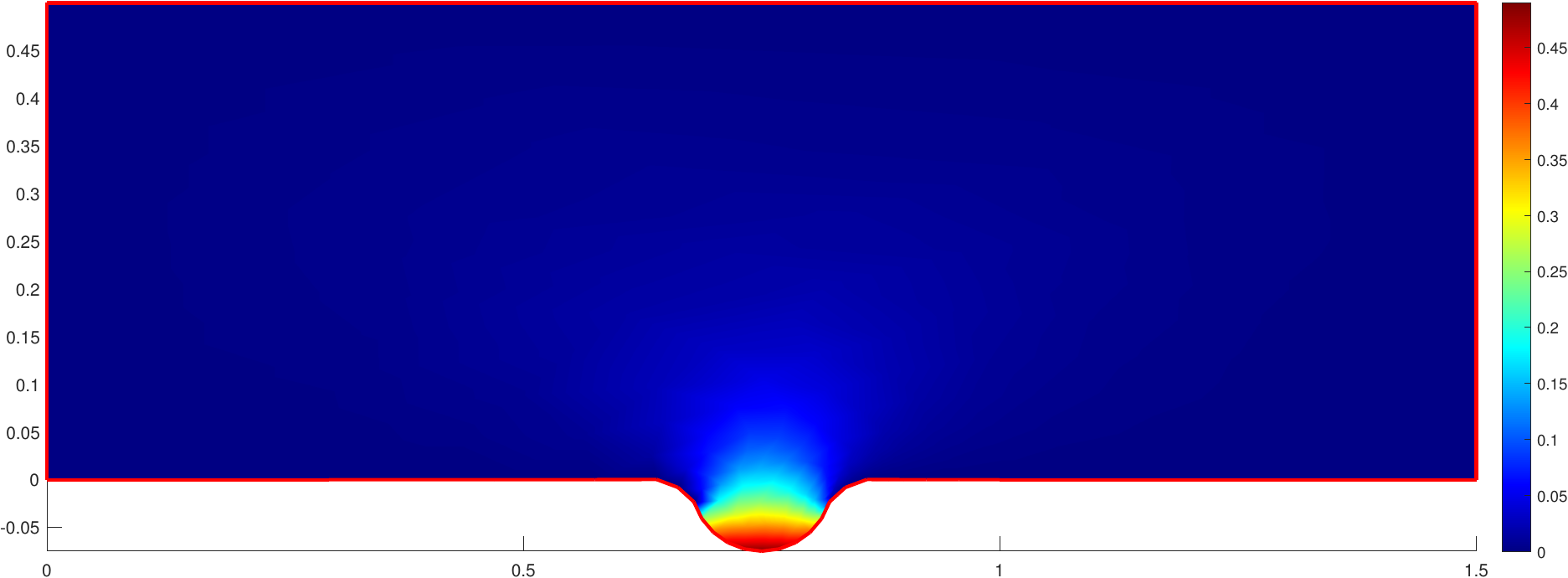}  
			\includegraphics[   width=.5\linewidth]{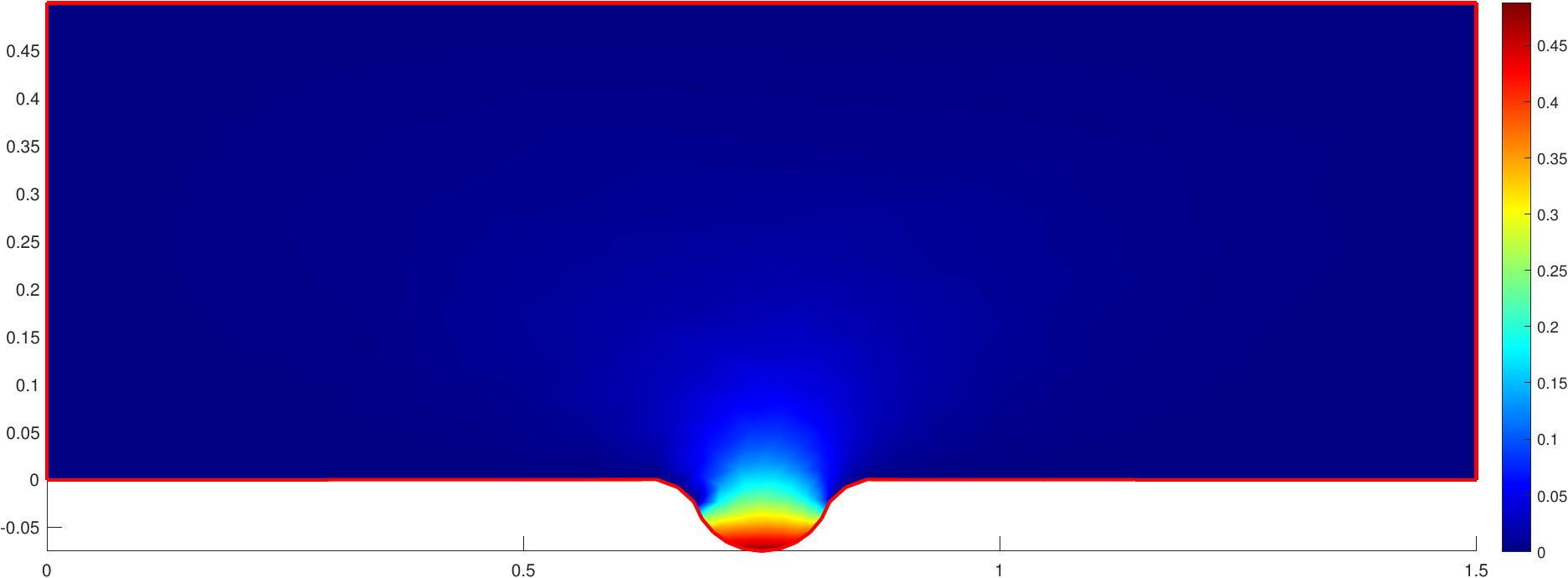}
	\end{minipage} 
	\caption{Example 1 : evolution of potential 
		at four time moments $t=0$ (line 1, column 1), 
		$t= \frac{T}{4}$ (line 1, column 2), $t= \frac{T}{2}$ (line 2, column 1) and $t=T$ (line 2, column 2). }
	\label{fig:test10}
\end{figure}  
\begin{figure}[pos=!ht]
	\begin{minipage}{\linewidth}
		\includegraphics[  width=.5\linewidth]{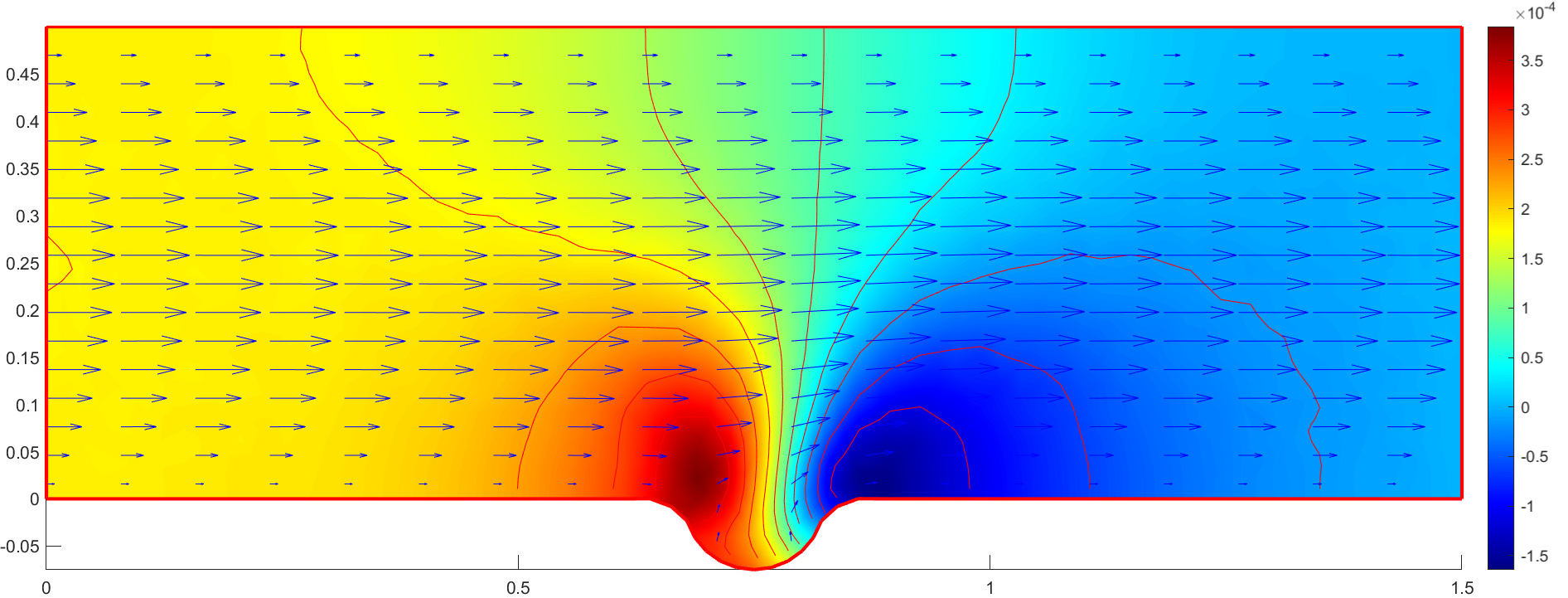}  
		\includegraphics[  width=.5\linewidth]{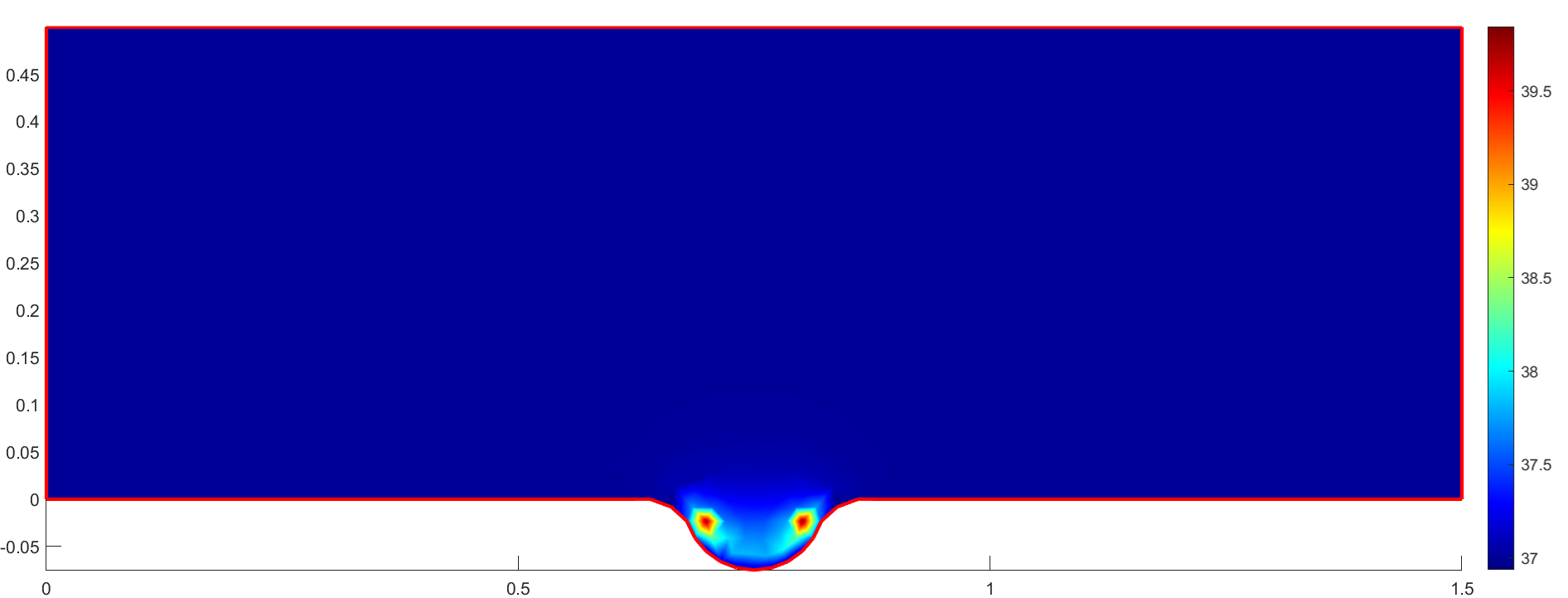}    
		\end{minipage}\\
	\begin{minipage}{\linewidth}
			\includegraphics[  width=.5\linewidth]{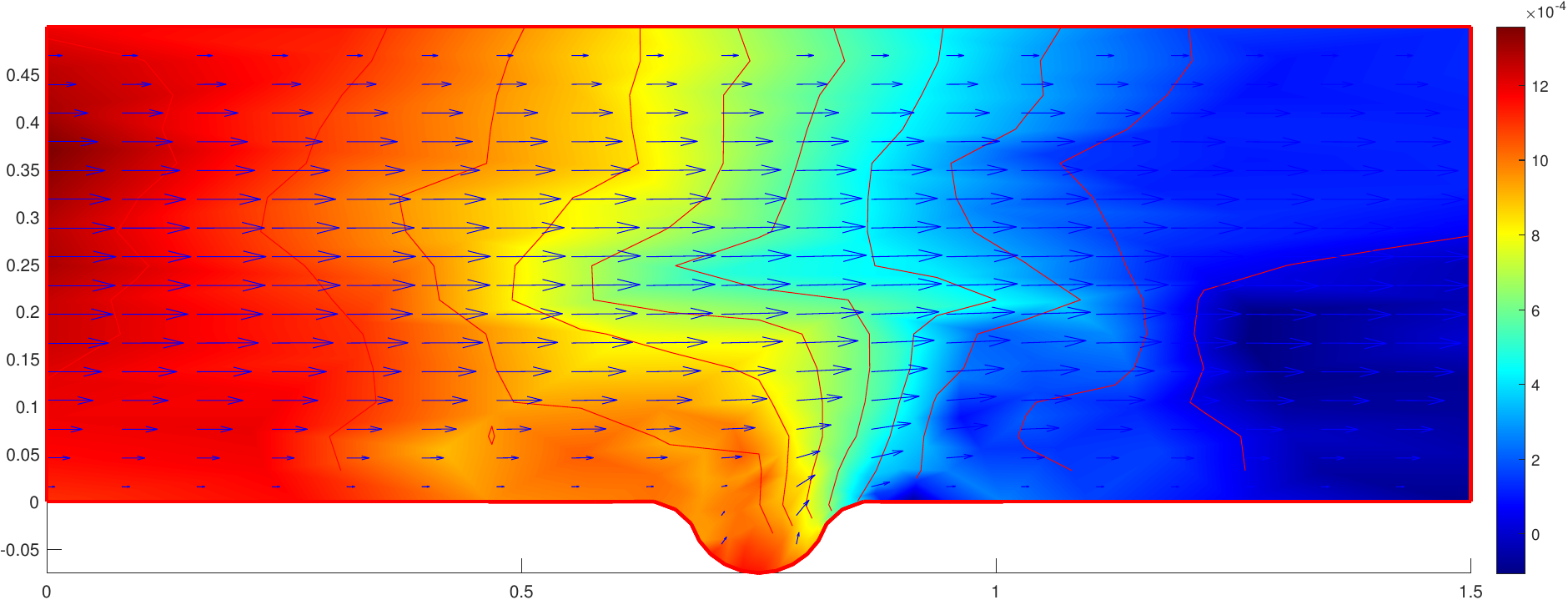}  
			\includegraphics[  width=.5\linewidth]{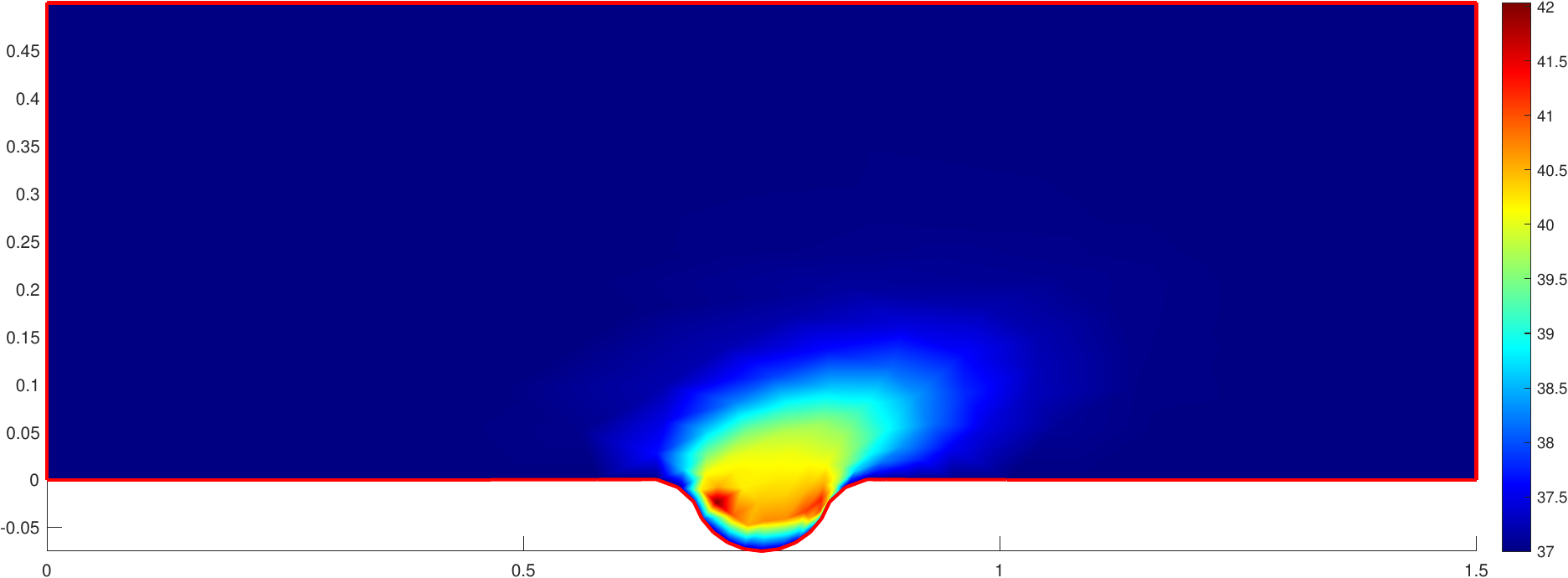}   
		\end{minipage}\\
	\begin{minipage}{\linewidth}
			\includegraphics[  width=.5\linewidth]{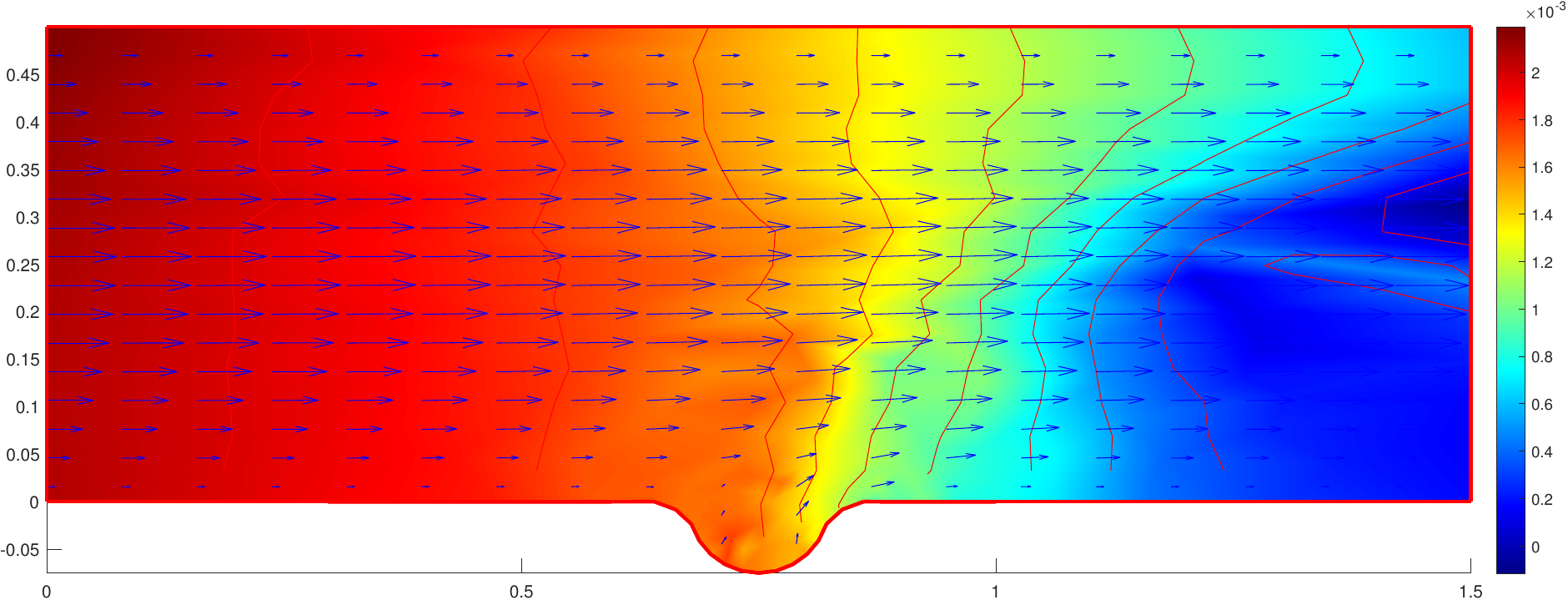}  
			\includegraphics[  width=.5\linewidth]{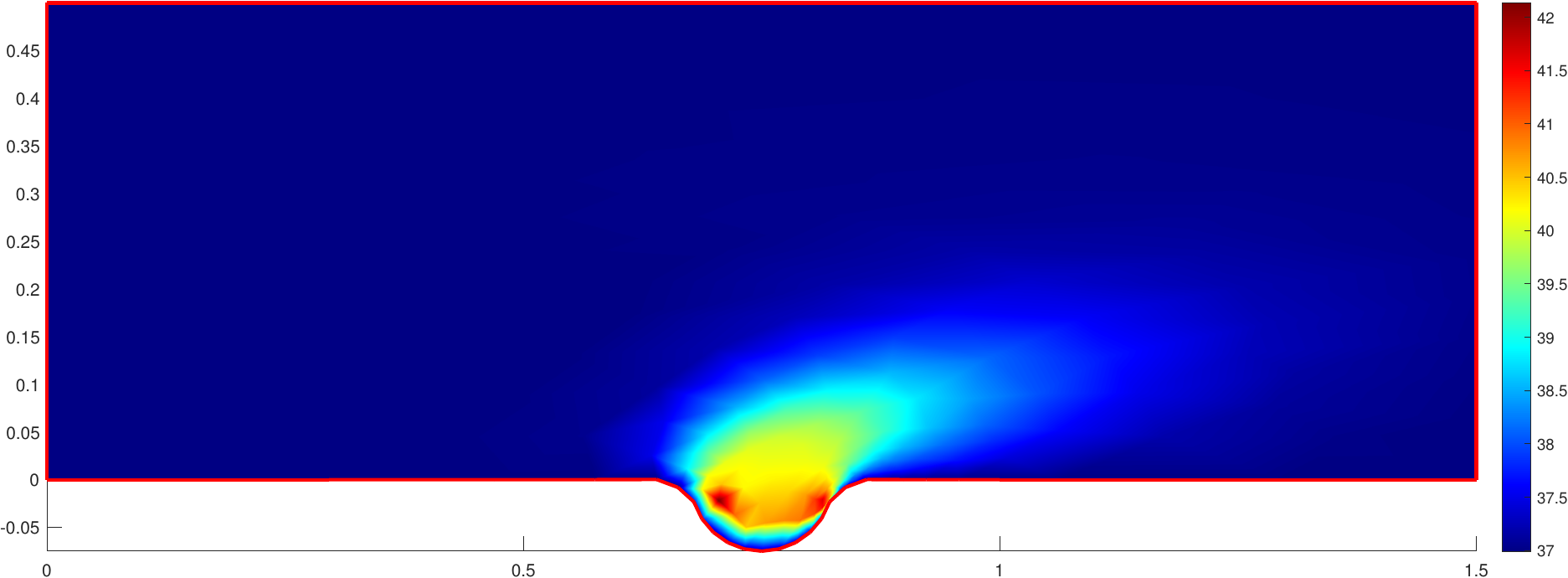}  
		\end{minipage}\\
	\begin{minipage}{\linewidth}
			\includegraphics[  width=.5\linewidth]{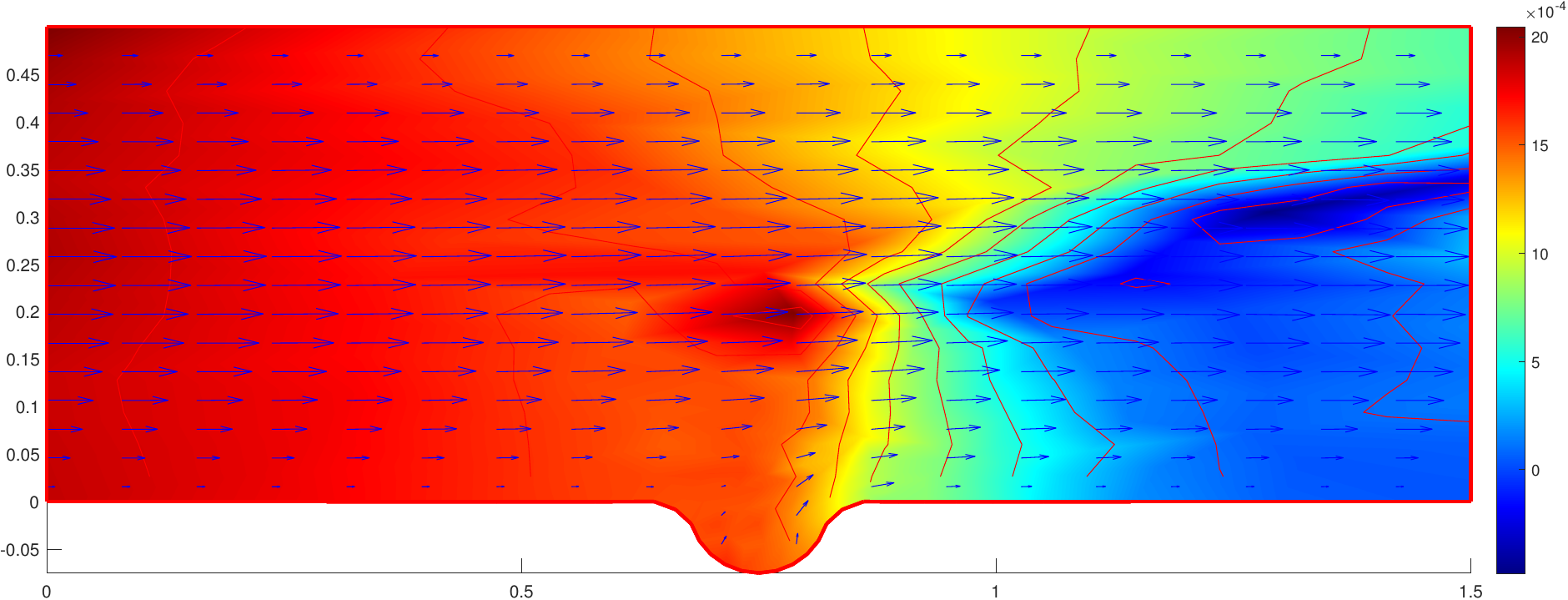}  
			\includegraphics[  width=.5\linewidth]{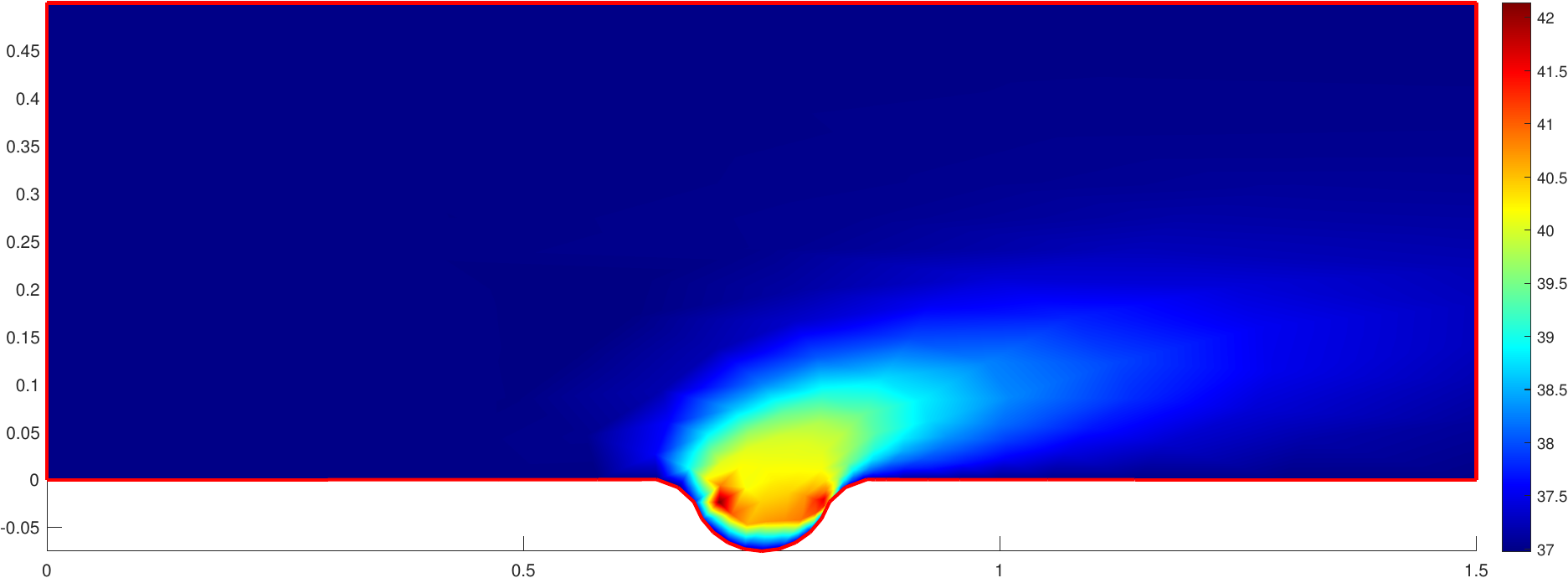}  
		\end{minipage}
	\caption{Example 1 : evolution of  velocity and pressure (column 1),  and heat (column2) 
			at four time moments $t=0$ (line 1), 
			$t= \frac{T}{4}$ (line 2), $t= \frac{T}{2}$ (line 3), and $t=T$ (line 4). }
	\label{fig:test1}
\end{figure}  
As the temperature changes are counted between $40^{\circ} \mathrm{C}$ and $42^{\circ} \mathrm{C}$ as a maximum value, evolving the electrical conductivity at these points we find that $\sigma$ varies between $0.627$ and $0.610$, i.e. a variation of the order $10^{-2}$. 
This is consistent with the results obtained. 
This remark is also applicable to the velocity field. Indeed, we notice that the motion of the fluid is almost the same during the time iterations.

	Let us now return to the effects of the dissipation terms. In fact, for quite large values of $g$, we have marked a rapid increase in the temperature as well as in the order of rotation of the fluid.  Thus, we arrive at an explosion of the values.

	\subsubsection{Example 2: external forces}\label{Example2}
	In this example, we are interested in the behavior of the heat when the fluid source term is non-zero, and also if we change the boundary condition in $\Gamma_3$. Impose a boundary condition on $\Gamma_3$ to limit the heat exchange with the exterior. 
	Indeed, we consider the condition $(\eta(\theta) \nabla \theta) \cdot \boldsymbol{n} + \alpha \theta =\alpha \theta_l$ also on $\Gamma_3$, and we take the fluid source $\boldsymbol F= - \left(\begin{array}{l}  0\\ 10^{-3}  9.81/303 \left( \theta - \theta_b  \right)   \end{array}\right)$ as in Boussinesq equations, and decrease $g$ to $1$. 
	\\
	We omit here the figures of the solutions at the initial iterations since they are almost the same as in the previous example. 
	We also omit the figures of the potential as there is no significant change during the iterations. 
	We represent on Figure \ref{fig:test2} the evolution of the velocity and pressure (column 1) and of the heat (column 2)  at times $t=\frac{T}{8}$, $t=\frac{T}{4}$, $t=\frac{T}{2}$  and $t=T$. 
	We also observe the rotation of the fluid in the areas subject to heat variations, especially in the area near the outlet boundary $\Gamma_3$.
	A result that we justify by the structure of the source term $F$, in particular the term $\theta-\theta_b$, indeed by the principle of maximum the velocity changes its sign according to the value of the temperature $\theta$ whether it is lower or higher than $\theta_b$. 
\begin{figure}[pos =!ht]
	\begin{minipage}{\linewidth}
		\includegraphics[  width=.5\linewidth]{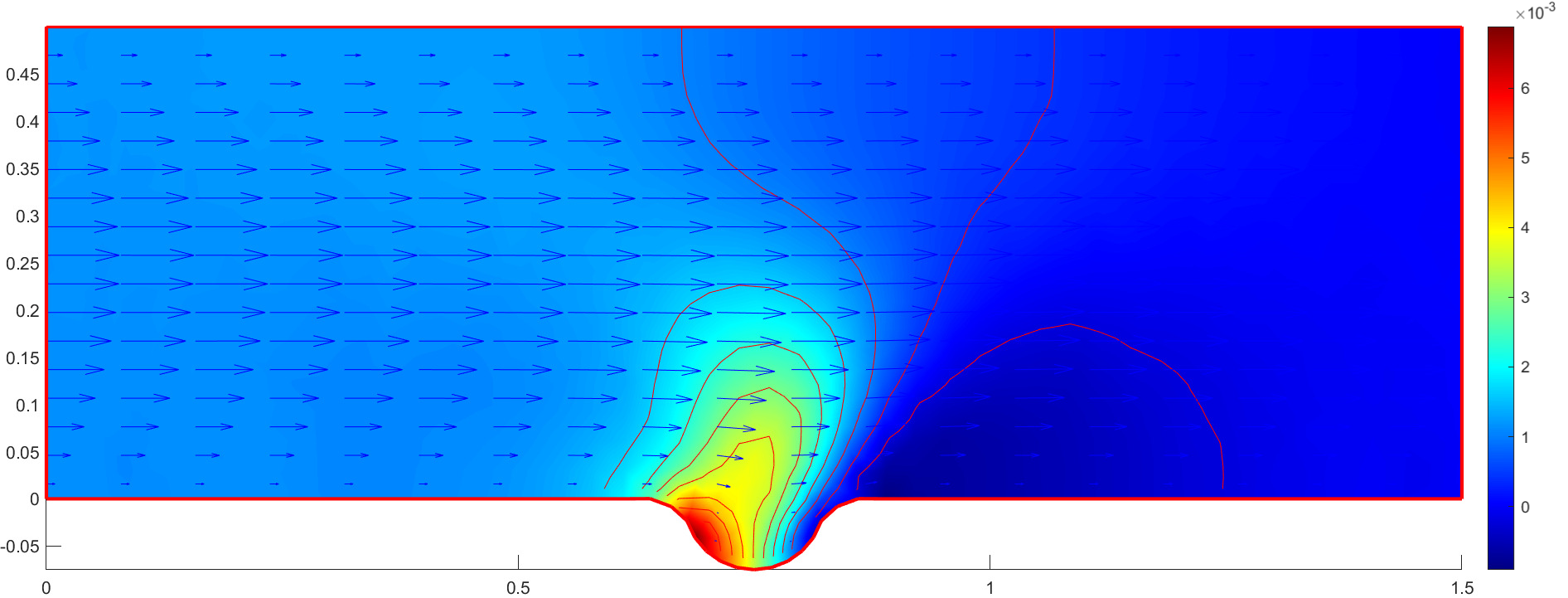}  
		\includegraphics[  width=.5\linewidth]{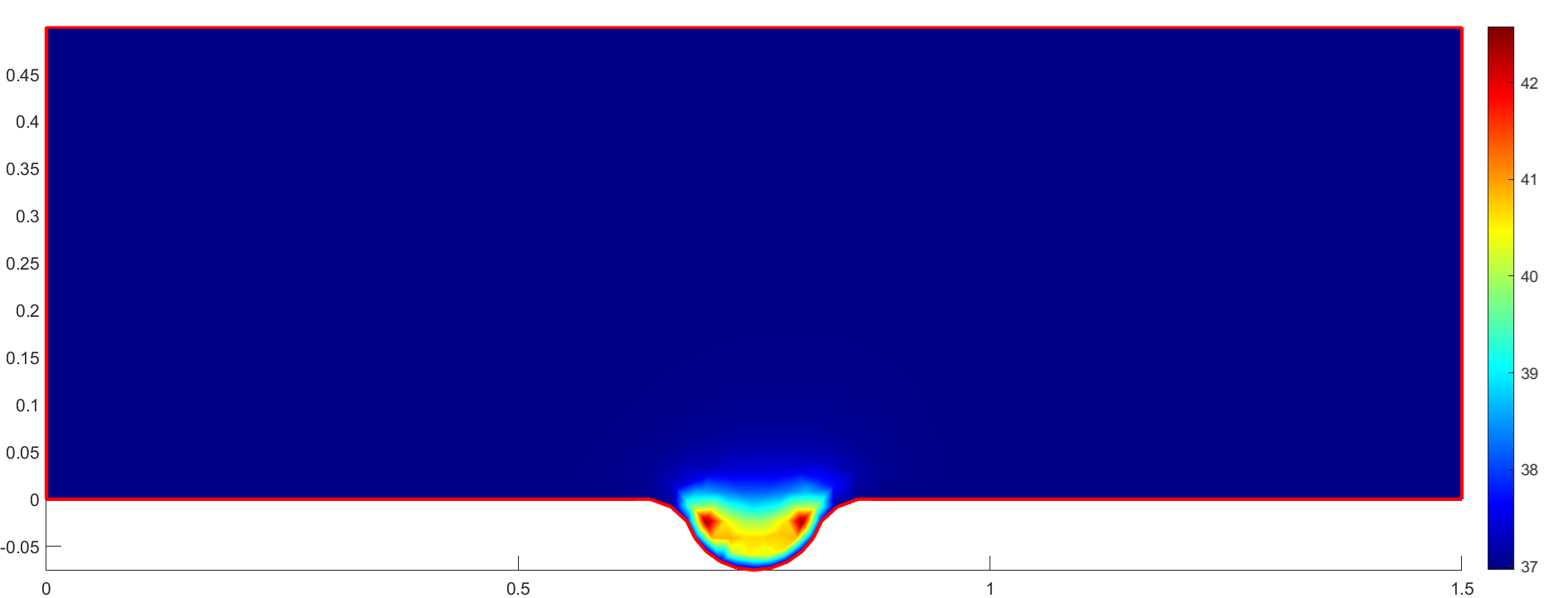}    
	\end{minipage}\\
	\begin{minipage}{\linewidth}
		\includegraphics[  width=.5\linewidth]{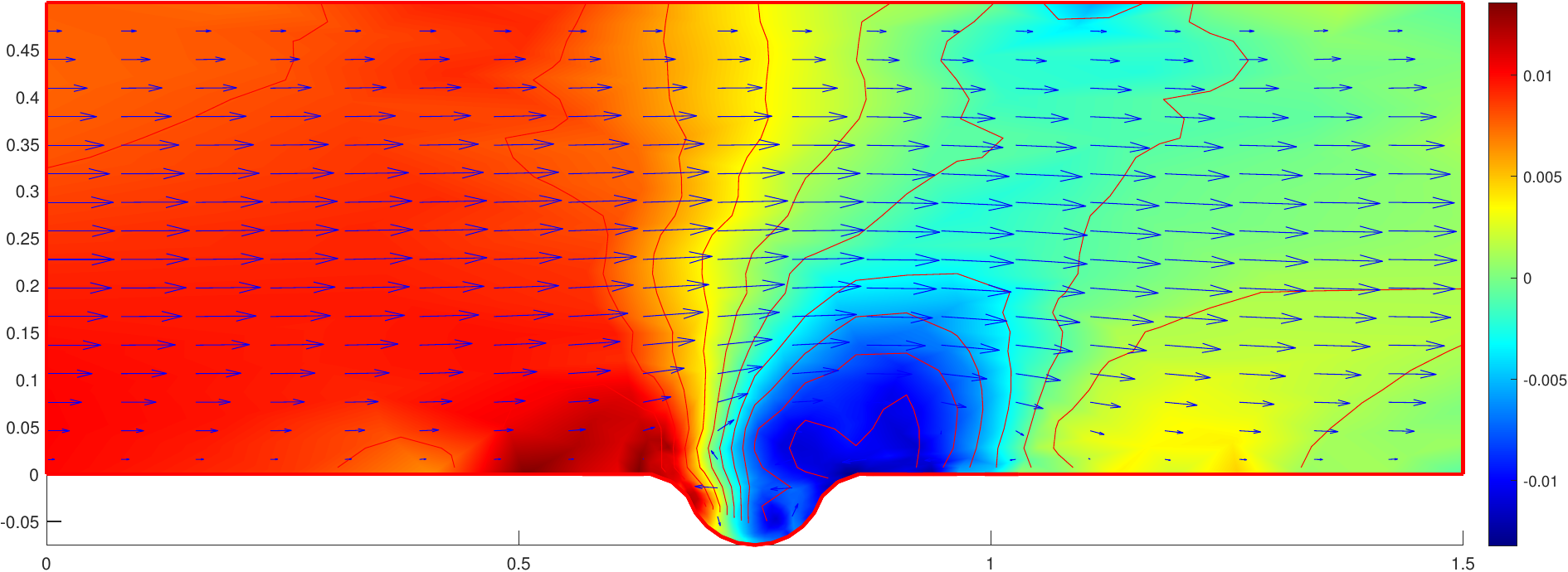}  
		\includegraphics[  width=.5\linewidth]{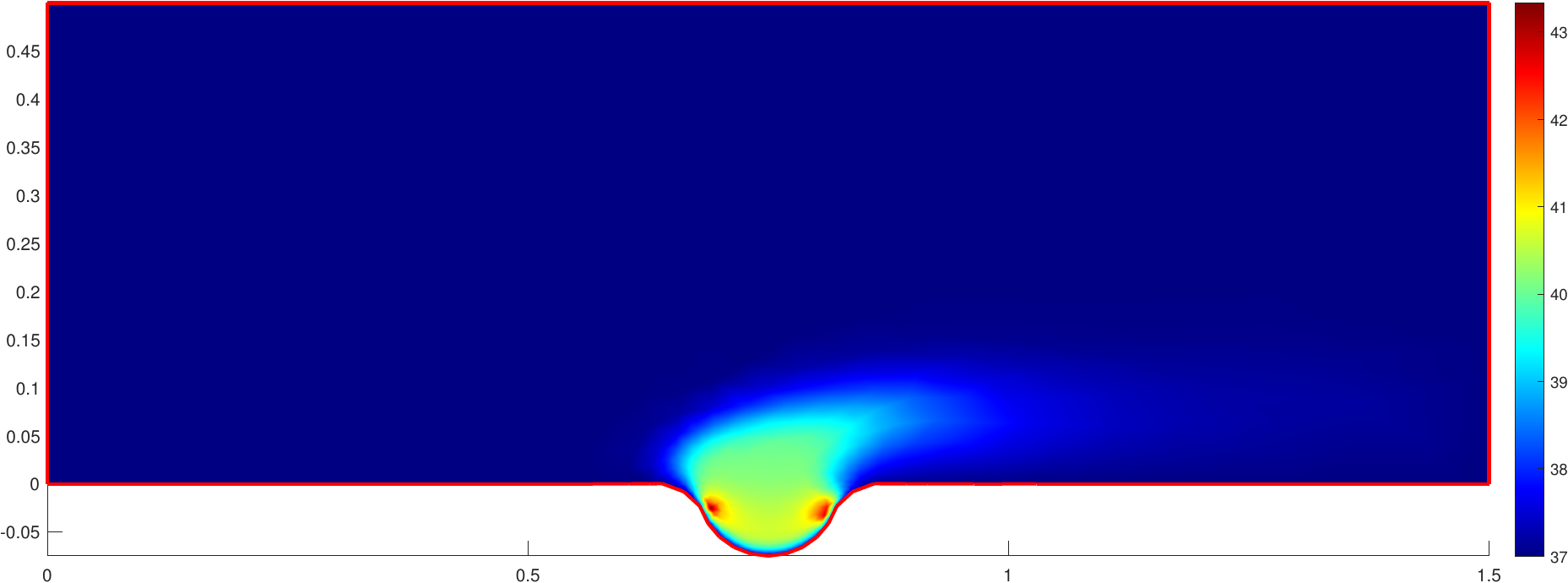}   
	\end{minipage}\\
	\begin{minipage}{\linewidth}
		\includegraphics[  width=.5\linewidth]{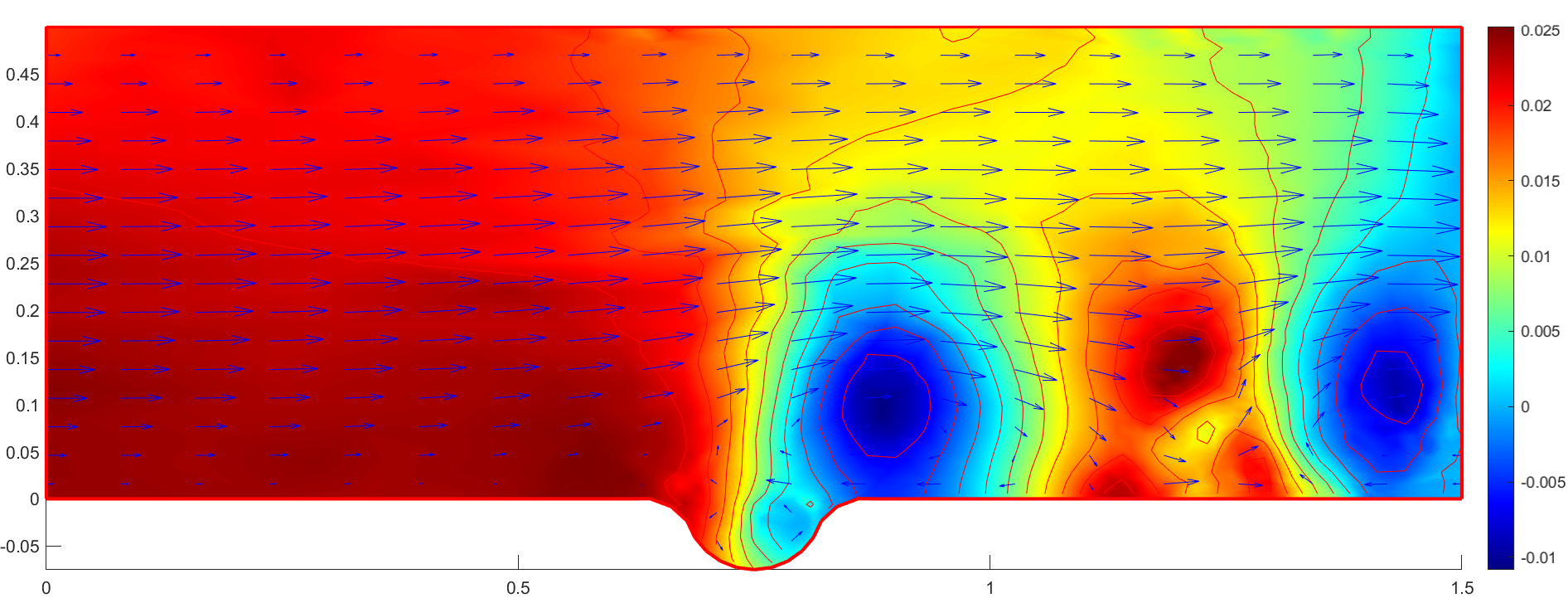}  
		\includegraphics[  width=.5\linewidth]{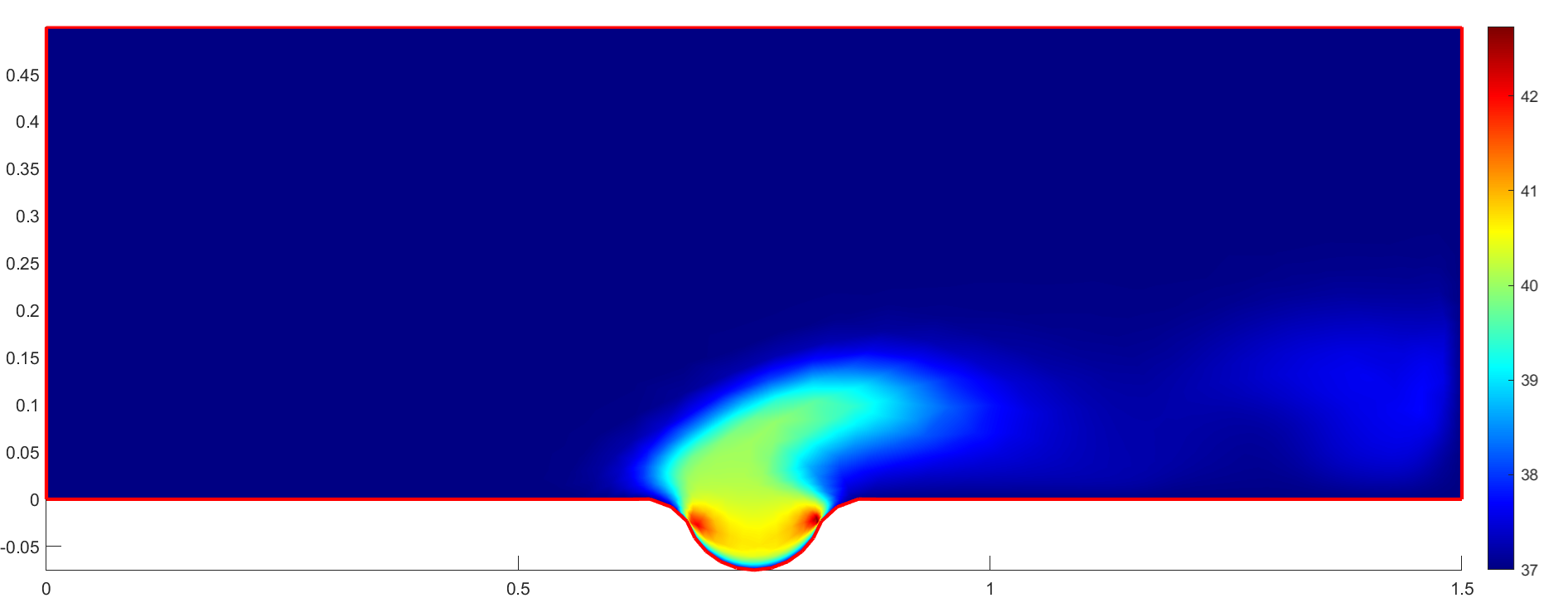}   
	\end{minipage}\\
	\begin{minipage}{\linewidth}
		\includegraphics[  width=.5\linewidth]{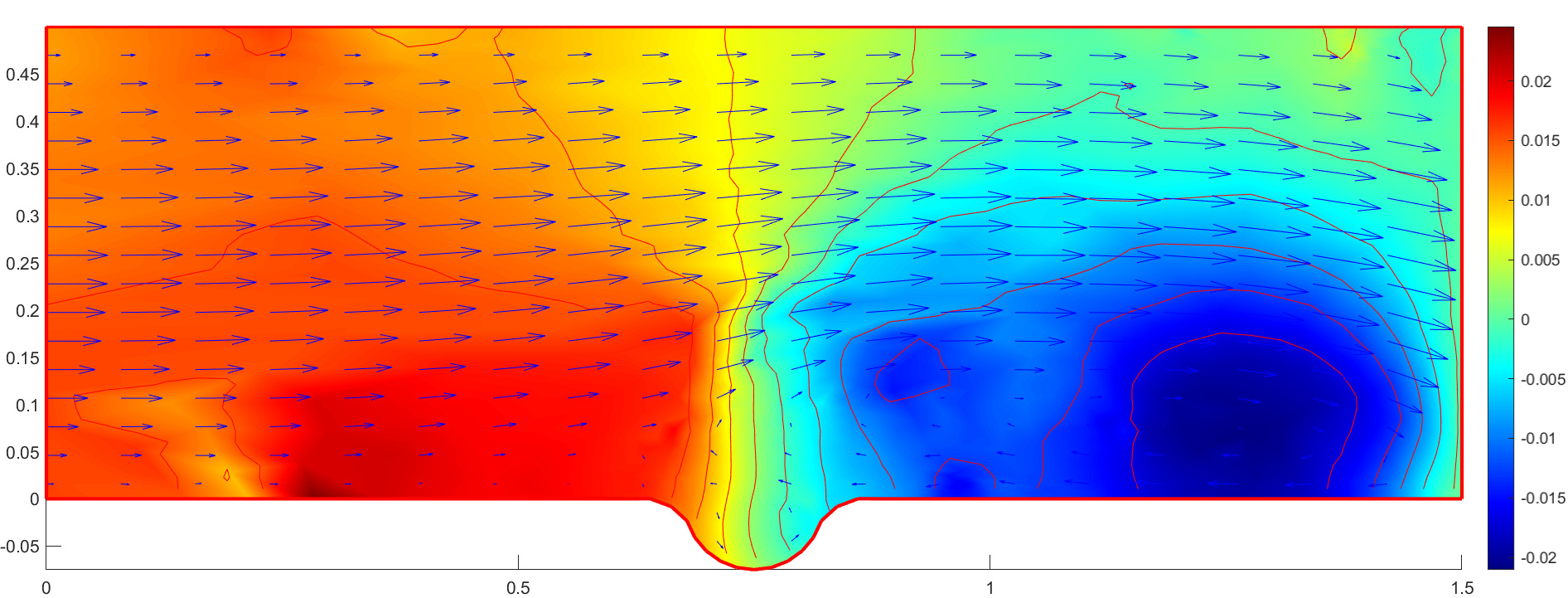}  
		\includegraphics[  width=.5\linewidth]{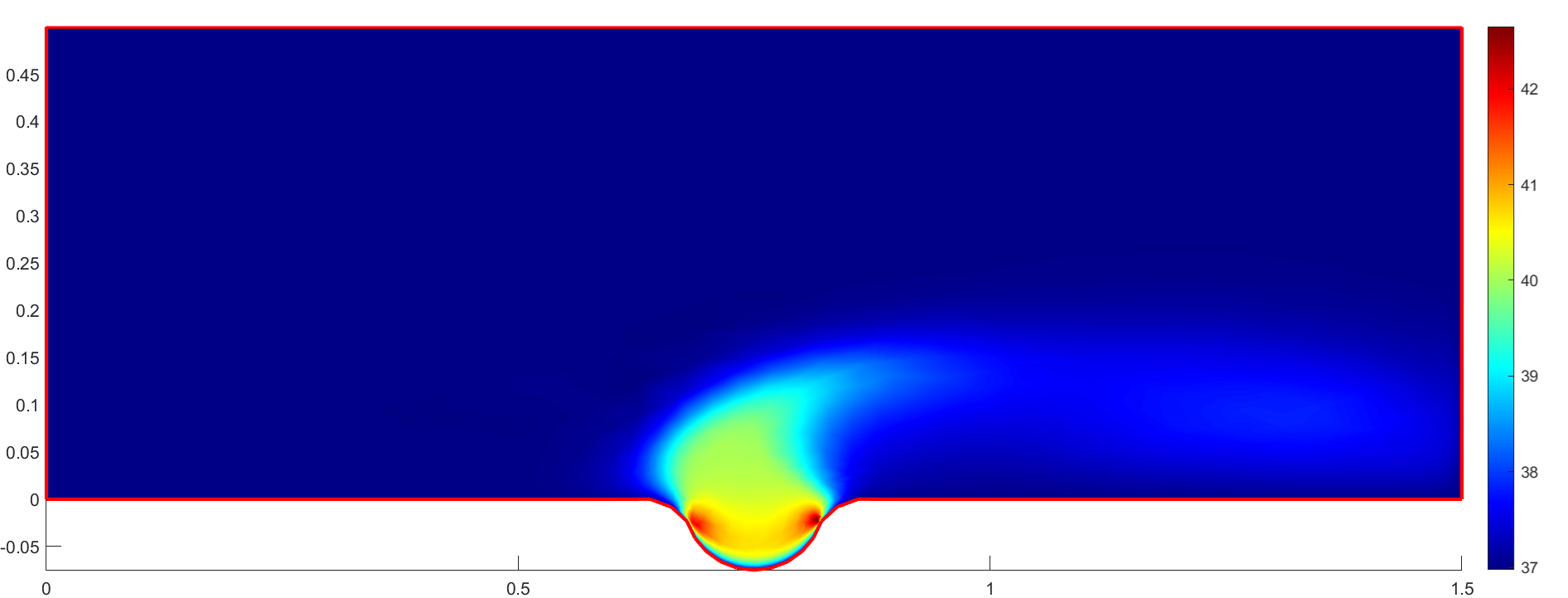}    
	\end{minipage}
	\caption{Example 2 : evolution of velocity and pressure (column 1), heat (column 2) 
		at four time moments $t=\frac{T}{8}$ (line 1), 
		$t= \frac{T}{4}$ (line 2), $t= \frac{T}{2}$ (line 3) and $t=T$ (line 4). }	
	\label{fig:test2}
\end{figure}  
	\subsubsection{Example 3: effect of the saline flow}
We note that with the configuration obtained in Example 2, the reduction in $g$ implies a reduction in the potential in the domain and consequently the calculated heat is reduced, but the temperature around the catheter reaches critical values between $40^{\circ}C$ and $42^{\circ}C$. It is therefore of course necessary to cool this area and lower its temperature. To do this, it is necessary to inject a fluid whose saline heat is $20^{\circ}C$, i.e. we used
${\boldsymbol{v}}=\boldsymbol{v}_s = \left(\begin{array}{l}  \displaystyle \frac{2}{r}(x-\frac L 2  +r)(\frac L 2 +r -x)(\frac L 2   -x) \\ \displaystyle  \frac{-2}{r} (x-\frac L 2  +r)(\frac L 2 +r -x)y \end{array}\right)$ and $T=T_s=20^{\circ}C$ on boundary $\Gamma_5$. Clearly, we notice that the injected saline flow $\boldsymbol{v}_s$ diminishes the calculated heats  (see Figure \ref{fig:test3}). This leads to the possibility of cooling the domain by the saline fluid from $\Gamma_{5}$(maximum heat between $39^{\circ}C$ and $40^{\circ}C$). In addition, we observe the rotation of the fluid in the areas subject to heat variations, especially in the area near the outlet boundary $\Gamma_3$.
\begin{figure}[pos =!ht]
	\begin{minipage}{\linewidth}
		\includegraphics[  width=.5\linewidth]{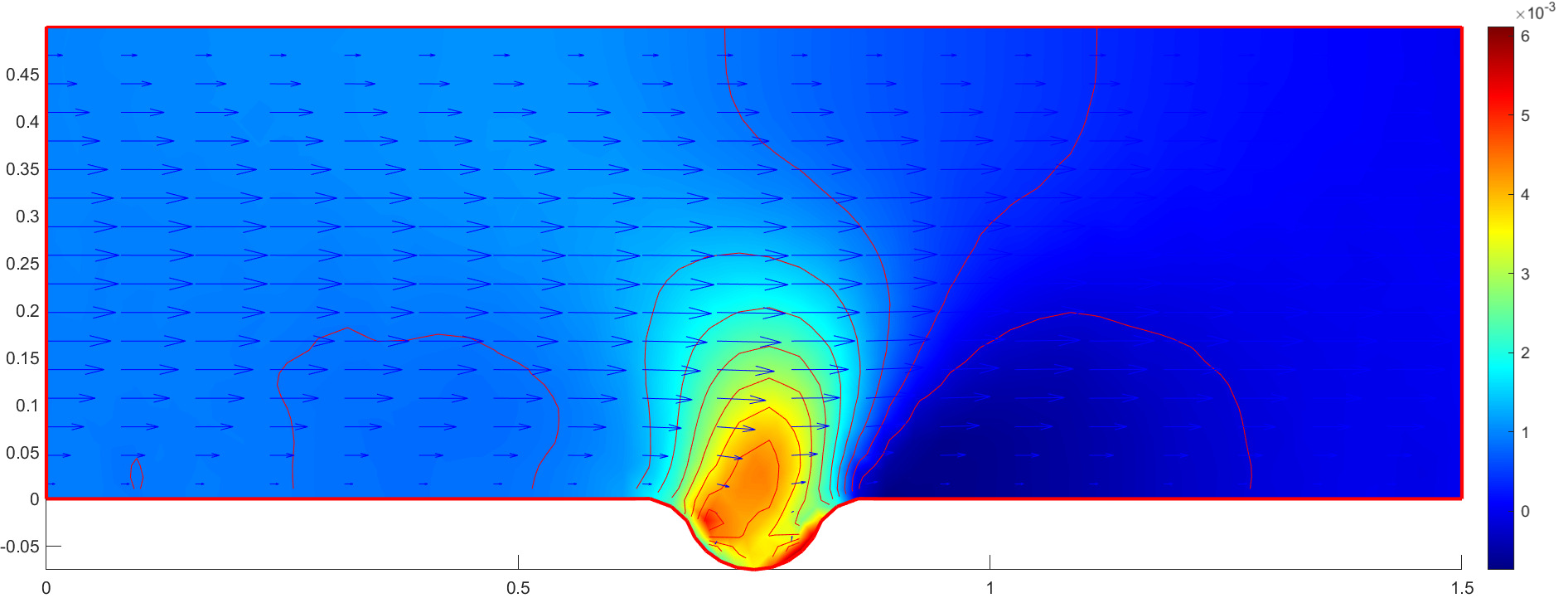}  
		\includegraphics[  width=.5\linewidth]{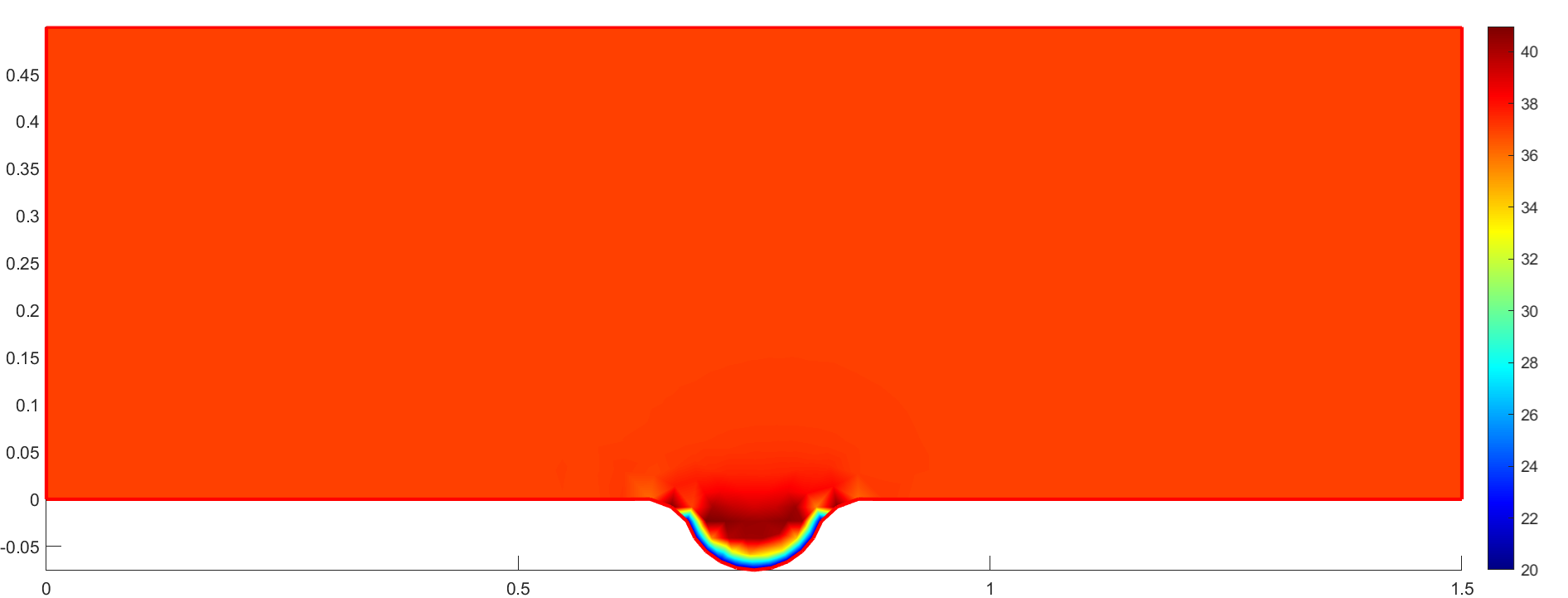}    
	\end{minipage}\\
	\begin{minipage}{\linewidth}
		\includegraphics[  width=.5\linewidth]{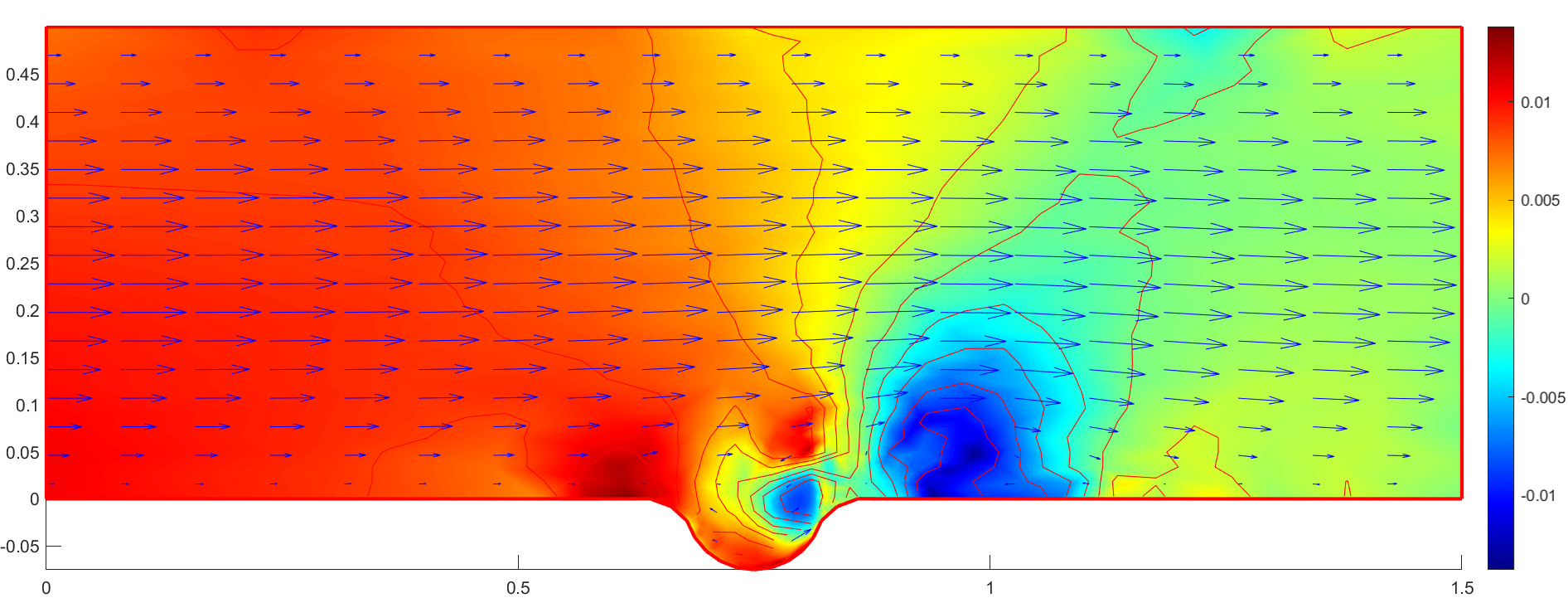}  
		\includegraphics[  width=.5\linewidth]{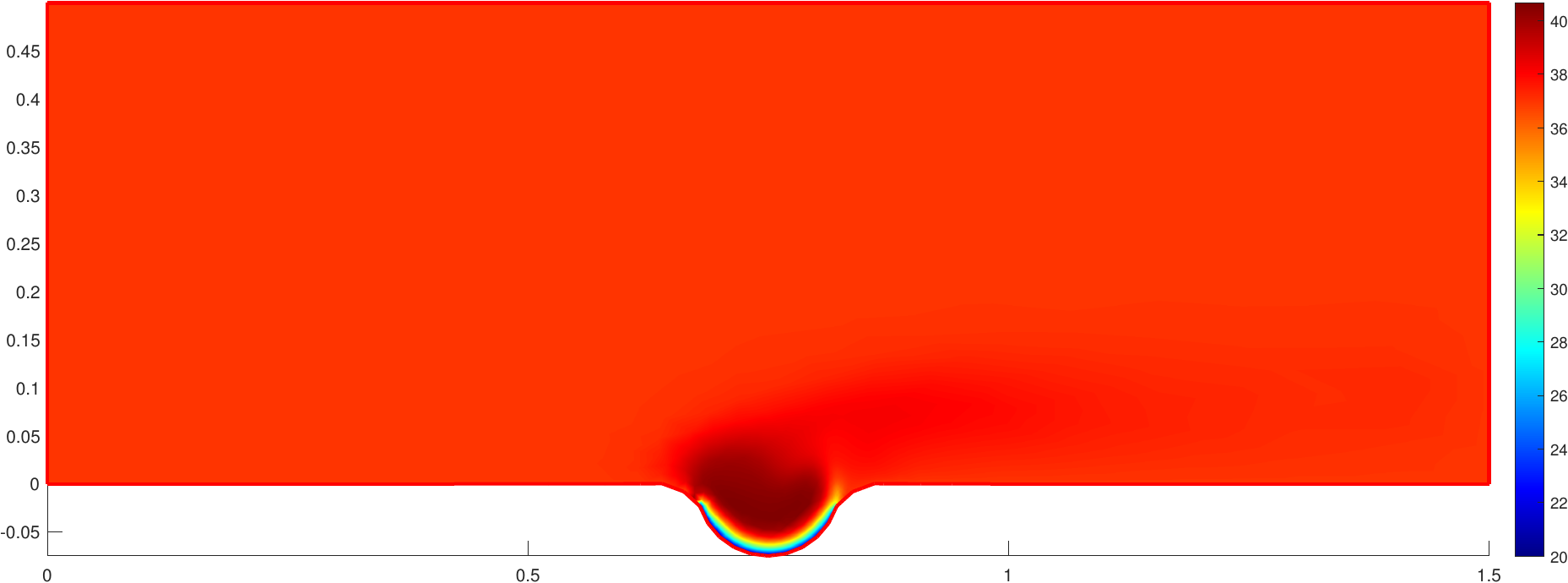}   
	\end{minipage}\\
	\begin{minipage}{\linewidth}
		\includegraphics[  width=.5\linewidth]{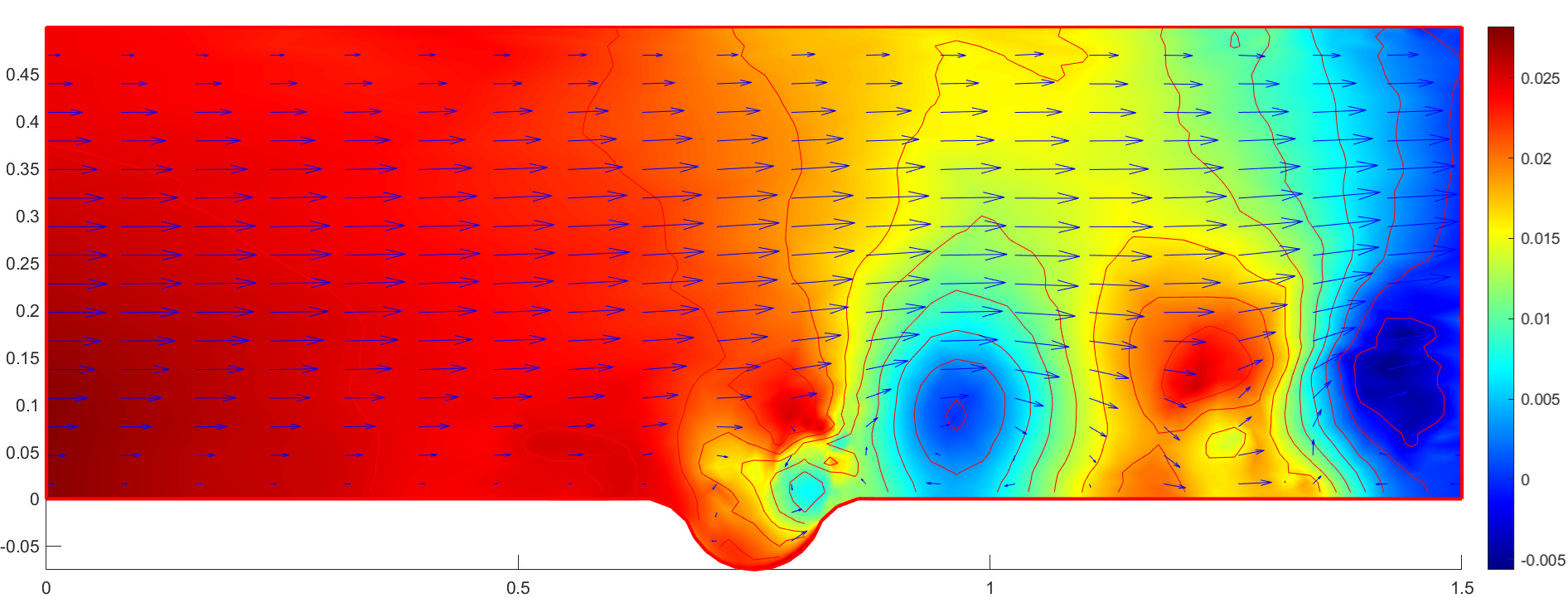}  
		\includegraphics[  width=.5\linewidth]{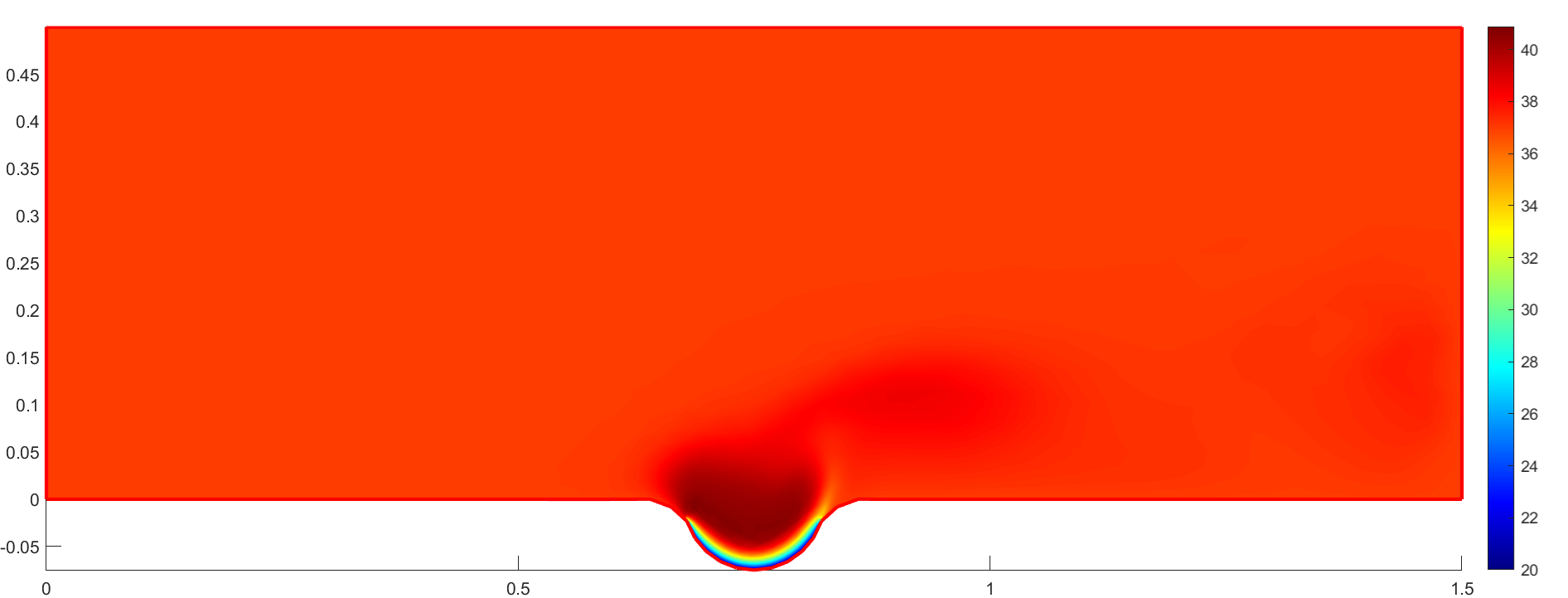}   
	\end{minipage}\\
	\begin{minipage}{\linewidth}
		\includegraphics[  width=.5\linewidth]{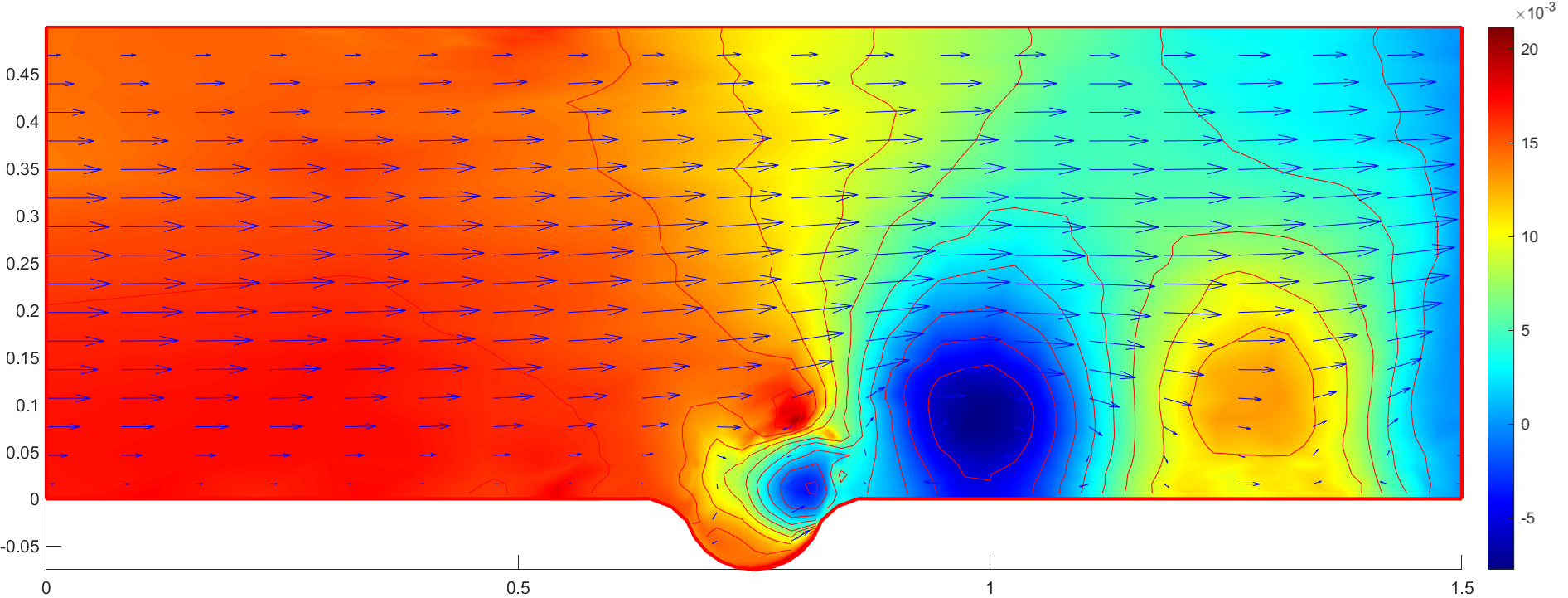}  
		\includegraphics[  width=.5\linewidth]{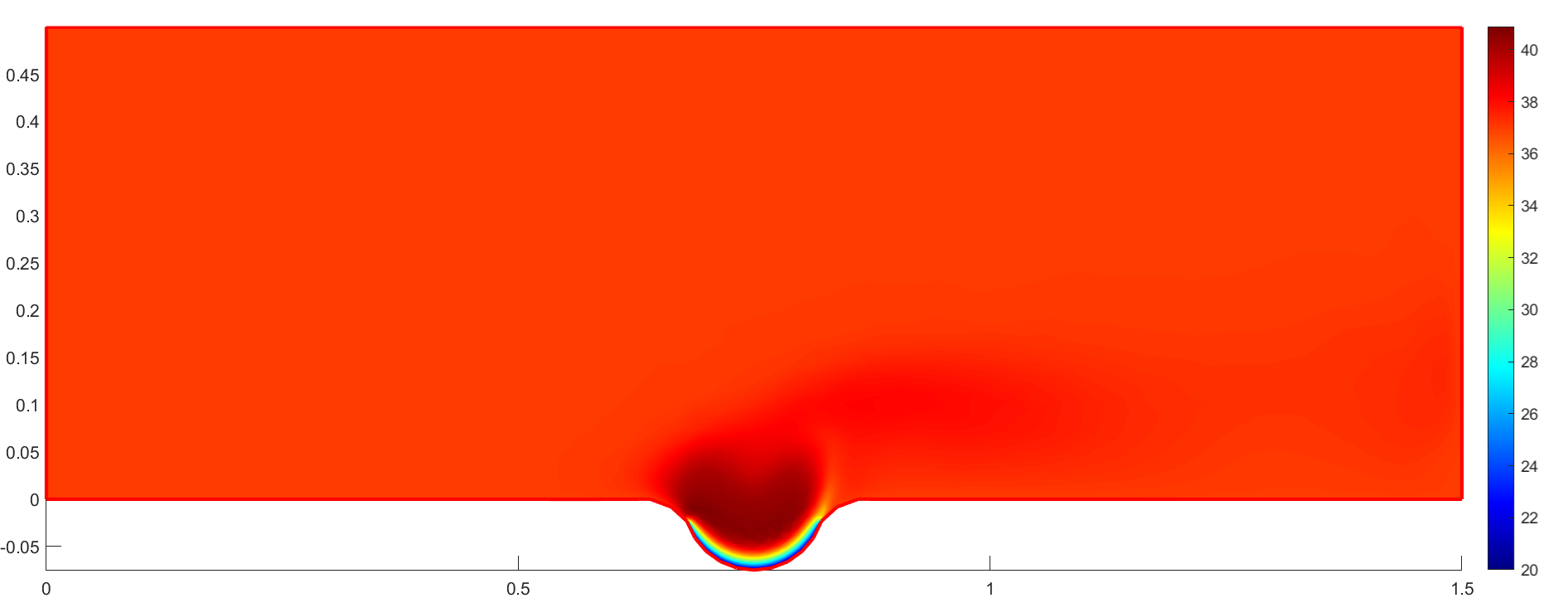}    
	\end{minipage}
	\caption{Example 3 : evolution of velocity and pressure (column 1), heat (column 2) 
		at four time moments $t=\frac{T}{8}$ (line 1), 
		$t= \frac{T}{4}$ (line 2), $t= \frac{T}{2}$ (line 3) and $t=T$ (line 4). }	
	\label{fig:test3}
\end{figure}  
	\subsubsection{Example 4: cooling factor}
	 We mention that with the configuration obtained in Example 3, we have achieved a reduction of the temperature in certain areas of the domain. However, this ceases to work from a certain level and the heat will be balanced because of the domain's homogeneity. 
To this end, we can add other cooling factors by assuming that the heat of the fluid will enter through $\Gamma_1$ with a different temperature than the domain one, i.e. $\theta= 35^{\circ} \mathrm{C}$. The results of this choice are shown in Figure \ref{fig:test4} with the same descriptions as in Figure \ref{fig:test3}. 

\begin{figure}[pos =!ht]
	\begin{minipage}{\linewidth}
	\includegraphics[  width=.5\linewidth]{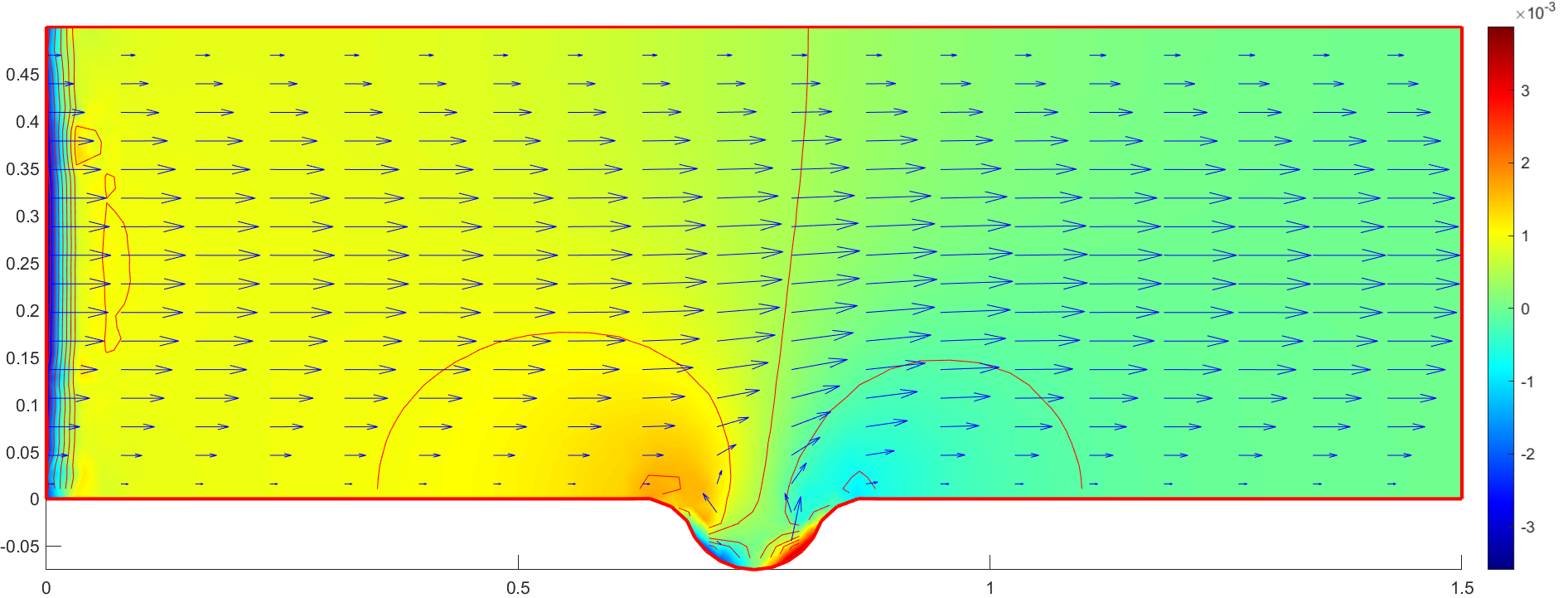}  
	\includegraphics[  width=.5\linewidth]{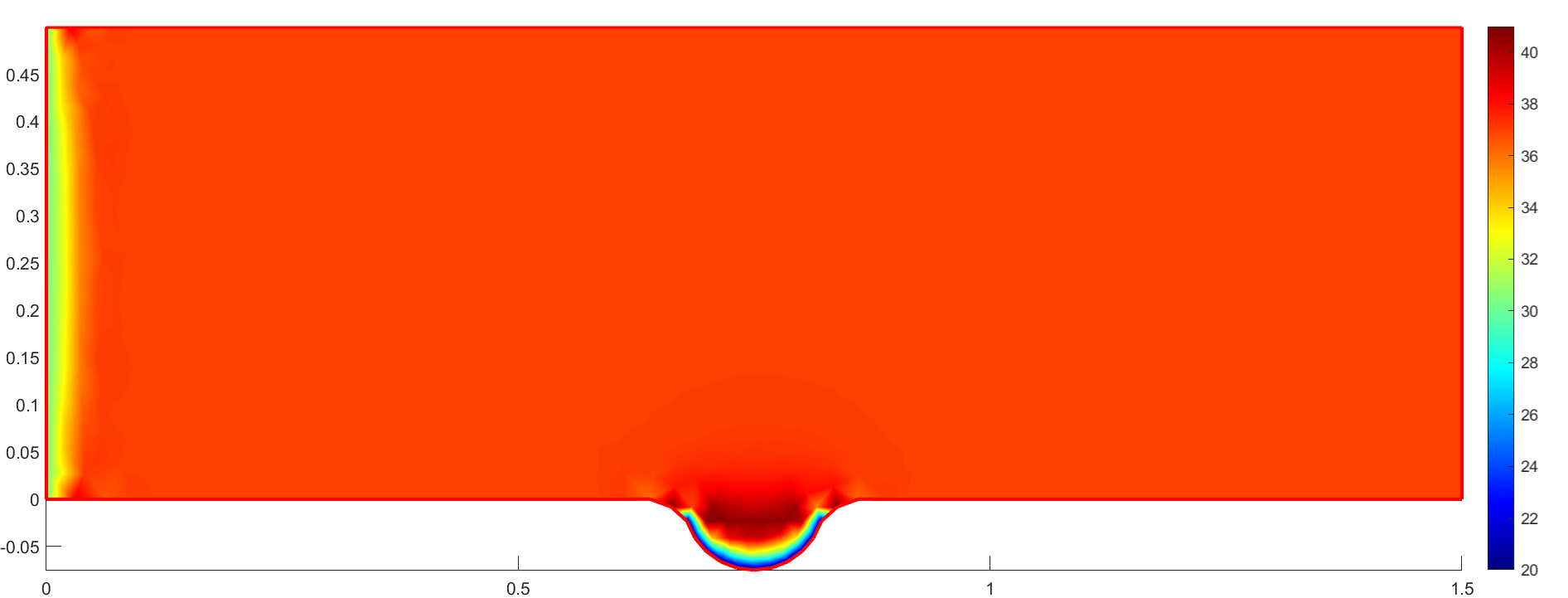}    
\end{minipage}\\
\begin{minipage}{\linewidth}
	\includegraphics[  width=.5\linewidth]{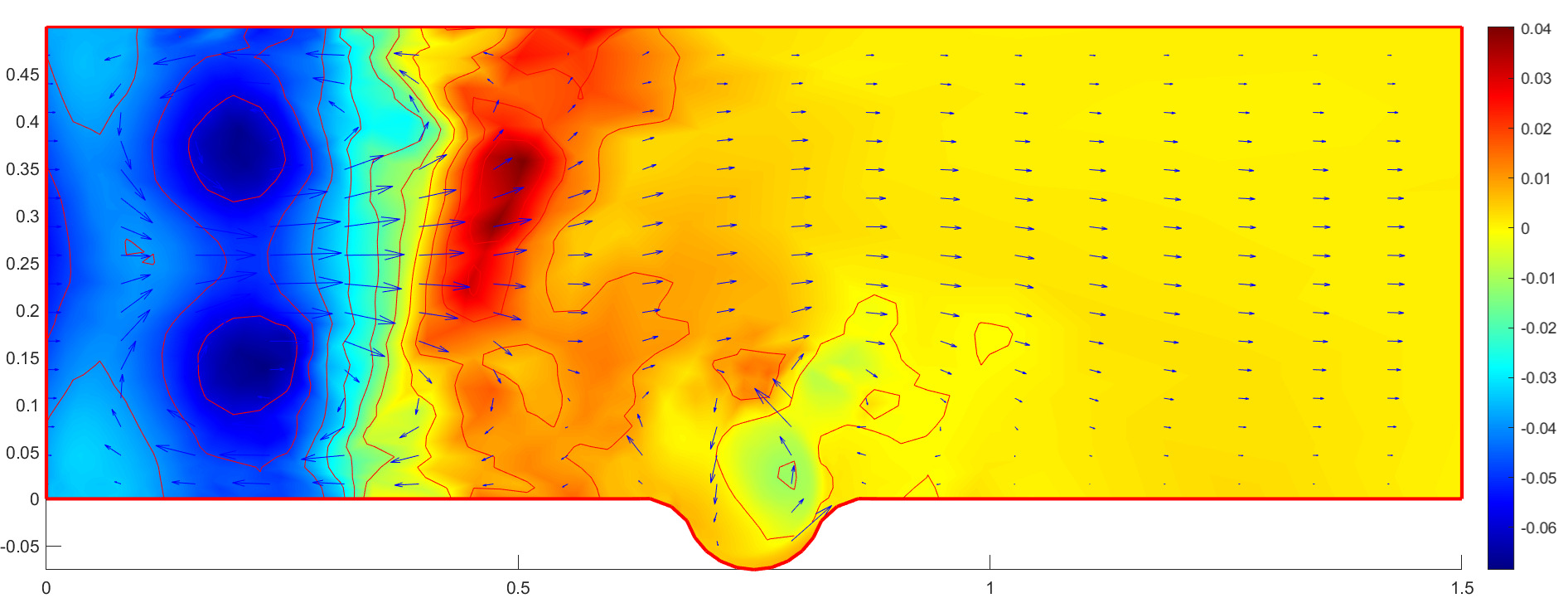}  
	\includegraphics[  width=.5\linewidth]{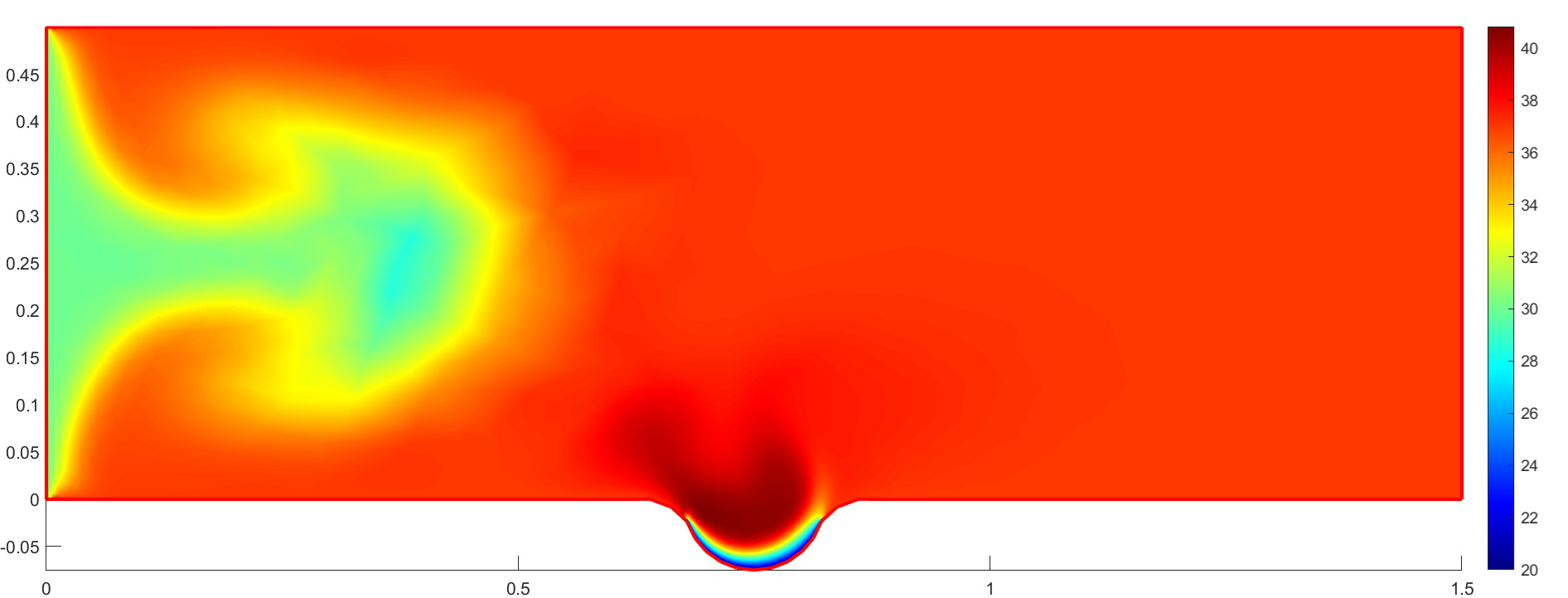}   
\end{minipage}\\
\begin{minipage}{\linewidth}
	\includegraphics[  width=.5\linewidth]{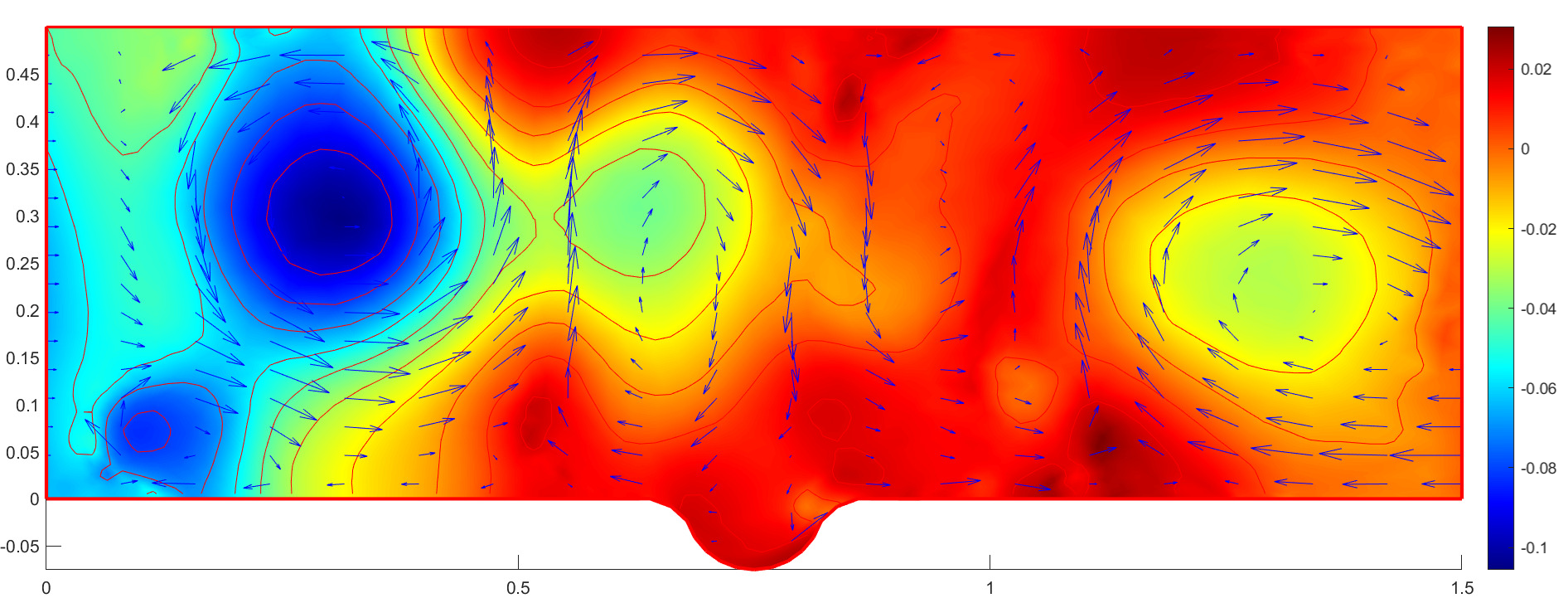}  
	\includegraphics[  width=.5\linewidth]{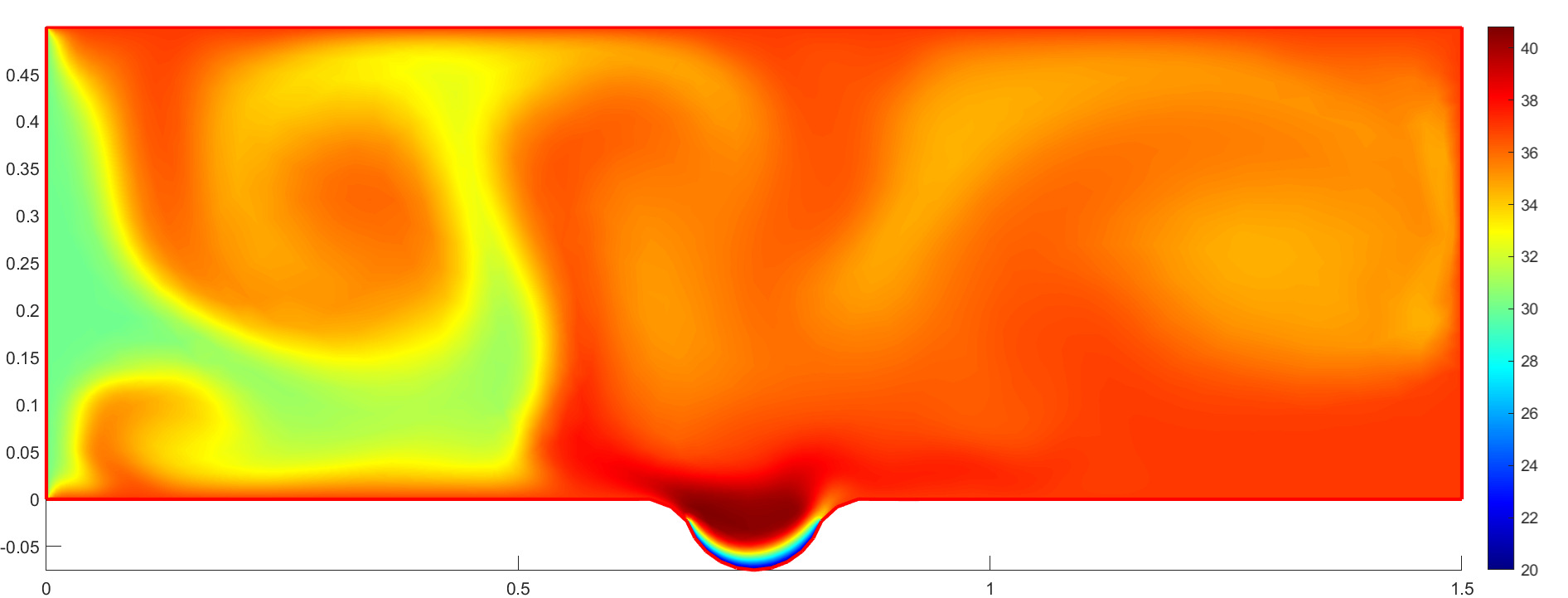}   
\end{minipage}\\
\begin{minipage}{\linewidth}
	\includegraphics[  width=.5\linewidth]{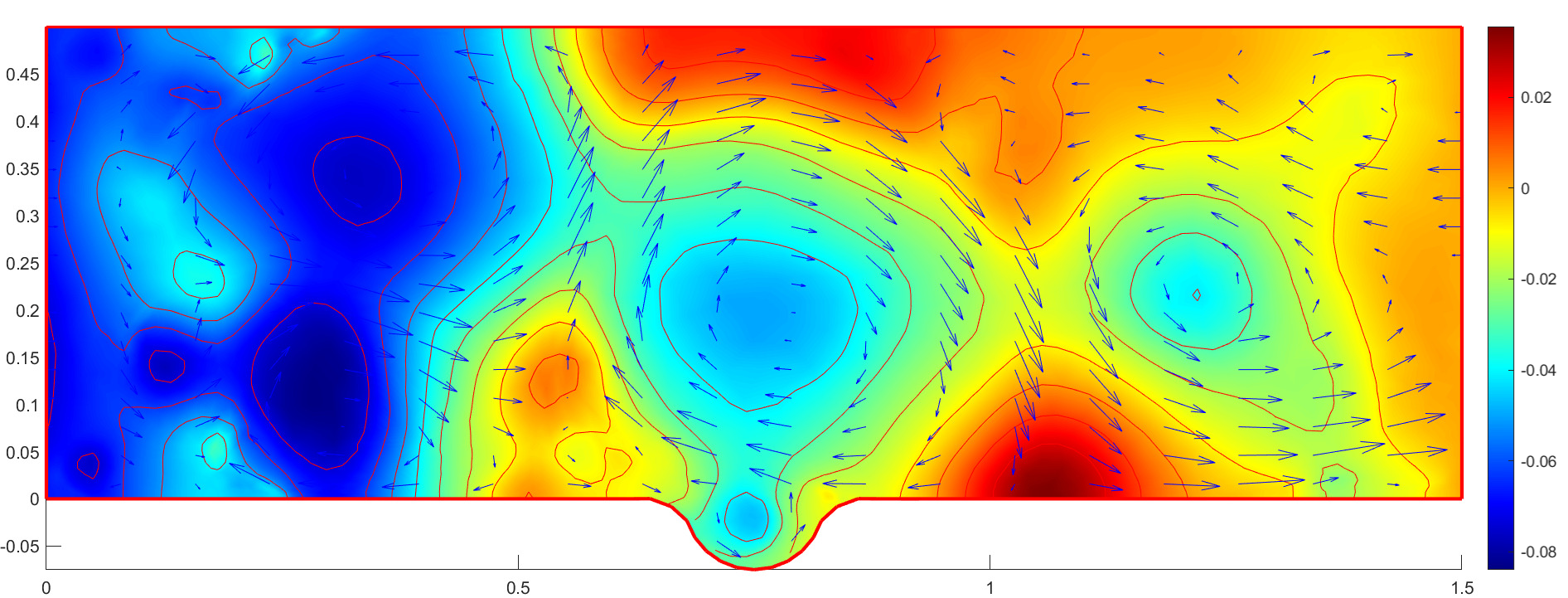}  
	\includegraphics[  width=.5\linewidth]{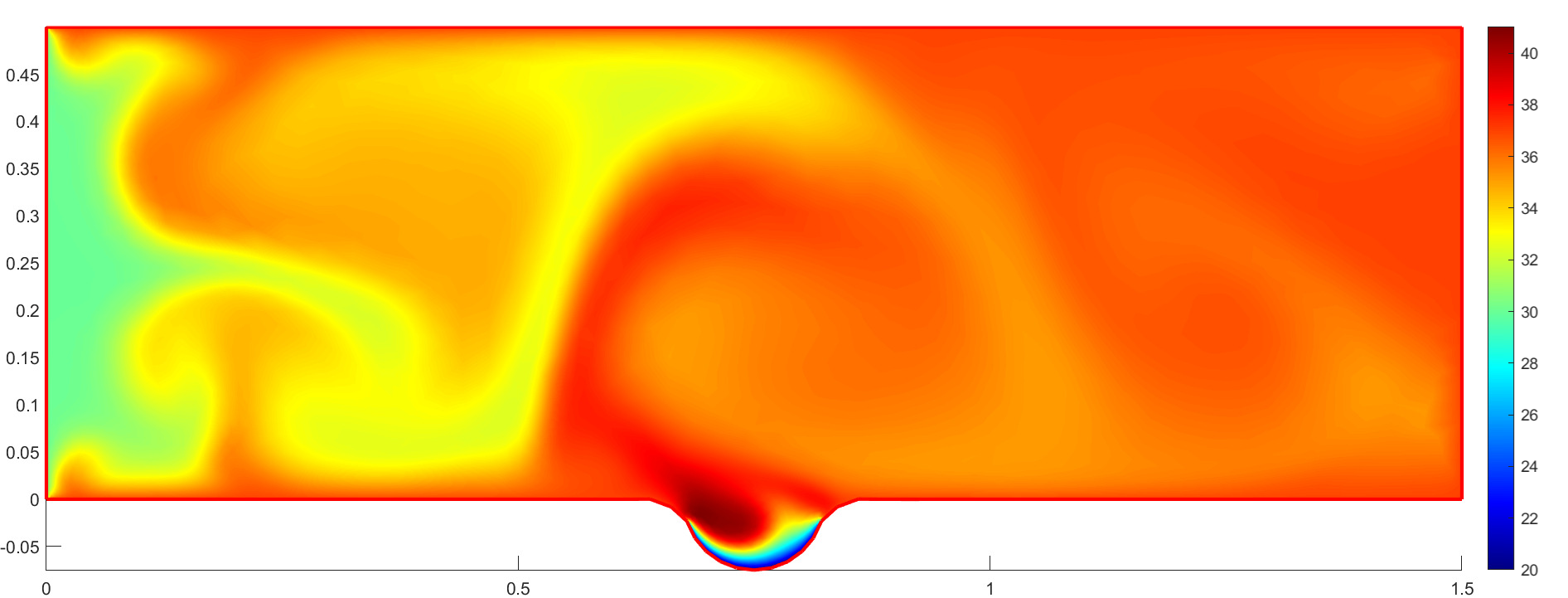}    
\end{minipage}
	\caption{Example 4 : evolution of velocity  and pressure (column 1), heat (column 2) 
		at four time moments $t=\frac{T}{8}$ (line 1), 
		$t= \frac{T}{4}$ (line 2), $t= \frac{T}{2}$ (line 3) and $t=T$ (line 4). }	
	\label{fig:test4}
\end{figure}


\section{Conclusion and perspectives}

In this paper, a nonlinear fluid-heat-potential system modeling radiofrequency ablation phenomena in cardiac tissue has been proposed. The existence of the global solutions using Schauder's fixed-point theory has been demonstrated, as well as their uniqueness under some additional conditions on the data, both in two-dimensional and three-dimensional space.  Numerical simulation in different cases have been illustrated in a two-dimensional space using the finite element method.

The phenomena of radiofrequency ablation in different tissues are procedures that make it possible to predict the temperature of the tissues during these procedures. For this reason, we believe that this work opens up interesting perspectives, such as optimal control models and inverse problems, namely the identification of the frequency factor of different types of tissue. 

As we were equipped in the last section, for $g$ large enough, we notice a rapid increase in temperature as well as in the order of rotation of the fluid. This motivates us to study particular cases where the source terms are less regular, the case $L^1$ for example. {However, it is important to note that the numerical resolution of the shemas proposed in this document is only one of the perspectives for future works}.

Other perspectives consist in deriving system \eqref{System} from a kinetic-fluid model. This can improve our knowledge from the modeling point of view, as the kinetic (mesoscopic) scale gives a more detailed insight into the involved interactions. However, for more details, we refer the interested reader to \cite{ABKMZ20}. Another interesting perspective could be to consider the stochastic aspect, see \cite{BNTZ23,BTZ22,Z23}.

\bibliographystyle{cas-model2-names} 
\bibliography{simple}

\begin{thebibliography}{53}
\expandafter\ifx\csname natexlab\endcsname\relax\def\natexlab#1{#1}\fi
\providecommand{\url}[1]{\texttt{#1}}
\providecommand{\href}[2]{#2}
\providecommand{\path}[1]{#1}
\providecommand{\DOIprefix}{doi:}
\providecommand{\ArXivprefix}{arXiv:}
\providecommand{\URLprefix}{URL: }
\providecommand{\Pubmedprefix}{pmid:}
\providecommand{\doi}[1]{\href{http://dx.doi.org/#1}{\path{#1}}}
\providecommand{\Pubmed}[1]{\href{pmid:#1}{\path{#1}}}
\providecommand{\bibinfo}[2]{#2}
\ifx\xfnm\relax \def\xfnm[#1]{\unskip,\space#1}\fi
\bibitem[{Adams(1975)}]{Adams}
\bibinfo{author}{Adams, R et~Fournier, J.}, \bibinfo{year}{1975}.
\newblock \bibinfo{title}{Espaces de {S}obolev, acad}.
\newblock \bibinfo{journal}{Presse, New York} \bibinfo{volume}{19}.
\bibitem[{Adams(2003)}]{Adams2003}
\bibinfo{author}{Adams, Robert A et~Fournier, J.}, \bibinfo{year}{2003}.
\newblock \bibinfo{title}{Espaces de {S}obolev mathématiques pures et
  appliquées ({A}msterdam), 140}.
\bibitem[{Ahmed et~al.(2008)Ahmed, Liu, Humphries and
  Nahum~Goldberg}]{ahmed2008}
\bibinfo{author}{Ahmed, M.}, \bibinfo{author}{Liu, Z.},
  \bibinfo{author}{Humphries, S.}, \bibinfo{author}{Nahum~Goldberg, S.},
  \bibinfo{year}{2008}.
\newblock \bibinfo{title}{Computer modeling of the combined effects of
  perfusion, electrical conductivity, and thermal conductivity on tissue
  heating patterns in radiofrequency tumor ablation}.
\newblock \bibinfo{journal}{International Journal of Hyperthermia}
  \bibinfo{volume}{24}, \bibinfo{pages}{577--588}.
\bibitem[{Akrivis and Larsson(2005)}]{Akrivis2005}
\bibinfo{author}{Akrivis, G.}, \bibinfo{author}{Larsson, S.},
  \bibinfo{year}{2005}.
\newblock \bibinfo{title}{Linearly implicit finite element methods for the
  time-dependent joule heating problem}.
\newblock \bibinfo{journal}{BIT Numerical Mathematics} \bibinfo{volume}{45},
  \bibinfo{pages}{429--442}.
\bibitem[{Allegretto et~al.(1999)Allegretto, Liu and Zhou}]{Allegretto1999}
\bibinfo{author}{Allegretto, W.}, \bibinfo{author}{Liu, Y.},
  \bibinfo{author}{Zhou, A.}, \bibinfo{year}{1999}.
\newblock \bibinfo{title}{A box scheme for coupled systems resulting from
  microsensor thermistor problems}.
\newblock \bibinfo{journal}{Dynamics of continuous discrete and impulsive
  systems} \bibinfo{volume}{5}, \bibinfo{pages}{209--223}.
\bibitem[{Allegretto and Xie(1992)}]{Allegretto1992}
\bibinfo{author}{Allegretto, W.}, \bibinfo{author}{Xie, H.},
  \bibinfo{year}{1992}.
\newblock \bibinfo{title}{Existence of solutions for the time-dependent
  thermistor equations}.
\newblock \bibinfo{journal}{IMA Journal of Applied Mathematics}
  \bibinfo{volume}{48}, \bibinfo{pages}{271--281}.
\bibitem[{Antonietti et~al.(2022)Antonietti, Vacca and
  Verani}]{antonietti2022virtual}
\bibinfo{author}{Antonietti, P.F.}, \bibinfo{author}{Vacca, G.},
  \bibinfo{author}{Verani, M.}, \bibinfo{year}{2022}.
\newblock \bibinfo{title}{{Virtual element method for the {N}avier--{S}tokes
  equation coupled with the heat equation}}.
\newblock \bibinfo{journal}{IMA Journal of Numerical Analysis}
  \bibinfo{volume}{43}, \bibinfo{pages}{3396--3429}.
\bibitem[{Atlas et~al.(2020)Atlas, Bendahmane, Karami, Meskine and
  Zagour}]{ABKMZ20}
\bibinfo{author}{Atlas, A.}, \bibinfo{author}{Bendahmane, M.},
  \bibinfo{author}{Karami, F.}, \bibinfo{author}{Meskine, D.},
  \bibinfo{author}{Zagour, M.}, \bibinfo{year}{2020}.
\newblock \bibinfo{title}{Kinetic-fluid derivation and mathematical analysis of
  a nonlocal cross-diffusion-fluid system}.
\newblock \bibinfo{journal}{Applied Mathematical Modelling}
  \bibinfo{volume}{82}, \bibinfo{pages}{379--408}.
\bibitem[{Aubin(1963)}]{Aubin1963}
\bibinfo{author}{Aubin, J.P.}, \bibinfo{year}{1963}.
\newblock \bibinfo{title}{Un thorme de compacit}.
\newblock \bibinfo{journal}{CR Acad. Sci. Paris} \bibinfo{volume}{256},
  \bibinfo{pages}{3}.
\bibitem[{(auth.)(2009)}]{Latif2009}
\bibinfo{author}{(auth.), L.M.J.}, \bibinfo{year}{2009}.
\newblock \bibinfo{title}{Heat Convection: Second Edition}.
\newblock \bibinfo{edition}{2} ed., \bibinfo{publisher}{Springer-Verlag Berlin
  Heidelberg}.
\bibitem[{Bendahmane et~al.(2023)Bendahmane, Nzeti, Tagoudjeu and
  Zagour}]{BNTZ23}
\bibinfo{author}{Bendahmane, M.}, \bibinfo{author}{Nzeti, H.},
  \bibinfo{author}{Tagoudjeu, J.}, \bibinfo{author}{Zagour, M.},
  \bibinfo{year}{2023}.
\newblock \bibinfo{title}{Stochastic reaction--diffusion system modeling
  predator--prey interactions with prey-taxis and noises}.
\newblock \bibinfo{journal}{Chaos} \bibinfo{volume}{33}, \bibinfo{pages}{Paper
  No. 073103, 26}.
\bibitem[{Bendahmane et~al.(2022)Bendahmane, Tagoudjeu and Zagour}]{BTZ22}
\bibinfo{author}{Bendahmane, M.}, \bibinfo{author}{Tagoudjeu, J.},
  \bibinfo{author}{Zagour, M.}, \bibinfo{year}{2022}.
\newblock \bibinfo{title}{Odd-{E}ven based asymptotic preserving scheme for a
  2{D} stochastic kinetic-fluid model}.
\newblock \bibinfo{journal}{Journal of Computational Physics}
  \bibinfo{volume}{471}, \bibinfo{pages}{Paper No. 111649, 25}.
\bibitem[{Bene{\v{s}}(2011)}]{Benes2011}
\bibinfo{author}{Bene{\v{s}}, M.}, \bibinfo{year}{2011}.
\newblock \bibinfo{title}{Strong solutions to non-stationary channel flows of
  heat-conducting viscous incompressible fluids with dissipative heating}.
\newblock \bibinfo{journal}{Acta applicandae mathematicae}
  \bibinfo{volume}{116}, \bibinfo{pages}{237--254}.
\bibitem[{Bene{\v{s}} and Ku{\v{c}}era(2007)}]{Benes2007}
\bibinfo{author}{Bene{\v{s}}, M.}, \bibinfo{author}{Ku{\v{c}}era, P.},
  \bibinfo{year}{2007}.
\newblock \bibinfo{title}{Non-steady {N}avier--{S}tokes equations with
  homogeneous mixed boundary conditions and arbitrarily large initial
  condition}.
\newblock \bibinfo{journal}{Carpathian Journal of Mathematics} ,
  \bibinfo{pages}{32--40}.
\bibitem[{Bene{\v{s}} et~al.(2022)Bene{\v{s}}, Pa{\v{z}}anin and
  Radulovi{\'c}}]{New-Benes2022}
\bibinfo{author}{Bene{\v{s}}, M.}, \bibinfo{author}{Pa{\v{z}}anin, I.},
  \bibinfo{author}{Radulovi{\'c}, M.}, \bibinfo{year}{2022}.
\newblock \bibinfo{title}{On viscous incompressible flows of nonsymmetric
  fluids with mixed boundary conditions}.
\newblock \bibinfo{journal}{Nonlinear Analysis: Real World Applications}
  \bibinfo{volume}{64}, \bibinfo{pages}{103424}.
\bibitem[{Bene{\v{s}} and Tich{\`y}(2015)}]{ref1}
\bibinfo{author}{Bene{\v{s}}, M.}, \bibinfo{author}{Tich{\`y}, J.},
  \bibinfo{year}{2015}.
\newblock \bibinfo{title}{On coupled {N}avier--{S}tokes and energy equations in
  exterior-like domains}.
\newblock \bibinfo{journal}{Computers and Mathematics with Applications}
  \bibinfo{volume}{70}, \bibinfo{pages}{2867--2882}.
\bibitem[{Berjano(2006)}]{Berjano2006}
\bibinfo{author}{Berjano, E.J.}, \bibinfo{year}{2006}.
\newblock \bibinfo{title}{Theoretical modeling for radiofrequency ablation:
  state-of-the-art and challenges for the future}.
\newblock \bibinfo{journal}{Biomedical engineering online} \bibinfo{volume}{5},
  \bibinfo{pages}{1--17}.
\bibitem[{Bul{\'\i}{\v{c}}ek et~al.(2016)Bul{\'\i}{\v{c}}ek, Diening and
  Schwarzacher}]{bulivcek2016existence}
\bibinfo{author}{Bul{\'\i}{\v{c}}ek, M.}, \bibinfo{author}{Diening, L.},
  \bibinfo{author}{Schwarzacher, S.}, \bibinfo{year}{2016}.
\newblock \bibinfo{title}{Existence, uniqueness and optimal regularity results
  for very weak solutions to nonlinear elliptic systems}.
\newblock \bibinfo{journal}{Analysis \& PDE} \bibinfo{volume}{9},
  \bibinfo{pages}{1115--1151}.
\bibitem[{Bul\'{\i}\v{c}ek et~al.(2009)Bul\'{\i}\v{c}ek, Feireisl and
  M\'{a}lek}]{BULICEK2009}
\bibinfo{author}{Bul\'{\i}\v{c}ek, M.}, \bibinfo{author}{Feireisl, E.},
  \bibinfo{author}{M\'{a}lek, J.}, \bibinfo{year}{2009}.
\newblock \bibinfo{title}{A {N}avier--{S}tokes--{F}ourier system for
  incompressible fluids with temperature dependent material coefficients}.
\newblock \bibinfo{journal}{Nonlinear Analysis. Real World Applications}
  \bibinfo{volume}{10}, \bibinfo{pages}{992--1015}.
\bibitem[{Cimatti(1992)}]{ref2}
\bibinfo{author}{Cimatti, G.}, \bibinfo{year}{1992}.
\newblock \bibinfo{title}{Existence of weak solutions for the nonstationary
  problem of the joule heating of a conductor}.
\newblock \bibinfo{journal}{Annali di Matematica pura ed applicata}
  \bibinfo{volume}{162}, \bibinfo{pages}{33--42}.
\bibitem[{Deteix et~al.(2014)Deteix, Jendoubi and Yakoubi}]{deteix2014coupled}
\bibinfo{author}{Deteix, J.}, \bibinfo{author}{Jendoubi, A.},
  \bibinfo{author}{Yakoubi, D.}, \bibinfo{year}{2014}.
\newblock \bibinfo{title}{A coupled prediction scheme for solving the
  {N}avier--{S}tokes and convection-diffusion equations}.
\newblock \bibinfo{journal}{SIAM Journal on Numerical Analysis}
  \bibinfo{volume}{52}, \bibinfo{pages}{2415--2439}.
\bibitem[{Deugoue et~al.(2021)Deugoue, Djoko and Fouape}]{deugoue2021globally}
\bibinfo{author}{Deugoue, G.}, \bibinfo{author}{Djoko, J.},
  \bibinfo{author}{Fouape, A.}, \bibinfo{year}{2021}.
\newblock \bibinfo{title}{Globally modified {N}avier--{S}tokes equations
  coupled with the heat equation: existence result and time discrete
  approximation}.
\newblock \bibinfo{journal}{Journal of Applied Analysis \& Computation}
  \bibinfo{volume}{11}, \bibinfo{pages}{2423--2458}.
\bibitem[{Elliott and Larsson(1995)}]{Elliott1995}
\bibinfo{author}{Elliott, C.M.}, \bibinfo{author}{Larsson, S.},
  \bibinfo{year}{1995}.
\newblock \bibinfo{title}{A finite element model for the time-dependent joule
  heating problem}.
\newblock \bibinfo{journal}{Mathematics of computation} \bibinfo{volume}{64},
  \bibinfo{pages}{1433--1453}.
\bibitem[{Formaggia et~al.(2009)Formaggia, Quarteroni and Veneziani}]{FQV09}
\bibinfo{author}{Formaggia, L.}, \bibinfo{author}{Quarteroni, A.},
  \bibinfo{author}{Veneziani, A.}, \bibinfo{year}{2009}.
\newblock \bibinfo{title}{Cardiovascular mathematics, volume 1 of ms\&a.
  modeling, simulation and applications}.
\bibitem[{Fouchet-Incaux(2014)}]{FI14}
\bibinfo{author}{Fouchet-Incaux, J.}, \bibinfo{year}{2014}.
\newblock \bibinfo{title}{Artificial boundaries and formulations for the
  incompressible {N}avier--{S}tokes equations: applications to air and blood
  flows}.
\newblock \bibinfo{journal}{SeMA Journal} \bibinfo{volume}{64},
  \bibinfo{pages}{1--40}.
\bibitem[{Gao(2016)}]{Gao2015}
\bibinfo{author}{Gao, H.}, \bibinfo{year}{2016}.
\newblock \bibinfo{title}{Unconditional optimal error estimates of
  {BDF}--{G}alerkin {FEM}s for nonlinear thermistor equations}.
\newblock \bibinfo{journal}{Journal of Scientific Computing}
  \bibinfo{volume}{66}, \bibinfo{pages}{504--527}.
\bibitem[{Gatica(1988)}]{Kurzweil.1986}
\bibinfo{author}{Gatica, J.A.}, \bibinfo{year}{1988}.
\newblock \bibinfo{title}{Ordinary differential equations: Introduction to the
  theory of ordinary differential equations in the real domain (jaroslav
  kurzweil)}.
\newblock \bibinfo{journal}{SIAM Review} \bibinfo{volume}{30},
  \bibinfo{pages}{512}.
\bibitem[{Gonz{\'a}lez-Su{\'a}rez and Berjano(2015)}]{materiel2}
\bibinfo{author}{Gonz{\'a}lez-Su{\'a}rez, A.}, \bibinfo{author}{Berjano, E.},
  \bibinfo{year}{2015}.
\newblock \bibinfo{title}{Comparative analysis of different methods of modeling
  the thermal effect of circulating blood flow during {RF} cardiac ablation}.
\newblock \bibinfo{journal}{IEEE Transactions on Biomedical Engineering}
  \bibinfo{volume}{63}, \bibinfo{pages}{250--259}.
\bibitem[{Gonz{\'a}lez-Su{\'a}rez et~al.(2016a)Gonz{\'a}lez-Su{\'a}rez,
  Berjano, Guerra and Gerardo-Giorda}]{gonzalez2016computational}
\bibinfo{author}{Gonz{\'a}lez-Su{\'a}rez, A.}, \bibinfo{author}{Berjano, E.},
  \bibinfo{author}{Guerra, J.M.}, \bibinfo{author}{Gerardo-Giorda, L.},
  \bibinfo{year}{2016}a.
\newblock \bibinfo{title}{Computational modeling of open-irrigated electrodes
  for radiofrequency cardiac ablation including blood motion-saline flow
  interaction}.
\newblock \bibinfo{journal}{PloS one} \bibinfo{volume}{11},
  \bibinfo{pages}{e0150356}.
\bibitem[{Gonz{\'a}lez-Su{\'a}rez et~al.(2016b)Gonz{\'a}lez-Su{\'a}rez,
  Berjano, Guerra and Gerardo-Giorda}]{materiel1}
\bibinfo{author}{Gonz{\'a}lez-Su{\'a}rez, A.}, \bibinfo{author}{Berjano, E.},
  \bibinfo{author}{Guerra, J.M.}, \bibinfo{author}{Gerardo-Giorda, L.},
  \bibinfo{year}{2016}b.
\newblock \bibinfo{title}{Computational modeling of open-irrigated electrodes
  for radiofrequency cardiac ablation including blood motion-saline flow
  interaction}.
\newblock \bibinfo{journal}{PloS one} \bibinfo{volume}{11},
  \bibinfo{pages}{e0150356}.
\bibitem[{Gonz{\'a}lez-Su{\'a}rez et~al.(2018)Gonz{\'a}lez-Su{\'a}rez,
  P{\'e}rez and Berjano}]{materiel3}
\bibinfo{author}{Gonz{\'a}lez-Su{\'a}rez, A.}, \bibinfo{author}{P{\'e}rez,
  J.J.}, \bibinfo{author}{Berjano, E.}, \bibinfo{year}{2018}.
\newblock \bibinfo{title}{Should fluid dynamics be included in computer models
  of {RF} cardiac ablation by irrigated-tip electrodes?}
\newblock \bibinfo{journal}{Biomedical engineering online}
  \bibinfo{volume}{17}, \bibinfo{pages}{1--14}.
\bibitem[{Guermond et~al.(2011)Guermond, Pasquetti and
  Popov}]{guermond2011entropy}
\bibinfo{author}{Guermond, J.L.}, \bibinfo{author}{Pasquetti, R.},
  \bibinfo{author}{Popov, B.}, \bibinfo{year}{2011}.
\newblock \bibinfo{title}{Entropy viscosity method for nonlinear conservation
  laws}.
\newblock \bibinfo{journal}{Journal of Computational Physics}
  \bibinfo{volume}{230}, \bibinfo{pages}{4248--4267}.
\bibitem[{Haemmerich(2010)}]{Dieter2010}
\bibinfo{author}{Haemmerich, D.}, \bibinfo{year}{2010}.
\newblock \bibinfo{title}{Mathematical modeling of impedance controlled
  radiofrequency tumor ablation and ex-vivo validation}, in:
  \bibinfo{booktitle}{2010 Annual International Conference of the IEEE
  Engineering in Medicine and Biology}, \bibinfo{organization}{IEEE}. pp.
  \bibinfo{pages}{1605--1608}.
\bibitem[{Hecht(2012)}]{hecht2012new}
\bibinfo{author}{Hecht, F.}, \bibinfo{year}{2012}.
\newblock \bibinfo{title}{New development in freefem++}.
\newblock \bibinfo{journal}{Journal of numerical mathematics}
  \bibinfo{volume}{20}, \bibinfo{pages}{251--266}.
\bibitem[{Johnson and Saidel(2002)}]{johnson2002}
\bibinfo{author}{Johnson, P.C.}, \bibinfo{author}{Saidel, G.M.},
  \bibinfo{year}{2002}.
\newblock \bibinfo{title}{Thermal model for fast simulation during magnetic
  resonance imaging guidance of radio-frequency tumor ablation}.
\newblock \bibinfo{journal}{Annals of Biomedical Engineering}
  \bibinfo{volume}{30}, \bibinfo{pages}{1152--1161}.
\bibitem[{Kufner et~al.(1977)Kufner, John and Fucik}]{Kufner}
\bibinfo{author}{Kufner, A.}, \bibinfo{author}{John, O.},
  \bibinfo{author}{Fucik, S.}, \bibinfo{year}{1977}.
\newblock \bibinfo{title}{Function spaces}. volume~\bibinfo{volume}{3}.
\newblock \bibinfo{publisher}{Springer Science \& Business Media}.
\bibitem[{Li et~al.(2014)Li, Gao and Sun}]{Li2014}
\bibinfo{author}{Li, B.}, \bibinfo{author}{Gao, H.}, \bibinfo{author}{Sun, W.},
  \bibinfo{year}{2014}.
\newblock \bibinfo{title}{Unconditionally optimal error estimates of a
  {C}rank--{N}icolson {G}alerkin method for the nonlinear thermistor
  equations}.
\newblock \bibinfo{journal}{SIAM Journal on Numerical Analysis}
  \bibinfo{volume}{52}, \bibinfo{pages}{933--954}.
\bibitem[{Li and Sun(2013)}]{Li2012}
\bibinfo{author}{Li, B.}, \bibinfo{author}{Sun, W.}, \bibinfo{year}{2013}.
\newblock \bibinfo{title}{Error analysis of linearized semi-implicit {G}alerkin
  finite element methods for nonlinear parabolic equations}.
\newblock \bibinfo{journal}{International Journal of Numerical Analysis and
  Modeling} \bibinfo{volume}{10}, \bibinfo{pages}{622--633}.
\bibitem[{Li and Yang(2015)}]{Li-Yang2015}
\bibinfo{author}{Li, B.}, \bibinfo{author}{Yang, C.}, \bibinfo{year}{2015}.
\newblock \bibinfo{title}{Uniform {BMO} estimate of parabolic equations and
  global well-posedness of the thermistor problem}.
\newblock \bibinfo{journal}{Forum Math. Sigma} \bibinfo{volume}{3},
  \bibinfo{pages}{Paper No. e26, 31}.
\bibitem[{L\'{o}pez~Molina et~al.(2017)L\'{o}pez~Molina, Rivera and
  Berjano}]{Ber17}
\bibinfo{author}{L\'{o}pez~Molina, J.A.}, \bibinfo{author}{Rivera, M.J.},
  \bibinfo{author}{Berjano, E.}, \bibinfo{year}{2017}.
\newblock \bibinfo{title}{Analytical transient-time solution for temperature in
  non perfused tissue during radiofrequency ablation}.
\newblock \bibinfo{journal}{Applied Mathematical Modelling}
  \bibinfo{volume}{42}, \bibinfo{pages}{618--635}.
\bibitem[{Martynenko and Khramtsov(2005)}]{Martynenko2005}
\bibinfo{author}{Martynenko, O.G.}, \bibinfo{author}{Khramtsov, P.P.},
  \bibinfo{year}{2005}.
\newblock \bibinfo{title}{Basic statements and equations of free convection}.
\newblock \bibinfo{journal}{Free-Convective Heat Transfer: With Many
  Photographs of Flows and Heat Exchange} , \bibinfo{pages}{1--79}.
\bibitem[{Mbehou(2018)}]{Mbehou2018}
\bibinfo{author}{Mbehou, M.}, \bibinfo{year}{2018}.
\newblock \bibinfo{title}{The theta $g$alerkin finite element method for
  coupled systems resulting from microsensor thermistor problems}.
\newblock \bibinfo{journal}{Mathematical Methods in the Applied Sciences}
  \bibinfo{volume}{41}, \bibinfo{pages}{1480--1491}.
\bibitem[{Meinlschmidt et~al.(2017a)Meinlschmidt, Meyer and
  Rehberg}]{MEINLSCHMIDT-part1}
\bibinfo{author}{Meinlschmidt, H.}, \bibinfo{author}{Meyer, C.},
  \bibinfo{author}{Rehberg, J.}, \bibinfo{year}{2017}a.
\newblock \bibinfo{title}{Optimal control of the thermistor problem in three
  spatial dimensions, {P}art 1: Existence of optimal solutions}.
\newblock \bibinfo{journal}{SIAM Journal on Control and Optimization}
  \bibinfo{volume}{55}, \bibinfo{pages}{2876--2904}.
\bibitem[{Meinlschmidt et~al.(2017b)Meinlschmidt, Meyer and
  Rehberg}]{MEINLSCHMIDT-part2}
\bibinfo{author}{Meinlschmidt, H.}, \bibinfo{author}{Meyer, C.},
  \bibinfo{author}{Rehberg, J.}, \bibinfo{year}{2017}b.
\newblock \bibinfo{title}{Optimal control of the thermistor problem in three
  spatial dimensions, {P}art 2: Optimality conditions}.
\newblock \bibinfo{journal}{SIAM Journal on Control and Optimization}
  \bibinfo{volume}{55}, \bibinfo{pages}{2368--2392}.
\bibitem[{Nolte et~al.(2021)Nolte, Vaidya, Baragona, Elevelt, Lavezzo, Maessen,
  Schulz and Veroy}]{Nikhil2021}
\bibinfo{author}{Nolte, T.}, \bibinfo{author}{Vaidya, N.},
  \bibinfo{author}{Baragona, M.}, \bibinfo{author}{Elevelt, A.},
  \bibinfo{author}{Lavezzo, V.}, \bibinfo{author}{Maessen, R.},
  \bibinfo{author}{Schulz, V.}, \bibinfo{author}{Veroy, K.},
  \bibinfo{year}{2021}.
\newblock \bibinfo{title}{Study of flow effects on temperature-controlled
  radiofrequency ablation using phantom experiments and forward simulations}.
\newblock \bibinfo{journal}{Medical Physics} \bibinfo{volume}{48},
  \bibinfo{pages}{4754--4768}.
\bibitem[{Ooi and Ooi(2017)}]{OOI17}
\bibinfo{author}{Ooi, E.H.}, \bibinfo{author}{Ooi, E.T.}, \bibinfo{year}{2017}.
\newblock \bibinfo{title}{Mass transport in biological tissues: Comparisons
  between single- and dual-porosity models in the context of saline-infused
  {R}adiofrequency {A}blation}.
\newblock \bibinfo{journal}{Applied Mathematical Modelling}
  \bibinfo{volume}{41}, \bibinfo{pages}{271--284}.
\bibitem[{Quarteroni et~al.(2017)Quarteroni, Manzoni and Vergara}]{QMV17}
\bibinfo{author}{Quarteroni, A.}, \bibinfo{author}{Manzoni, A.},
  \bibinfo{author}{Vergara, C.}, \bibinfo{year}{2017}.
\newblock \bibinfo{title}{The cardiovascular system: mathematical modelling,
  numerical algorithms and clinical applications}.
\newblock \bibinfo{journal}{Acta Numerica} \bibinfo{volume}{26},
  \bibinfo{pages}{365--590}.
\bibitem[{{Salmon, St\'ephanie} et~al.(2012){Salmon, St\'ephanie}, {Sy,
  Soyibou} and {Szopos, Marcela}}]{Salmon12}
\bibinfo{author}{{Salmon, St\'ephanie}}, \bibinfo{author}{{Sy, Soyibou}},
  \bibinfo{author}{{Szopos, Marcela}}, \bibinfo{year}{2012}.
\newblock \bibinfo{title}{Cerebral blood flow simulations in realistic
  geometries}.
\newblock \bibinfo{journal}{ESAIM: Proc.} \bibinfo{volume}{35},
  \bibinfo{pages}{281--286}.
\bibitem[{Villard et~al.(2005)Villard, Soler and Gangi}]{villard2005}
\bibinfo{author}{Villard, C.}, \bibinfo{author}{Soler, L.},
  \bibinfo{author}{Gangi, A.}, \bibinfo{year}{2005}.
\newblock \bibinfo{title}{Radiofrequency ablation of hepatic tumors:
  simulation, planning, and contribution of virtual reality and haptics}.
\newblock \bibinfo{journal}{Computer Methods in Biomechanics and Biomedical
  Engineering} \bibinfo{volume}{8}, \bibinfo{pages}{215--227}.
\bibitem[{Wongchadakul et~al.(2023)Wongchadakul, Datta and
  Rattanadecho}]{wongchadakul2023natural}
\bibinfo{author}{Wongchadakul, P.}, \bibinfo{author}{Datta, A.K.},
  \bibinfo{author}{Rattanadecho, P.}, \bibinfo{year}{2023}.
\newblock \bibinfo{title}{Natural convection effects on heat transfer in a
  porous tissue in 3-d radiofrequency cardiac ablation}.
\newblock \bibinfo{journal}{International Journal of Heat and Mass Transfer}
  \bibinfo{volume}{204}, \bibinfo{pages}{123832}.
\bibitem[{Xu(1994)}]{xu1994thermistor}
\bibinfo{author}{Xu, X.}, \bibinfo{year}{1994}.
\newblock \bibinfo{title}{The thermistor problem with conductivity vanishing
  for large temperature}.
\newblock \bibinfo{journal}{Proceedings of the Royal Society of Edinburgh
  Section A: Mathematics} \bibinfo{volume}{124}, \bibinfo{pages}{1--21}.
\bibitem[{Yuan and Liu(1994)}]{Yuan1994}
\bibinfo{author}{Yuan, G.}, \bibinfo{author}{Liu, Z.}, \bibinfo{year}{1994}.
\newblock \bibinfo{title}{Existence and uniqueness of the {C}$^{\alpha}$
  solution for the thermistor problem with mixed boundary value}.
\newblock \bibinfo{journal}{SIAM Journal on Mathematical Analysis}
  \bibinfo{volume}{25}, \bibinfo{pages}{1157--1166}.
\bibitem[{Zagour(2023)}]{Z23}
\bibinfo{author}{Zagour, M.}, \bibinfo{year}{2023}.
\newblock \bibinfo{title}{Toward multiscale derivation of behavioral dynamics:
  Comment to “what is life? active particles tools towards behavioral
  dynamics in social-biology and economics”, by b. bellomo, m. esfahanian, v.
  secchini, and p. terna}.
\newblock \bibinfo{journal}{Physics of Life Reviews} \bibinfo{volume}{46},
  \bibinfo{pages}{273--274}.

\end{thebibliography}

\end{document}